\documentclass[reqno]{amsart}
\usepackage[top=0.9in,bottom=0.9in,left=1.0in,right=1.0in]{geometry}
\usepackage{amscd}
\usepackage{amssymb}
\usepackage{multicol}
\usepackage{makecell}
\usepackage[T1]{fontenc}
\usepackage[latin1]{inputenc}
\usepackage{tikz}
\usetikzlibrary{shapes.geometric, arrows}
\usepackage[all]{xy}
\usepackage{times}
\usepackage{comment}

\usepackage{flafter}
\usepackage{placeins}
\usepackage{enumitem}
\usepackage[OT2,T1]{fontenc}
\DeclareSymbolFont{cyrletters}{OT2}{wncyr}{m}{n}
\DeclareMathSymbol{\Sha}{\mathalpha}{cyrletters}{"58}

\usepackage{centernot}
\usetikzlibrary{positioning}
\usepackage{mathrsfs}
\usepackage[color]{showkeys}
\usepackage{url}

\definecolor{refkey}{rgb}{1,1,1}
\definecolor{labelkey}{rgb}{1,1,1}
\definecolor{cite}{rgb}{0.9451,0.2706,0.4941}
\definecolor{ruri}{rgb}{0.0078,0.4022,0.8010}

\usepackage[%
bookmarks=true,bookmarksnumbered=true,%
colorlinks=true,linkcolor=purple,citecolor=orange%
 ]{hyperref}

\def\G{{\rm \mathbb{G}}}

\def\Aut{{\rm Aut}}

\def\Aut{{\rm Aut}}

\theoremstyle{plain}

\newtheorem{theorem}{Theorem}[section]

\newtheorem{proposition/example}[theorem]{Proposition/Example}
\newtheorem{proposition}[theorem]{Proposition}
\newtheorem{corollary}[theorem]{Corollary}
\newtheorem{lemma}[theorem]{Lemma}

\theoremstyle{definition}
\newtheorem{definition}[theorem]{Definition}
\newtheorem{remark}[theorem]{Remark}

\newtheorem{example}[theorem]{Example}

\newtheorem{conjecture/question}[theorem]{Conjecture/Question}

\newtheorem{remark/definition}[theorem]{Remark/Definition}
\newtheorem{definition/notation}[theorem]{Definition/Notation}

\numberwithin{equation}{section}

 \theoremstyle{remark}

\numberwithin{equation}{section}

\newcommand{\Ker}{\operatorname{Ker}}

\usepackage{extarrows}\usepackage{bm}

\begin{document}
\title{\textbf{Parahoric Reduction Theory of Formal Connections (or Higgs Fields)}}

\author{Zhi Hu}

\address{  \textsc{School of Mathematics and Statistics, Nanjing University of Science and Technology, Nanjing 210094, China}}

\email{halfask@mail.ustc.edu.cn}

\author{Pengfei Huang}

\address{ \textsc{School of Mathematics, Nanjing University, Nanjing  210093, China}}

\email{pfhwangmath@gmail.com}

\author{Ruiran Sun}
 \address{School of Mathematical Sciences, Xiamen University, Xiamen 361005, China}
 
\email{ruiransun@xmu.edu.cn}

\author{Runhong Zong}

\address{ \textsc{School of Mathematics, Nanjing University, Nanjing  210093, China}}
\email{rzong@nju.edu.cn}



\date{}

\begin{abstract}
In this paper, we establish the parahoric reduction theory of formal connections (or Higgs fields) on a formal principal bundle with parahoric structures. This generalizes the reduction theory due to Hukuhara-Levelt-Turrittin and Babbitt-Varadarajan in the case without parahoric structures \cite{bv}, and Boalch's result in the case of regular singularity \cite{b}. As applications, we first prove that a formal connection is intrinsically regular if and only if it is a parahoric formal connection, in combination with the works of Bremer-Sage \cite{ks} and Chen-Kamgarpour \cite{bk}, providing various criteria for the regularity of formal connections. We then introduce a notion of relative regularity for formal connections and give  criterions. Finally, we demonstrate a parahoric version of Frenkel-Zhu's Borel reduction theorem of formal connections \cite{fz}.
\end{abstract}

\maketitle
\tableofcontents

\section{Introduction}
\subsection{Motivations: Why Need Parahoric Structures? }
For a reductive algebraic group $G$ over a local field $\mathbb{K}$, the local Langlands
conjecture predicts that  an irreducible complex representation of the locally compact group $G(\mathbb{K})$ should correspond  to a Langlands parameter (i.e. a homomorphism $\varphi$ from the Weil-Deligne group of $\mathbb{K}$ to the complex Langlands dual group $^LG$) together with an irreducible representation $\rho$ of the component group of the centralizer of $\varphi$.
In the geometric Langlands world,  the role of Langlands parameters
is played by principal  $^LG$-bundles endowed with  connections over the formal punctured disc
 \cite{fg1,f}. These Langlands parameters are classified as unramified, tamely ramified, or wildly ramified, depending on the restriction of $\varphi$ on the inertia group of the Galois group $\mathrm{Gal}(\overline{\mathbb{K}}/\mathbb{K})$ being trivial, on the wild ramification subgroup of $\mathrm{Gal}(\overline{\mathbb{K}}/\mathbb{K})$ being trivial, or on the wild ramification subgroup being non-trivial, respectively \cite{f, f1}. In a geometric context, this classification is translated into the classification of formal connections according to their singularities. Specifically, these three classes correspond to formal connections without singularities, with regular singularities, and with irregular singularities, respectively \cite{ks}.

A more precise classification of Langlands parameters is given by the depth of $\varphi$, which is defined as the smallest number $d\geq 0$ such that $\varphi$ is trivial on the $r$-th ramification subgroup of $\mathrm{Gal}(\overline{\mathbb{K}}/\mathbb{K})$ for all $r > d$.  On the other hand, Moy and Prasad defined depth in terms of $\mathbb{R}$-filtrations of the parahoric subgroups of $G(\mathbb{K})$ \cite{m1, m2}. It is expected that these two notions of depth coincide \cite{ab}. In the geometric setting, depth is translated into a certain invariant of formal connections called slope. Specifically, zero, and positive slopes correspond to regular singularity (including  nonsingularity) and irregular singularity, respectively.   Similarly, there are also two approaches to defining slope: one is derived from the reduction theory of formal connections as suggested by  Frenkel and Gross \cite{fs}, and the other is from the geometric version of minimal $K$-type theory, as developed by Bremer and Sage \cite{bs}. Bremer and Sage also showed the equivalence of these two definitions of slope (cf. \cite[Proposition 2.20]{bs}). We revisited this equivalence for regular case via the third category, i.e. the category of weighted parahoric formal connections, as normal  forms of regular formal connections\footnote{For brevity, throughout this paper we refer to a regular singular formal connection simply as a regular formal  connection.} defined  intrinsically or extrinsically.  In particular, this implies that the gauge group of regular formal connections is actually reduced to a parahoric subgroup of $G(\mathbb{K})$.
The other numerical invariants on Langlands parameters can be also translated into those on formal connections, for example, Swan conductor corresponds to irregularity, which are both integers \cite{ks}.

Another motivation for incorporating parahoric structures originates from Boalch's pioneering work on the $G$-version of (regular) Riemann-Hilbert correspondence \cite{b}, which generalizes Deligne's classical work \cite{de}. In the context  of vector bundles (i.e. $G=\mathrm{GL}(n,\mathbb{C})$), recall that   Deligne's  work predicts the equivalence of the following categories:
$$\begin{tikzpicture}
\draw [-, thick, color=black] (-4.7,0.5) -- (-0.4,0.5) -- (-0.4,1.5) -- (-4.7,1.5) -- (-4.7,0.5);
\draw [-, thick, color=black] (-4.7,-0.5) -- (-0.4,-0.5) -- (-0.4,-1.5) -- (-4.7,-1.5) -- (-4.7,-0.5);
\draw [-, thick, color=black] (5,0.5) -- (1,0.5) -- (1,1.5) -- (5,1.5) -- (5,0.5);
\draw [-, thick, color=black] (5,-0.5) -- (1,-0.5) -- (1,-1.5) -- (5,-1.5) -- (5,-0.5);
\draw [<->, thick] (-0.4,1) -- (1,1);
\draw [<->, thick] (-0.4,-1) -- (1,-1);
\draw [<->, thick] (-2.5,0.45) -- (-2.5,-0.45);
\draw [<->, thick] (3,0.45) -- (3,-0.45);
\node at (-2.5,1.2) {\text{algebraic regular connections}};
\node at (-2.5,0.8) {\text{on $X \backslash D$}};
\node at (3,1.2) {\text{logarithmic connections}};
\node at (3,0.8) {\text{on $X$}};
\node at (-2.5,-0.8) {\text{holomorphic connections}};
\node at (-2.5,-1.2) {\text{on $(X \backslash D)^{\textrm{an}}$}};
\node at (3,-0.8) {\text{monodromy representations}};
\node at (3,-1.2) {\text{of $\pi_1(X \backslash D)$}};
\node [right] at (-2.5,0) {};
\node [above] at (0,-1) {};
\end{tikzpicture}$$
where $X$ is a complex smooth projective curve,  $D$ is a reduced divisor on $X$, and $(X \backslash D)^{\textrm{an}}$ is the underlying analytic manifold of $X\backslash D$. In particular, the equivalence on the left hand side is derived from Deligne extension (with chosen logarithms of the local monodromy around $D$) and GAGA \cite{ach}.

 However, when one considers a general complex reductive group $G$, two problems arise immediately. One is that the category of logarithmic $G$-connections is too small to cover all regular connections by $G$-gauge transformations; the other is that the exponential map is not necessarily surjective, which implies that certain monodromies may not lie in its image. To address these problems, it is essential to incorporate parahoric structures, which provide additional degrees of freedom beyond logarithmic connections. This allows us to  establish a local $G$-version of the Riemann-Hilbert correspondence
$$
\begin{tikzpicture}
\draw [-, thick] (-5.6,0.5) -- (-1.4,0.5) -- (-1.4,1.5) -- (-5.6,1.5) -- (-5.6,0.5);
\draw [-, thick] (-5.6,-0.5) -- (-1.4,-0.5) -- (-1.4,-1.5) -- (-5.6,-1.5) -- (-5.6,-0.5);
\draw [-, thick] (-0.4,-1.6) -- (1,-1.6) -- (1,1.6) -- (-0.4,1.6) -- (-0.4,-1.6);
\draw [-, thick] (1.8,-0.5) -- (3.8,-0.5) -- (3.8,0.5) -- (1.8,0.5) -- (1.8,-0.5);
\draw [<->, thick] (-1.4,-1) -- (-0.4,-1);
\draw [<->, thick] (-1.4,1) -- (-0.4,1);
\draw [<->, thick] (1,0) -- (1.8,0);
\node at (-3.5,1.25) {\text{ intrinsically regular}};
\node at (-3.5,0.75) {\text{  formal $G$-connections on $\Delta^\times$}};
\node at (-3.5,-0.75) {\text{ extrinsically regular}};
\node at (-3.5,-1.25) {\text{ formal $G$-connections on $\Delta^\times$}};
\node [above] at (-0.9,-1) {$(b)$};
\node [above] at (-0.9,1) {$(a)$};
\node at (0.55,0) {\text{\rotatebox{270}{ parahoric formal}}};
\node at (0.15,0) {\text{\rotatebox{270}{ $G$-connections on  $\Delta^\times$}}};
\node at (2.75,0.2) {\text{monodromy}};
\node at (2.75,-0.2) {\text{data}};
\node [above] at (1.4,0) {$(c)$};
\end{tikzpicture}
$$
where $\Delta^\times=\mathrm{Spec}(\mathbb{K})$ for $\mathbb{K}=\mathbb{\mathbb{C}}(\!(z)\!)$ is the formal punctured disc. The equivalence $(a)$ is obtained from extending Babbitt-Varadarajan's reduction theory to parahoric formal connections (cf. Theorem \ref{d3} and Corollary \ref{d4}), the equivalence $(b)$ is given by Bremer-Sage's geometric minimal $K$-type theory (cf. Theorem \ref{dd}), and the equivalence $(c)$ is due to Boalch (cf. \cite[Theorem D or Theorem 6]{b}).
It is particularly important to point out that varying the parahoric weights upon the Bruhat-Tits building is indispensable, as in general, parahoric  formal connections under a fixed parahoric weight only form a subset of the set of regular formal connections (cf. \cite[Section 5]{b}).

From a practical point of view, introducing parahoric structures also offers several advantages. Firstly, it simplifies the handling of certain ramified irregular formal connections, as the nilpotent leading terms can be absorbed into the parahoric formal connection (i.e., the ``very good'' formal connection suggested by Boalch (cf. \cite[Formula 4]{bq}), whose wild nonabelian Hodge correspondence was considered in \cite{HS23}), instead of working on ramified covers. Secondly, the gauge group can be reduced to a smaller one, such as the Iwahori subgroup of $G(\mathbb{A})$ for $\mathbb{A}=\mathbb{C}[\![z]\!]$.

\subsection{Reduction Theory of Formal Connections}
The reduction theory of formal connections is established through the fundamental work of Babbitt and Varadarajan \cite{bv}. More precisely, they provide an effective algorithm for reducing formal connections to Levelt-Turrittin's canonical forms. This process is succinctly illustrated in the following flowchart. For further details, we refer the reader to Herrero's comprehensive paper \cite{h}.

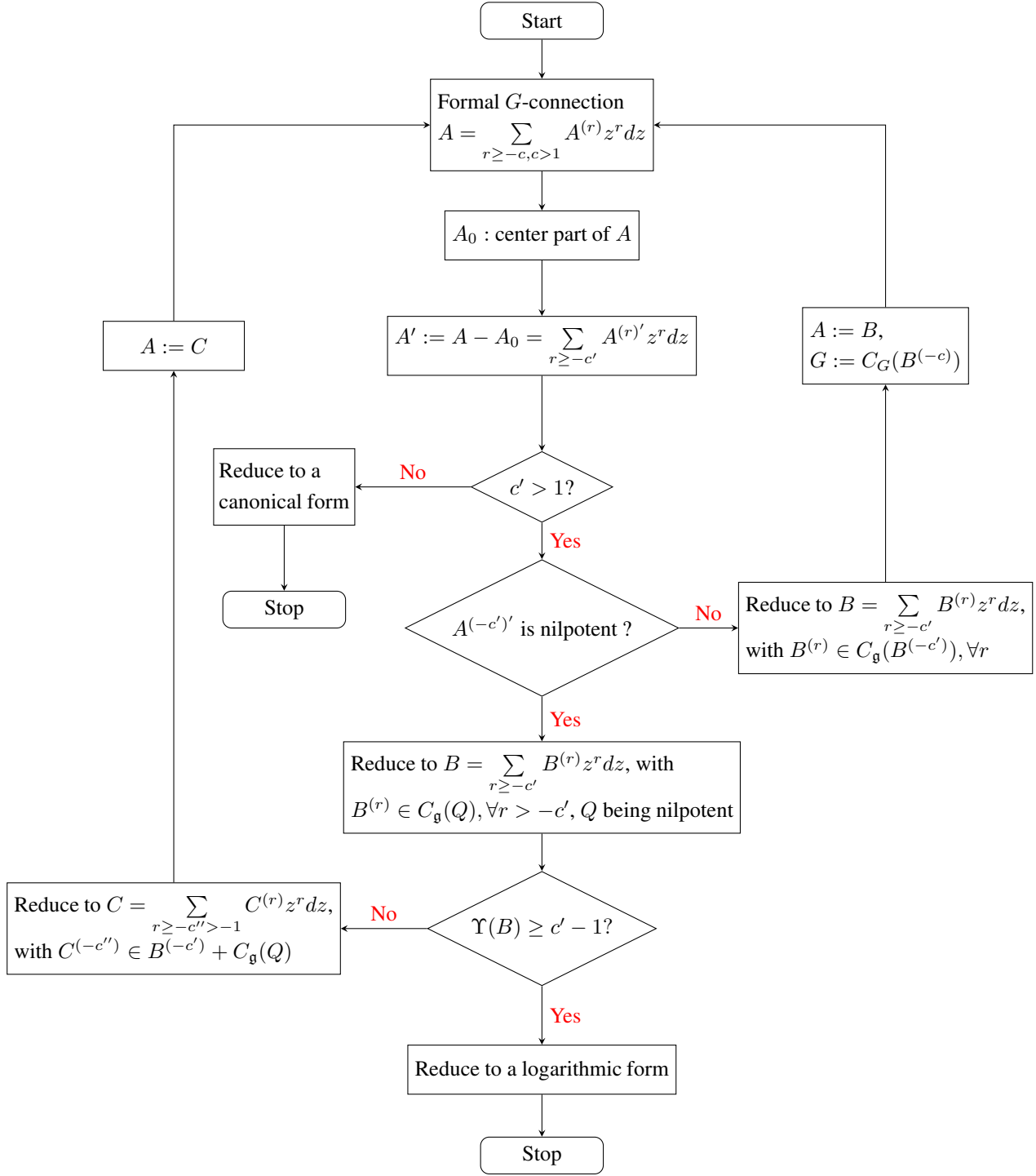
\tikzstyle{startstop} = [rectangle, rounded corners, minimum width = 2cm, minimum height=0.6cm,text centered, draw = black]
\tikzstyle{io} = [trapezium, trapezium left angle=70, trapezium right angle=110, minimum width=2cm, minimum height=0.8cm, text centered, draw=black]
\tikzstyle{process} = [rectangle, minimum width=2.3cm, minimum height=0.8cm, text centered, draw=black]
\tikzstyle{decision} = [diamond, aspect = 2, text centered, draw=black]
\tikzstyle{arrow} = [->,>=stealth]
\vspace*{-15pt}
\[
\begin{tikzpicture}[node distance=0.8cm]
\node[startstop](start){Start};
\node[process, below of = start, yshift = -0.9cm](in1){\makecell[l]{Formal $G$-connection\\ $A=\sum\limits_{r\geq-c, c>1}A^{(r)}z^rdz$}};
\node[process, below of = in1, yshift = -1cm](dec0){ $A_0:$ center part of $A$};
\node[process, below of = dec0, yshift = -1cm](dec1){ $A':=A-A_0=\sum\limits_{r\geq-c'}A^{(r)'}z^rdz$};
\node[process, below of = dec0, yshift = -1cm,xshift=5.6cm](dec6){ \makecell[l]{$A:=B$,\\
$G:=C_G(S)$}};
\node[process, below of = dec0, yshift = -1cm,xshift=-6cm](dec7){$A:=C$};
\node[decision, below of = dec1, yshift = -1.5cm](pro0){$c'>1$?};
\node[process, below of = dec1, yshift = -1.5cm,xshift=-4.2cm](dec01){\makecell[l]{\hspace{-5pt}$\text{Reduce to a}\hspace{-5pt}$\\
\hspace{-5pt}$\text{canonical form}\hspace{-5pt}$}};
\node[startstop, below of = dec01, yshift = -1.2cm](dec02){Stop};
\node[decision, below of = pro0, yshift = -1.5cm](pro1){$A^{(-c')'}$ is nilpotent ?};
\node[process, below of = pro1, yshift = -1.8cm](dec3){\makecell[l]{Reduce to $B=\sum\limits_{r\geq-c'}B^{(r)}z^rdz$, \\ where
$B^{(r)}\in C_{\mathfrak{g}}(Q), \forall r>-c'$, $Q$ is nilpotent}};
\node[process, below of = pro0, yshift = -1.5cm,xshift=5.6cm](dec4){\makecell[l]{Reduce to $B=\sum\limits_{r\geq-c'}B^{(r)}z^rdz$,\\ where  $B^{(r)}\in C_{\mathfrak{g}}(S), \  \forall r $, \\ for semisimple part $S$ of $B^{(-c')}$ }};
\node[decision, below of = dec3, yshift = -1.5cm](pro2){$\Upsilon(B)\geq c'-1 ?$};
\node[process, below of = pro2, yshift = -1.5cm](out1){\makecell[l]{Reduce to a logarithmic form}};
\node[process, below of = dec3, yshift = -1.5cm,xshift=-6cm](dec5){\makecell[l]{Reduce to $C=\sum\limits_{r\geq-c''>-1}C^{(r)}z^rdz$, \\where $C^{(-c'')}\in B^{(-c')}+C_{\mathfrak{g}}(Q)$}};
\node[startstop, below of = out1, yshift = -0.6cm](stop){Stop};
\draw [arrow] (start) -- (in1);
\draw [arrow] (in1) -- (dec0);
\draw [arrow] (dec0) -- (dec1);
\draw [arrow] (dec01) -- (dec02);
\draw [arrow] (dec1) -- (pro0);
\draw [arrow] (pro0) -- node [above] {\color{red}{No}} (dec01);
\draw [arrow] (pro0) -- node [right] {\color{red}{Yes}} (pro1);
\draw [arrow] (pro1) -- node [right] {\color{red}{Yes}} (dec3);
\draw [arrow] (pro1) -- node [above] {\color{red}{No}} (dec4);
\draw [arrow] (pro2) -- node [above] {\color{red}{No}} (dec5);
\draw [arrow] (dec3) -- (pro2);
\draw [arrow] (dec4) -- (dec6);
\draw [arrow] (dec6) |- (in1);
\draw [arrow] (dec5) -- (dec7);
\draw [arrow] (dec7) |- (in1);
\draw [arrow] (pro2) -- node [right] {\color{red}{Yes}} (out1);
\draw [arrow] (out1) -- (stop);
\end{tikzpicture}
\]
Roughly speaking, one utilizes the ``naive leading term'' of formal connections, which reflects the apparent singularity, to produce suitable gauge transformations defined over   $\overline{\mathbb{K}}$. These transformations can eliminate the apparent singularities caused by nilpotent coefficients, leading to a reduction of the formal connections. Babbitt-Varadarajan's algorithm terminates after a finite number of steps by decreasing the dimensions of centralizers in the derived subalgebra and increasing the dimensions of certain nilpotent orbits. For higher-dimensional bases, it is unclear whether a generalization of Babbitt-Varadarajan's reduction algorithm exists. Remarkably, Mochizuki  showed that meromorphic flat connections on algebraic vector bundles over a smooth proper complex algebraic variety with a normal crossing divisor admit good formal structures via the  spectral decomposition of corresponding harmonic bundles obtained by wild version of Corlette theorem \cite{moc,moc1}; later, Kedlaya proved the same result in a more general situation by using  totally different methods \cite{ke,ke1}.  Their results  can be  regarded as a higher-dimensional generalization of the Hukuhara-Levelt-Turrittin's Jordan-type decomposition of formal connections \cite{huk, l, k}. We also expect a $G$-version of Mochizuki-Kedlaya's theorem.

\vspace*{5pt}

In this paper, we generalize the work of Babbitt-Varadarajan and Boalch to formal connections (or Higgs fields) on formal principal bundles with parahoric structures. The crucial ingredient of  Babbitt-Varadarajan's algorithm is reducing each coefficient into suitable centralizers, which guarantees the effectiveness of the algorithm. Under the parahoric context, we show this reduction also works through parahoric gauge transformations (cf. Theorem \ref{x} and Theorem \ref{xx}). The analysis is divided into two cases based on the nilpotency of the Levi  residue of the leading term. Notably, we consistently express a formal connection in terms of a reduced form under the  chosen parahoric weight  and trivialization of bundle, and the reduction process is realized  by  inductions on total parahoric degrees. We also provide a parahoric version of Hukuhara-Levelt-Turrittin's reduction theorem such that the normal form
decomposes into   an irregular part with coefficients  valued in a Cartan subalgebra and a regular  part with  coefficients that are valued in a parahoric subalgebra and commute with all irregular coefficients (cf. Theorem \ref{rd}).

We  present several applications. Inspired by  the parahoric reduction theory for the non-nilpotent case, we introduce the notion of relative  regularity of formal connections and provide  criteria for it. The fundamental feature of relatively regular formal connection is that the Cartan projection is a logarithmic form. This notion can be used to detect  the regularity of abelianization of a formal connection (cf. Definition \ref{rrd} and Theorem \ref{rl}).    The second application is to demonstrate the existence of a regular nilpotent reduction of formal connections, achieved through  Kazhdan-Lusztig's theory, and then applying  the parahoric reduction theory for the nilpotent case  offers a parahoric version of Frenkel-Zhu's Borel reduction theorem \cite{fz}. It is worth noting that Frenkel-Zhu's result can also be derived from the existence of a possibly degenerate oper structure on meromorphic $G$-connections over a smooth curve, by applying a global-to-local argument (cf. \cite[Corollary 7.3]{ar})\footnote{The authors have generalized  Arinkin's result to the parahoric context \cite{hh}, which will appear soon.}, which
generalizes  Drinfeld-Simpson's Borel reduction theorem in the absence of connections \cite{ds}.

\subsection{Main Results}

Now we summarize our main results mentioned above as follows.

\begin{theorem}[= Theorem \ref{ddd}]
Let $\mathsf{A}$ be a formal connection on a formal principal $G$-bundle $\mathbf{P}$. Then the following are equivalent.
\begin{enumerate}
\item There exists a trivialization $e$ of $\mathbf{P}$ and $\bar g\in G(\overline{\mathbb{K}})$ such that the $\widetilde{\mathrm{Ad}}_{\bar g}$-gauge transformation of $\mathsf{A}(e)$ is of logarithmic form, where $\overline{\mathbb{K}}$ is  the algebraic closure of $\mathbb{K}$.
 \item For all representations $(V,\rho)$ consisting of a finite-dimensional vector space $V$ and  a homomorphism   $\rho:G\rightarrow \mathrm{GL}(V)$, the induced connection $\mathsf{A}_\rho$ on the associated  vector bundle $\mathbf{P}_\rho=\mathbf{P}\times_\rho V$ is  regular, namely there is a trivialization $e$ of $\mathbf{P}_\rho$ such that $\mathsf{A}_\rho(e)$ is of logarithmic form.
 \item $\mathbf{P}$ is endowed with a $\Theta$-parahoric structure such that $\mathsf{A}$ is a $\Theta$-parahoric formal connection.
 \end{enumerate}
 When  $G$ is not a torus, the above conditions are also equivalent to the following condition.
  \begin{enumerate}
    \item[(4)] Given  a maximal torus $T\subset G$ with Lie algebra $\mathfrak{t}$ and a Borel subgroup $B\subset G$ containing $T$,
 there is a weight $\Theta\in \mathfrak{t}_{\mathbb{R}}$ belonging to a  Jacobson-Morozov $sl_2(\mathbb{C})$-triple $\{E,2\Theta,F\}$, where $F$ is the sum of basis
vectors of simple negative root spaces with respect to  $B$, such that $\mathbf{P}$ is endowed with a $\Theta$-parahoric structure to make $\mathsf{A}$ into  a $\Theta$-parahoric formal connection.
  \end{enumerate}
\end{theorem}

\begin{theorem}[= Theorem  \ref{x}, Theorem \ref{xx}]Let $ \mathsf{A}$ be  a formal connection (or a Higgs field) on a   formal parahoric principal $G$-bundle $(\mathbf{P}, \Theta\in\mathfrak{t}_{\mathbb{R}}, e)$  with   the $(\Theta,e)$-reduced representation
\begin{align*}
 \mathsf{A}(e)=\sum\limits_{r\geq -c}\hat A^{(r)}z^{r}dz,
\end{align*}   where $c>1$. Fixing an integer $k$ ($2\leq k\leq c$).
\begin{enumerate}
  \item  Assume
\begin{itemize}
  \item $\hat A^{(-l)}$ lies in the center $\mathfrak{z}_\mathfrak{g}$ of  $\mathfrak{g}$ for any $l>k$;
  \item the semisimple part of $\mathrm{Res}(\hat A^{(-k)})$ in $\widehat{\mathfrak{l}}_\Theta$  is nonzero.
\end{itemize}
Then there exist a  trivialization $e'\in[e]_\Theta$ and a nonzero
semisimple element $S\in\mathfrak t$ such that the
$(\Theta,e')$-reduced representation
$$
\mathsf A(e')
=
\sum_{r\geq-c}\hat B^{(r)}z^r dz,
$$
where
$$
\hat B^{(r)}
=
\sum_{\lambda+i\geq0}
\sum_{i\in\mathbb Z}
X_{\lambda,i}^{(r)}z^i,
$$
satisfies the following properties:
\begin{enumerate}
\item $\hat B^{(-l)}=\hat A^{(-l)}$ for any $l>k$;
\item
$\mathrm{Res}(\hat B^{(-k)})$ is obtained from
$\mathrm{Res}(\hat A^{(-k)})$ by an $\widehat L_\Theta$-adjoint action;
\item denoting by $\mathrm{Res}_0(\hat B^{(-k)})$ the constant component of $\mathrm{Res}(\hat B^{(-k)})$,
the semisimple part of $\mathrm{Res}_0(\hat B^{(-k)})$ is $S$;
\item
$
[S,X_{\lambda,i}^{(r)}]=0
$
for all $r$, $\lambda$, and $i$.
\end{enumerate}
  \item Assume
\begin{itemize}
  \item $\hat A^{(-l)}$ lies in the center $\mathfrak{z}_\mathfrak{g}$ of  $\mathfrak{g}$ for any $l>k$;
  \item  $\mathrm{Res}\hat A^{(-k)}$ is a nonzero nilpotent element in  $\widehat{\mathfrak{l}}_\Theta$.
\end{itemize}
Then there exist a trivialization
$e'\in [e]_\Theta$ and a Jacobson-Morozov $\mathfrak{sl}(2,\mathbb{C})$-triple $(\hat P, \hat Q,  H)$ in $\widehat{\mathfrak{l}}_\Theta$ with $H\in \mathfrak{t}$ and $\hat P$ being obtained from
$\mathrm{Res}(\hat A^{(-k)})$ by an $\widehat L_\Theta$-adjoint action such that the $(\Theta,e')$-reduced representation
\begin{align*}
 \mathsf{A}(e')=\sum\limits_{r\geq -c}\hat B^{(r)}z^{r}dz
\end{align*}
satisfies the following properties:
  \begin{enumerate}
      \item $\hat B^{(-l)}=\hat A^{(-l)}$ for $l>k$;
      \item $\mathrm{Res}(\hat B^{(-k)})=\hat P$;
      \item   writing $\mathsf{A}(e')=(\sum\limits_{l=k+1}^{c}\hat B^{(-l)}z^{-l}+z^{-k}(\hat P+\sum\limits_{n_\nu>0}\hat B[n_\nu]))dz$ with $\hat B[n_\nu]\in\widehat{\mathfrak{p}}_\Theta[n_\nu]$, we have $[\hat Q,\hat B[n_\nu]]=0$ for each $n_\nu>0$.
    \end{enumerate}
\end{enumerate}

\end{theorem}

\begin{theorem}[= Theorem \ref{rl}]
Let $ \mathsf{A}$ be a formal connection  on a  principal $G$-bundle $\mathbf{P}$, then the following are equivalent.
\begin{enumerate}
  \item $ \mathsf{A}$ is  relatively regular.
  \item The induced formal $G_{\mathrm{ab}}$-connection  is regular, where $G_{\mathrm{ab}}=G/[G,G]$.
  \item $\mathbf{P}$ is endowed with a $\Theta$-parahoric structure $[e]_\Theta$ ($\Theta\in \mathfrak{t}_\mathbb{R}$) such that there are a form $\hat{Q}=\sum\limits_{r=-c\leq-2}^{-2}Q^{(r)}z^rdz$ with each  $Q^{(r)}\in \mathfrak{t}$  regular semisimple in $\mathfrak{g}$  and $\hat u\in \widehat{U}_\Theta$ satisfying that the gauge transformation by $\hat u$ changes $\hat Q+\mathsf A(e)$ into
$\hat Q+\hat R$,
                       where $\hat R=\sum\limits_{r\geq -1}R^{(r)}z^r dz$ with each  $R^{(r)}\in \mathfrak{t}$.
\end{enumerate}
\end{theorem}
\begin{theorem}[= Theorem \ref{rd}, Corollary \ref{cor}]
Let $\mathsf A$ be a formal connection on a formal parahoric principal
$G$-bundle $(P,\Theta\in\mathfrak t_{\mathbb R},e)$.
\begin{enumerate}
  \item there exist a positive integer $b$, a finite cover $(\Delta^\times)^{\sharp_b}$ defined by
$
        z=\zeta^b,
$
and a trivialization $e'$ of pullback $G$-bundle $\mathbf{P}^{\sharp_b}$  over $(\Delta^\times)^{\sharp_b}$ such that the pullback formal connection $\mathsf A^{\sharp_b}$ on $\mathbf{P}^{\sharp_b}$ has a representation
$$
\mathsf A^{\sharp_b}(e')
 =
 (
   \sum_{j=-c}^{-2} D_j\zeta^j
   +\hat C\,\zeta^{-1}
 )d\zeta,
 $$
where
$
        D_j\in\mathfrak t
       $, $
        \hat C\in\mathfrak{b}[\zeta]\bigcap\widehat{\mathfrak p}_{b\Theta}
$ for some Borel subalgebra $\mathfrak{b}\subset \mathfrak{g}$ containing $\mathfrak{t}$,
and
$
          [D_j,\hat C]=0
$
for all $j$.
  \item If the weight $\Theta$ is rational, namely
$ \alpha(\Theta)\in\mathbb Q
       $ for every root  $\alpha$, then
there exist a positive integer $b$, a finite cover
$
        z=\zeta^b,
$
and a trivialization $e'\in T_{\mathbf{P}^{\sharp_b}}$  such that
$$
\mathsf A^{\sharp_b}(e')
 =
 (
   \sum_{j=-c}^{-2} D_j\zeta^j
   +\hat C\,\zeta^{-1}
 )d\zeta,
 $$
where
$
        D_j\in\mathfrak t
       $,  $
        \hat C\in\widehat{\mathfrak l}_{b\Theta}
$,
and
$
          [D_j,\hat C]=0
$
for all $j$.
\end{enumerate}
\end{theorem}

\begin{theorem}[= Corollary \ref{xzz}]
Let  $ \mathsf{A}$ be a formal connection on a   formal parahoric principal $G$-bundle $(\mathbf{P}, \Theta, e)$  with the $\Theta$-reduced representation
\begin{align*}
 \mathsf{A}(e)=\sum\limits_{r\geq -c}\hat A^{(r)}z^{r}dz,
\end{align*}  where $c>1$. Put
 $$c'=\left\{
             \begin{array}{ll}
              c , & \hbox{$\mathrm{Res}(\hat A^{(-c)})$\textrm{ is a nilpotent element in  } $\widehat{\mathfrak{l}}_\Theta$;} \\
              c+1, & \hbox{otherwise.}
             \end{array}
           \right.
  $$
 Then there are a  Jacobson--Morozov $\mathfrak{sl}(2,\mathbb C)$-triple $(\hat P, \hat Q, H)$ in $\widehat{\mathfrak{l}}_\Theta$ with regular nilpotent $\hat P, \hat Q$, and
 $\hat g \in \widehat{G}$ such that under the trivialization $\hat ge$, $\mathsf{A}$ has a  representation 
\begin{align*}
  \mathsf{A}(\hat ge)=z^{-c'}(\hat P+\sum\limits_{n_\nu>0}\hat B[n_\nu])dz
\end{align*}satisfying $B[n_\nu]\in \mathrm{Ker(\mathrm{ad}_{\hat Q}|_{\widehat{\mathfrak{p}}_\Theta[n_\nu]})} $.
\end{theorem}
\

\noindent\textbf{Acknowledgements}.
The authors would like to express their gratitude to  Philip Boalch, Andres Fernandez Herrero, and Daniel Sage for their insightful communications and assistance in understanding their deep work. The authors also thank   Jianzhi  Han, Carlos Simpson, and   Kaitao  Xie for their valuable discussions, and thank  Hansheng Diao and Lingfei Yi for their interest in the present work.

Z. Hu is supported by the National Natural Science Foundation of China (No. 12271253). P. Huang was partially supported by the European Research Council (ERC) under the European Union's Horizon 2020 research and innovation program (grant agreement No 101018839) and Deutsche Forschungsgemeinschaft (DFG, Projektnummer 547382045). R. Sun is supported by an NSFC grant (No. 12501059) and a grant from Xiamen University (No. X2450215). R. Zong is supported by the National Natural Science Foundation of China (Key Program 12331002).

\section{Parahoric Subgroups}
In this section, we briefly recall some basic knowledge of parahoric subgroups.
Let $\mathbb{K}=\mathbb{C}(\!(z)\!)$ be the  field of formal Laurent series over $\mathbb{C}$ with the ring $\mathbb{A}=\mathbb{C}[\![z]\!]$ of integers, where $z$ denotes a uniformizing parameter.  For a positive natural number $b$, let $\mathbb{K}_b$ be the finite Galois
extension of $\mathbb{K}$ with Galois group canonically isomorphic to  the group $\mu_b$ of $b$-roots
of unity, and we define $\overline{\mathbb{K}}=\bigcup\limits_{b\in \mathbb{Z}^{>0}}\mathbb{K}_b$ as the algebraic closure of $\mathbb{K}$.
Let $G$ be a connected complex reductive Lie group with  Lie algebra $\mathfrak{g}$, and we set $\widehat{G}=G(\mathbb{K})$, the group of $\mathbb{K}$-points of $G$,   and $\widehat{\mathfrak{g}}=\mathfrak{g}\otimes_{\mathbb{C}}\mathbb{K}$ be its Lie algebra.
 A subgroup conjugate to the
preimage of a Borel subgroup $B$ of ${G}$ under the residue map $G(\mathbb{A})\rightarrow
G$ (i.e. taking $z$ to 0) is called an Iwahori subgroup. A proper subgroup of $\widehat{G}$ that is a finite union of double cosets of an Iwahori subgroup is called a parahoric subgroup. In particular, a parahoric subgroup is a  subgroup of $\widehat{G}$ containing an Iwahori subgroup.

Parahoric subgroups can be described by the theory of Bruhat-Tits building.
 Let $T\subset G$ be a maximal torus and its  Lie algebra is denoted by $\mathfrak{t}$.
Let $N_T$ be the normalizer of $T$, so that the Weyl group $W$ of $G$ is isomorphic to $N_T/T$. Let $E^*_T=\mathrm{Hom}(T,\mathbb{C}^*)$ be the group of characters of $T$, which is regarded as a lattice in $\mathfrak{t}^*$ by the canonical embedding $\chi\mapsto d\chi$, and similarly let $(E_T)_*=\mathrm{Hom}(\mathbb{C}^*,T)$ be the group of cocharacters  of $T$, which is regarded as a lattice in $\mathfrak{t}$ by the  canonical embedding $\nu\mapsto d\nu(1)$. The root decomposition of $\mathfrak{g}$ is given by $\mathfrak{g}=\bigoplus\limits_{\alpha \in\mathfrak{t}^*} \mathfrak{g}_\alpha$, where $\mathfrak{g}_0=\mathfrak{t}$ and $\mathfrak{g}_\alpha$ is the corresponding root space when $\alpha$ lies in the root system $\bigtriangleup$  of $\mathfrak{g}$.  A standard apartment  $\mathcal{A}$ is an affine space isomorphic to $\mathfrak{t}_{\mathbb{R}}:=(E_T)_*\otimes_{\mathbb{Z}}\mathbb{R}$.  $\mathcal{A}$ carries a cell structure, more precisely,  the facets are intersections of half-spaces determined by the affine hyperplanes $\{\Theta\in \mathcal{A}:\alpha(\Theta)=j\}$ for some $\alpha\in\bigtriangleup, j\in \mathbb{Z}$. The affine Weyl group $\widehat{W}:=N_{T}(\mathbb{K})/T(\mathbb{A})=W\ltimes(E_T)_*$ acts on the apartment $\mathcal{A}$ by affine transformation, more explicitly, $W$ acts on $\mathcal{A}$ via the adjoint action and the action of  $ (E_T)_*$  on $\mathcal{A}$ is a translation  as $z^\nu\cdot \Theta=\Theta-\nu$. Obviously, the action of the affine Weyl group preserves the cell structure on $\mathcal{A}$.

 Given $\Theta\in \mathcal{A}$, which is called a weight, the corresponding extended parahoric subgroup $\widehat{P}_\Theta$ of $\widehat{G}$ is defined as
\begin{align*}
 \widehat{P}_\Theta=\{\hat g\in \widehat{G}: z^\Theta \hat gz^{-\Theta}\textrm{ has  a limits as}\ z\rightarrow0\textrm{ along any ray}\}.
\end{align*}
   More explicitly, $\widehat{P}_\Theta$ is generated by \cite{b}
\begin{itemize}
  \item elements in $\widehat{L}_\Theta=\{z^{-\Theta} hz^{\Theta}\}$, where $h$ lies in the centralizer $C_G(\exp(2\pi\sqrt{-1}\Theta))$  of $\exp(2\pi\sqrt{-1}\Theta)$ in $G$,
  \item elements of the form $\exp{(Yz^i)}$, where $Y\in \mathfrak{g}_\alpha$ with $\alpha(\Theta)+i>0$ or $Y\in \mathfrak{t}, i>0$,
  \item elements of the form $\exp{(z^N\hat Y)}$, where $N$ is a sufficiently
large integer  and $\hat Y$ is a formal power series valued in $\mathfrak{g}$.
\end{itemize}
Therefore,  $\widehat{P}_\Theta$ is the semidirect product of the Levi subgroup $\widehat{L}_\Theta\simeq C_G(\exp(2\pi\sqrt{-1}\Theta))$ and the pro-unipotent radical $\widehat{U}_\Theta=\{\hat g\in \widehat{G}: z^\Theta \hat gz^{-\Theta}\textrm{ tends to 1  along any ray}\}$.
Let $\widehat{\mathcal{L}}_\Theta$ be the subgroup of $\widehat{L}_\Theta$ corresponding to the identity component of $C_G(\exp(2\pi\sqrt{-1}\Theta))$, then the parahoric subgroup $\widehat{\mathcal{P}}_\Theta$ associated to the weight $\Theta$ is defined as the group generated by $\widehat{\mathcal{L}}_\Theta$ and $\widehat{U}_\Theta$.
The Lie algebra $\mathfrak{g}$ decomposes as
$\mathfrak{g}=\bigoplus\limits_{\lambda\in \mathbb{R}}\mathfrak{g}^{(\Theta)}_{\lambda}$, where $\mathfrak{g}^{(\Theta)}_{\lambda}$ is the $\lambda$-eigenspace of the action $\mathrm{ad}_\Theta$, then
the Lie algebra of $\widehat{P}_\Theta$ and $\widehat{\mathcal{P}}_\Theta$ is given by
 \begin{align*}
   \widehat{\mathfrak{p}}_\Theta=\{\sum\limits_{\lambda+i\geq 0}\sum\limits_{i\in \mathbb{Z}} X_{\lambda,i}z^i\in \widehat{\mathfrak{g}}: X_{\lambda,i}\in \mathfrak{g}^{(\Theta)}_{\lambda}\}=\widehat{\mathfrak{l}}_\Theta\oplus \widehat{\mathfrak{u}}_\Theta,
 \end{align*}
 where
 \begin{align*}
    \widehat{\mathfrak{l}}_\Theta&=\{\sum\limits_{\lambda+i=0}\sum\limits_{i\in \mathbb{Z}} X_{\lambda,i}z^i\in \widehat{\mathfrak{g}}: X_{\lambda,i}\in \mathfrak{g}^{(\Theta)}_{\lambda}\},\\
  \widehat{\mathfrak{u}}_\Theta&=\{\sum\limits_{\lambda+i>0}\sum\limits_{i\in \mathbb{Z}} X_{\lambda,i}z^i\in \widehat{\mathfrak{g}}: X_{\lambda,i}\in \mathfrak{g}^{(\Theta)}_{\lambda}\}
 \end{align*}
 are the Lie algebras of $\widehat{L}_\theta$ ($\widehat{\mathcal{L}}_\theta$) and $\widehat{U}_\theta$, respectively.

 The Bruhat-Tits building $\mathbf{BT}(\widehat{G})$ associated to $\widehat{G}$ is defined as the quotient of $\widehat{G}\times\mathcal{A}$ by the equivalent relation that  $(\hat g,\Theta)\sim (\hat h, \Xi)$ iff there exists some $n\in N_{T}(\mathbb{K})$ such that $\Xi=n\cdot \Theta$ and $\hat g^{-1}\hat h n\in \widehat{P}_\Theta$ \cite{be}.
 $\mathbf{BT}(\widehat{G})$ is covered  by translations of the standard apartments. These are the apartments of $\mathbf{BT}(\widehat{G})$ inherit cell structures, and they  are parameterized by the split maximal tori in $\widehat{G}$. Obviously,  $\widehat{G}$ acts on $\mathbf{BT}(\widehat{G})$, then given $x\in\mathbf{BT}(\widehat{G})$, one defines the stabilizer of $x$ as $\mathrm{St}(x)=\{\hat g\in \widehat{G}: \hat g\cdot x=x\}$. In particular, when viewing $\Theta\in \mathcal{A}$ as a point $x$ lying in $\mathbf{BT}(\widehat{G})$, we have $\mathrm{St}(x)=\widehat{P}_\Theta$ (\cite[Lemma 12]{b}).  The  parahoric subgroup $\widehat{\mathcal{P}}_x$ of $\widehat{G}$ associated to the point $x\in\mathbf{BT}(\widehat{G}) $ is defined as the  identity component of  $\mathrm{St}(x)$. Hence, the parahoric subgroups are in bijective correspondence with the
facets of $\mathbf{BT}(\widehat{G})$. In particular, an Iwahori subgroup is the identity component of the  stabilizer of a  facet with maximal dimension (i.e. an alcove) of  $\mathbf{BT}(\widehat{G})$. Fixing an alcove $\mathcal{F}$,  if we are happy to work modulo $\widehat{G}$-conjugation, we only have to work with facets in the closure $\bar{\mathcal{F}}$ of  $\mathcal{F}$ since every facet  of $\mathbf{BT}(\widehat{G})$ lies in a  $\widehat{G}$-orbit of some facet of $\bar{\mathcal{F}}$.

\begin{definition}
For $\hat X=\sum\limits_{\lambda+i\geq 0}\sum\limits_{i\in \mathbb{Z}} X_{\lambda,i}z^i\in \widehat{\mathfrak{p}}_\Theta$, the residue $ \mathrm{Res}(\hat X)$ of  $\hat X$ is defined as its Levi component, namely
\begin{align*}
 \mathrm{Res}(\hat X)=\sum\limits_{\lambda+i= 0}\sum_{i\in \mathbb{Z}} X_{\lambda,i}z^i.
\end{align*}
\end{definition}

\begin{proposition}\label{1}Let $l_{\Theta}: \widehat{P}_\Theta\rightarrow \widehat{L}_\Theta$ be the Levi projection, then for any $\hat g\in \widehat{P}_\Theta$, the following diagram
$$\CD
   \widehat{\mathfrak{p}}_\Theta\ @> \mathrm{Res}>> \widehat{\mathfrak{l}}_\Theta \\
  @V \mathrm{Ad}_{\hat g} VV @V  \mathrm{Ad}_{l_{\Theta}(\hat g)} VV  \\
  \widehat{\mathfrak{p}}_\Theta\ @> \mathrm{Res}>> \widehat{\mathfrak{l}}_\Theta
\endCD
$$ commutes.
\end{proposition}

\begin{proof}
Firstly, we consider generator of $\widehat{P}_\Theta$ with the form ${\hat g}=z^{-\Theta} hz^{\Theta}$, where $h\in C_G(\exp(2\pi\sqrt{-1}\Theta))$,  then we have
\begin{align*}
\mathrm{Ad}_h(\mathrm{Ad}_{z^{\Theta}} (\hat X))&=\sum\limits_{\lambda+i\geq 0}\sum\limits_{i\in \mathbb{Z}}\mathrm{Ad}_h(X_{\lambda,i})z^{\lambda +i }\\
&=\sum\limits_{\lambda+i\geq 0}\sum\limits_{i\in \mathbb{Z}}\sum\limits_{k\in\mathbb{Z}}Y_{\lambda+k,i}z^{\lambda +i }\end{align*}
where  $\mathrm{Ad}_h(X_{\lambda,i})=\sum\limits_{k\in\mathbb{Z}}Y_{\lambda+k,i}$ with $Y_{\lambda+k,i}\in\mathfrak{g}^{(\Theta)}_{\lambda+k}$, and have
\begin{align*}\mathrm{Ad}_{z^{-\Theta}}(\mathrm{Ad}_h(\mathrm{Ad}_{z^{\Theta}} (\hat X)))=\sum\limits_{\lambda+i\geq 0}\sum\limits_{i\in \mathbb{Z}}\sum\limits_{k\in\mathbb{Z}}Y_{\lambda+k,i}z^{i-k}.
\end{align*}
Therefore,
\begin{align*}
  \mathrm{Res}(\mathrm{Ad}_{\hat g}(\hat X))=\sum\limits_{\lambda+i= 0}\sum_{i\in\mathbb{Z}}\sum\limits_{k\in\mathbb{Z}}Y_{\lambda+k,i}z^{i-k}=\mathrm{Ad}_{\hat g}(\mathrm{Res}(\hat X)).
\end{align*}

 For the other two types of generators of $\widehat{P}_\Theta$, it suffices to consider $\hat g=\exp{(Y_{\mu,j}z^j)}$ with $Y_{\mu,j}\in \mathfrak{g}^{(\Theta)}_{\mu}$ and $\mu+j\geq 0$, then we have
\begin{align*}
  \mathrm{Ad}_{\hat g} (\hat X)=\sum\limits_{\lambda+i\geq 0}\sum\limits_{i\in \mathbb{Z}}(X_{\lambda,i}z^i+[Y_{\mu,j},X_{\lambda,i}]z^{i+j}+\frac{1}{2}[Y_{\mu,j},[Y_{\mu,j},X_{\lambda,i}]]z^{i+2j}+\cdots).
\end{align*}
Note that
\begin{align*}
  [\Theta,\underbrace{[Y_{\mu,j},[Y_{\mu,j},\cdots,[Y_{\mu,j}}_{n},X_{\lambda,i}]\cdots]]]=(\lambda+n\mu)\underbrace{[Y_{\mu,j},[Y_{\mu,j},\cdots,[Y_{\mu,j}}_{n},X_{\lambda,i}]\cdots]].
\end{align*}
To calculate $\mathrm{Res}( \mathrm{Ad}_{\hat{g}} (\hat X))$, we should pick the terms $\underbrace{[Y_{\mu,j},[Y_{\mu,j},\cdots,[Y_{\mu,j}}_{n},X_{\lambda,i}]\cdots]]$ with $\lambda+i+n(\mu+j)=0$. Since $\lambda+i\geq 0, \mu+j\geq 0$, we must have $\lambda+i=\mu+j= 0$. Therefore, we obtain
\begin{align*}
  \mathrm{Res}(\mathrm{Ad}_{\hat g}(\hat X))=\left\{
                                             \begin{array}{ll}
                                               \mathrm{Res}(\hat X), & \hbox{$\mu+j\neq 0$;} \\
                                              \mathrm{Ad}_{\exp{(Y_{\mu,j}z^j)}}(\mathrm{Res}(\hat X)), & \hbox{$\mu+j=0$.}
                                             \end{array}
                                           \right.
\end{align*}
We complete the proof.
\end{proof}

\section{Formal Connections (or Higgs Fields) on Parahoric Principal Bundles }

The formal disc $\mathrm{Spec}(\mathbb{A})$ and the formal punctured disc $\mathrm{Spec}(\mathbb{\mathbb{K}})$ are  denoted by $\Delta$ and $\Delta^{\times}$, respectively. The module of K\"{a}hler differentials $\Omega^1_{\mathbb{A}/\mathbb{C}}$ is spanned as an $\mathbb{A}$-module by formal elements $df$ for every $f\in\mathbb{A}$, subject to the Leibniz rule. The completion $\widehat{\Omega^1_{\mathbb{A}/\mathbb{C}}}=\lim\limits_{\longleftarrow\atop n}\Omega^1_{\mathbb{A}/\mathbb{C}}/z^n\Omega^1_{\mathbb{A}/\mathbb{C}}$ is a free $\mathbb{A}$-module of rank 1, and we have the natural completion map $\mathcal{C}:\Omega^1_{\mathbb{A}/\mathbb{C}} \rightarrow\widehat{\Omega^1_{\mathbb{A}/\mathbb{C}}}$. Then we define $\Omega^1_{\mathbb{K}/\mathbb{C}}=\widehat{\Omega^1_{\mathbb{A}/\mathbb{C}}}[\frac{1}{z}]$. For simplicity, the $\mathcal{C}$-image of $dz$ is also denoted by $dz$.
A formal principal $G$-bundle $\mathbf{P}$ is a principal $G$-bundle over $\Delta^{\times}$. Since $\mathbf{P}$ ia trivializable,  we always have a trivialization $e: \Delta^{\times}\rightarrow \mathbf{P}$, which induces an isomorphism  $e: \Aut(\mathbf{P})\rightarrow\widehat{G}$ of groups.
 be The set of trivializations on $\mathbf{P}$ is denoted by $T_\mathbf{P}$.
\begin{definition}
Given a weight $\Theta\in\mathfrak t_{\mathbb R}$,
a $\Theta$-parahoric structure on a formal
principal $G$-bundle $\mathbf P$ is the  $\widehat{P}_\Theta$-obit $[e]_{\Theta}$ of a trivialization $e\in T_\mathbf{P}$. We also call the triple $(\mathbf P, \Theta, e)$ a  formal parahoric principal $G$-bundle.
\end{definition}

\begin{remark}If $\widehat{P}_\Theta$ is self-normalizing  in $\widehat{G}$ (for example, if $G$ is semisimple and simply-connected),  the above definition is equivalent to specifying a subgroup $\mathcal{P}\subset \Aut(\mathbf{P})$ such that $e(\mathcal{P})=\widehat{P}_\Theta$ for some trivialization $e\in T_\mathbf{P}$.
\end{remark}

\begin{definition}\
\begin{enumerate}
  \item Given a formal principal $G$-bundle $\mathbf{P}$, let  $\mathsf{A}$ be a function from $T_\mathbf{P}$ to $\Omega^1(\widehat{\mathfrak{g}})=\Omega^1_{\mathbb{K}/\mathbb{C}}\otimes_\mathbb{K}\widehat{\mathfrak{g}}$.
For $e_1,e_2\in T_\mathbf{P}$ with $\hat g=e_2\circ e_1^{-1}\in \Aut (\Delta^\times\times G) \simeq\widehat{G}$,
\begin{itemize}
  \item if  $ \mathsf{A}(e_2)=\mathrm{Ad}_{\hat g}(\mathsf{A}(e_1))$, i.e. $ \mathsf{A}(e_2)$ is the $\mathrm{Ad}_{\hat g}$-gauge transformation of $\mathsf{A}(e_1)$,  then we call $\mathsf{A}$  a formal Higgs field on  $\mathbf{P}$;
  \item if $ \mathsf{A}(e_2)=\widetilde{\mathrm{Ad}}_{\hat g}(\mathsf{A}(e_1)):=\mathrm{Ad}_{\hat g}(\mathsf{A}(e_1))+\mathcal{C}(\hat g^*\omega)$, where $\omega$ is the Maurer-Cartan form on $G$,  i.e. $ \mathsf{A}(e_2)$ is the $\widetilde{\mathrm{Ad}}_{\hat g}$-gauge transformation of $\mathsf{A}(e_1)$, then we call $\mathsf{A}$  a formal connection on  $\mathbf{P}$.
\end{itemize}

  \item Given a  formal parahoric principal $G$-bundle $(\mathbf{P}, \Theta, e)$, a formal  connection (or Higgs field) $\mathsf{A}$ on  $\mathbf{P}$ is called  $\Theta$-parahoric if  $\mathsf{A}(e)$ is a $\Theta$-logarithmic form, i.e. it  can be written as $\mathsf{A}(e)=\hat g\frac{dz}{z}$ for $\hat g\in \widehat{\mathfrak{p}}_\Theta$.
\end{enumerate}

\end{definition}

\begin{remark}Let $\mathsf{A}$ be a  $\Theta$-parahoric formal connection (or Higgs field), by
\cite[Lemma 3]{b}, for any trivialization $e'\in [e]_\Theta$,  $\mathsf{A}(e')$ is a $\Theta$-logarithmic form.  Therefore, the above definition makes sense.
\end{remark}

\begin{remark}
In  the global setting, the existence of connections on $\mathcal{G}_\Theta$-torsor $\mathbb{T}$ over a complex smooth projective curve $X$ with a reduced divisor $D=\{x_1,\cdots,x_n\}$ \cite {ppp,BS,ppp1}, where  $\Theta=\{\Theta_1,\cdots, \Theta_n\}$ is a collection of weights at $x_i$ and $\mathcal{G}_\Theta$ is a parahoric Bruhat-Tits group scheme over $(X,D)$, is an interesting question. Indeed, when $\Theta_1,\cdots, \Theta_n$ are rational weights, as the generalization of Biswas's result for the parabolic case (\cite[Theorem 4.1]{bis}), one has the following criterion. Namely,    there is a connection on  $\mathbb{T}$ such that the  Levi part of the constant term of its residue at each $x_i$ is conjugate to $\Theta_i$ if and only if for every Levi subgroup scheme $\mathcal{L}$ of $\mathcal{G}_\Theta$ and every $\mathcal{L}$-reduction $\mathbb{T}_{\sigma}$ of $\mathbb{T}$ given by a section $\sigma: X\rightarrow \mathbb{T}/\mathcal{L}$, the parahoric  degree (cf. \cite{KSZ24,HKSZ22} for the definition) of $\mathbb{T}_{\sigma}$
with respect to any  character $\chi: \mathcal{L}\rightarrow\mathbb{G}_m$ is zero.
\end{remark}

Let $ \mathsf{A}$ be a formal connection (or Higgs field)  on a  formal parahoric principal $G$-bundle $\mathbf{P}$. Fixing a weight $\Theta\in\mathfrak t_{\mathbb R}$ and a   trivialization $e\in T_\mathbf{P}$,  $ \mathsf{A}$ can be   written   as the following   $(\Theta,e)$-representation
\begin{align}\label{ds}
 \mathsf{A}(e)=\sum\limits_{r\geq -c}\hat A^{(r)}z^{r}dz,
\end{align}
where $ \hat A^{(r)}\in \widehat{\mathfrak{p}}_\Theta$. The integer $j$ is called the leading index of $(\Theta,e)$-representation if $\hat A^{(-j)}$ is nonzero and  $\hat A^{(-j-1)}=\cdots\hat A^{(-c)}=0$.
We define the finite dimensional subspaces $\widehat{\mathfrak{p}}_\Theta[l]$ of  $\widehat{\mathfrak{p}}_\Theta$ for a non-negative number $l$ as
\begin{align*}
 \widehat{\mathfrak{p}}_\Theta[l]=\{\sum\limits_{\lambda+i=l}\sum\limits_{i\in \mathbb{Z}} X_{\lambda,i}z^{i}\in \widehat{\mathfrak{p}}_\Theta: X_{\lambda,i}\in \mathfrak{g}^{(\Theta)}_{\lambda}\},
\end{align*}
 then we have $$\widehat{\mathfrak{p}}_\Theta=\prod\limits_{j\in \mathbb{Z}^{\geq0}}z^j\widehat{\mathfrak{q}}_\Theta$$
 for $$\widehat{\mathfrak{q}}_\Theta=\bigoplus\limits_{0\leq l<1}\widehat{\mathfrak{p}}_\Theta[l]\simeq \widehat{\mathfrak{p}}_\Theta/z\widehat{\mathfrak{p}}_\Theta.$$
  we write
  \begin{align*}
   \hat A^{(r)}=\sum\limits _{0\leq l_1<l_2<\cdots}\hat A^{(r)}[l_{\mu^{(r)}}]
  \end{align*}
for $\hat  A^{(r)}[l_{\mu^{(r)}}]\in \widehat{\mathfrak{p}}_\Theta[l_{\mu^{(r)}}]$, and write
\begin{align}\label{m}
  A^{(r)}[l_{\mu^{(r)}}]=\sum\limits_{\lambda+i^{(\mu^{(r)})}=l_{\mu^{(r)}}}\sum\limits_{i^{(\mu^{(r)})}\in \mathbb{Z}} X_{\lambda,i^{(\mu^{(r)})}}z^{i^{(\mu^{(r)})}}
  \end{align}
 Let the set  $\mathbf{S}_\Theta=\{r,l_{\mu^{(r)}},i^{(\mu^{(r)})}\}$ collect the data  appearing in the nonzero summands of \eqref{m}, then we  call it the shape of the representation \eqref{ds}.

\begin{definition}
 We say a  $(\Theta,e)$-representation \eqref{ds}  is reduced  if
\begin{itemize}
  \item $\hat A^{(-c)}$ is nonzero,
  \item  any nonzero term $ \hat A^{(r)}$ lies in $ \widehat{\mathfrak{q}}_\Theta$.
\end{itemize}
The reduced  $(\Theta,e)$-representation  is unique, and we call the shape of the reduced $(\Theta,e)$-representation is the $(\Theta,e)$-shape of the formal connection (or Higgs field).
\end{definition}
\begin{remark}For a formal parahoric principal $G$-bundle $(\mathbf{P},\Theta, e)$, if $\mathsf{A}$ is a formal Higgs field on $\mathbf{P}$ or $\mathsf{A}$ is a formal connection on $\mathbf{P}$ such that  the reduced $(\Theta,e)$-representation has  the  leading index $c>1$, then the leading index is independent of
   the choice of the trivializations in the orbit $[e]_\Theta$.  Hence for these cases, we also call the leading index of the reduced $(\Theta,e)$-representation the $\Theta$-order of $ \mathsf{A}$.\end{remark}

\begin{example}\label{exa}Let $ \mathsf{A}$ be  a formal connection (or a Higgs field) on  a   formal principal $G$-bundle $\mathbf{P}$, and under a trivialization $e\in T_\mathbf{P}$  we write
 \begin{align*}
    \mathsf{A}(e)=\sum\limits_{r\geq -c}A^{(r)}z^{r}dz
  \end{align*}
  for $A^{(r)}\in\mathfrak{g}$. Choose a maximal torus $T$ of $G$ and a Borel subgroup $B\subset G$ containing $T$, then $B$ determines a positive (negative) system $\triangle^+(B)$ ($\triangle^-(B)$) of roots of $\mathfrak{g}$.
  Consider a weight $\Theta=\frac{\rho^\vee}{C}$, where $\rho^\vee$ is the half-sum of  all positive coroots and $C$ is a sufficiently large number. Since
$0<\alpha(\Theta)<1$ for each  $\alpha\in \triangle^+(B)$,  $\Theta$ belongs to an alcove \cite{bp}, thus
$\widehat P_\Theta$ is an Iwahori subgroup. Moreover, we have the decomposition
$$\widehat{\mathfrak{p}}_\Theta=\mathfrak{t}[[z]]\bigoplus\mathfrak{n}^+[[z]]\bigoplus z\mathfrak{n}^-[[z]],$$
where $\mathfrak{n}^{\pm}=\bigoplus\limits_{\alpha\in \triangle^{\pm}(B)}\mathfrak{g}_\alpha$.
 We have a $(\Theta,e)$-representation
  \begin{align*}
  \mathsf{A}(e)=\sum\limits_{r\geq -c}\hat A^{(r)}z^{r-1}dz,
  \end{align*}
  where $\hat A^{(r)}=zA^{(r)}\in \widehat{\mathfrak{p}}_{\Theta}$.
\end{example}

\begin{definition} Let $ \mathsf{A}$ be a  formal connection (or a Higgs field) on a formal parahoric principal $(\mathbf{P}, \Theta, e)$, then for an integer $c\geq 1$ we define
\begin{align*}
M_\Theta(\mathsf{A},e;-c)=\{\hat g\in \widehat{G}:\mathsf{A}({\hat g}^{-1} e)\in z^{-c} \widehat{\mathfrak{p}}_\Theta dz\},
\end{align*}
which is an ind-subscheme of $\widehat{G}$, and define the quotients
\begin{align*}
N_\Theta(\mathsf{A},e;-c)&=M_\Theta(\mathsf{A},e;-c)/\widehat{P}_\Theta,\\
N^0_\Theta(\mathsf{A},e;-c)&=M_\Theta(\mathsf{A},e;-c)/\widehat{\mathcal{P}}_\Theta.
\end{align*}
They are both  closed ind-subschemes of the affine $\Theta$-parahoric flag variety $\widehat{G}/\widehat{P}_\Theta$. According to the terminology of \cite{fz}, one can call  $N_\Theta(\mathsf{A},e;-c)$ the  $\Theta$-parahoric deformed affine Springer fiber (or the  $\Theta$-parahoric  affine Springer fiber).
\end{definition}

\begin{remark}To our knowledge, the  notion  of $\Theta$-parahoric flag variety was  first introduced by Pappas and Rapoport \cite{p}. After fixing a reference trivialization of formal principal $G$-bundle $\mathbf{P}$, $\widehat{G}/\widehat{P}_\Theta$ is the moduli space of $\Theta$-parahoric structures on $\mathbf{P}$.
 In particular, there are many studies of the following two special cases in literature.
 \begin{itemize}
   \item When $\widehat{P}_\Theta=G(\mathbb{A})$ (i.e. $\Theta=0$), $\widehat{G}/G(\mathbb{A})$ is called an affine Grassmannian, which is treated as the moduli space of principal $G$-bundles over $\Delta$ with a trivialization on its restricting to $\Delta^\times$. This is a key object in the study of geometric representation theory and geometric Langlands program \cite{bl,MV1,fa,a,f,MV2,ac,z,Laf}.
   \item When $\widehat{P}_\Theta$ is an Iwahori subgroup $\widehat I$ of $\widehat{G}$,  $\widehat{G}/\widehat I$ is called an affine flag variety, which is a fibration over $\widehat{G}/G(\mathbb{A})$ with the fibers $G/B$ \cite{g,fg,be}.
 \end{itemize}
\end{remark}

\begin{proposition} \label{cv} Let  $ \mathsf{A}$ be a formal connection  on a   formal parahoric principal $G$-bundle $(\mathbf{P}, \Theta, e)$, the for an integer  $c\geq1$,  we have
\begin{align*}
  \dim_\mathbb{C}N_\Theta(\mathsf{A},e;-c)
 <(c+1+2\lfloor\lambda_m\rfloor)\dim_\mathbb{C}\mathfrak{g},
\end{align*}
where  $\lambda_m$ denotes  the maximum of  eigenvalues of $\mathrm{ad}_\Theta$-action on $\mathfrak{g}$, and $\lfloor\lambda_m\rfloor$ denotes the biggest  integer smaller  than or equal to $\lambda_m$.
\end{proposition}
\begin{proof}
For $\hat g\in  M_{\Theta}(\mathsf{A},e;-c)$, we define the following $\mathbb{C}$-vector spaces
\begin{align*}
T_{\hat g}&=\{\hat X\in\widehat{\mathfrak{g}}:\hat{\mathcal{D}}\hat X\in \widehat{\mathfrak{p}}_\Theta \frac{dz}{z^c}\}/\widehat{\mathfrak{p}}_\Theta,\\
T'_{\hat g}&=\{\hat X\in\widehat{\mathfrak{g}}:\hat{\mathcal{D}}\hat X\in \mathfrak{g}(\mathbb{A}) \frac{dz}{z^{c+\lfloor\lambda_m\rfloor}}\}/ \mathfrak{g}(\mathbb{A}),\\
T''_{\hat g}&=\{\hat X\in\widehat{\mathfrak{g}}:\hat{\mathcal{D}}\hat X\in \mathfrak{g}(\mathbb{A}) \frac{dz}{z^{c+\lfloor\lambda_m\rfloor}}\}/(\widehat{\mathfrak{p}}_\Theta\bigcap \mathfrak{g}(\mathbb{A})),
\end{align*}
where $\hat{\mathcal{D}}=d-\mathrm{ad}_{\mathsf{A}({\hat g}^{-1}e)}$, $\mathfrak{g}(\mathbb{A})=\mathfrak{g}\otimes_\mathbb{C}\mathbb{A}$.
There is a canonical isomorphism $ T_{\hat g}= T_{\hat g\hat h}$ for every $\hat h\in \widehat{P}_\Theta$, hence $T_{\hat g}$ is canonically isomorphic to the tangent space of $N_\Theta(\mathsf{A},e;-c)$ at $\hat g\widehat{P}_\Theta$, where the orbit $\hat g\widehat{P}_\Theta$ is treated as a point lying in $ N_\Theta(\mathsf{A},e;-c)$. It follows from \cite[Lemma 8]{fz} that
\begin{align*}
  \dim_\mathbb{C}T'_{\hat g}\leq (c+\lfloor\lambda_m\rfloor)\dim_\mathbb{C}\mathfrak{g},
\end{align*}
meanwhile, we have
\begin{align*}
 &\dim_\mathbb{C}(\mathfrak{g}(\mathbb{A})/(\widehat{\mathfrak{p}}_\Theta\bigcap \mathfrak{g}(\mathbb{A})))\\
 = &\ \sum_{i=0}^{\lfloor\lambda_m\rfloor}\sum_{\lambda<-i}\dim_\mathbb{C}\mathfrak{g}_\lambda^{(\Theta)}
 <\ (1+\lfloor\lambda_m\rfloor)\dim_\mathbb{C}\mathfrak{g}.
\end{align*}
 Therefore, the (in)equalities
\begin{align*}
  \dim_\mathbb{C}T_{\hat g}\leq \dim_\mathbb{C}T''_{\hat g}
  =\dim_\mathbb{C}T'_{\hat g}+\dim_\mathbb{C}(\mathfrak{g}(\mathbb{A})/(\widehat{\mathfrak{p}}_\Theta\bigcap \mathfrak{g}(\mathbb{A})))
\end{align*}leads to the proposition.
\end{proof}
\begin{remark}
 If $ \mathsf{A}$ is a formal Higgs field, in general, the above proposition does not hold anymore. Actually,   as pointed out in \cite{fz}, the tangent space of  $N(\mathsf{A},e;-c)$ is   never finite-dimensional  because it is highly non-reduced.
\end{remark}

\begin{definition}For $\hat g\in M_\Theta(\mathsf{A},e;-c)$ with $c\geq 1$, we define
the map $\chi: M_\Theta(\mathsf{A},e;-c)\rightarrow \widehat{\mathfrak{l}}_\Theta$   as
 \begin{align*}
   \chi(\hat g)=\left\{
                  \begin{array}{ll}
                  \mathrm{ Res}( \hat B^{(-c)})+\delta_{c,1}\Theta, & \hbox{$\mathsf{A}$\textrm{ has a reduced }$(\Theta,{\hat g}^{-1} e)$-\textrm{representation }$\hat B^{(-c)}z^{-c}dz+\cdots$;} \\
                    \delta_{c,1}\Theta, & \hbox{otherwise.}
                  \end{array}
                \right.
 \end{align*}
Moreover,
by Proposition \ref{1}, this map  induces a map from $ N_\Theta(\mathsf{A},e;-c)$ to the quotient $\widehat{\mathfrak{l}}_\Theta/\widehat{L}_\Theta$ via the adjoint action, which is also denoted by $\chi$.
\end{definition}
\begin{remark}
For  ${\hat g}=z^{-\Theta} hz^{\Theta}$ with $h\in C_G(\Theta)$, we calculate $d{\hat g}{{\hat g}}^{-1}=(\mathrm{Ad}_{\hat g}(\Theta)-\Theta)\frac{dz}{z}$. Therefore, when $c=1$, the additional  term $\Theta$ compensates for the Levi component of the logarithmic derivative, which  ensures that the map $\chi$ descends.
\end{remark}
Given a weight $\Theta\in \mathfrak{t}_\mathbb{R}$, let $\widehat{I}_\Theta$ be an Iwahori subgroup contained in $\widehat{P}_\Theta$ with Lie algebra $\widehat{\mathfrak{i}}_\Theta$, and let $\widehat{B}_\Theta=l_\Theta(\widehat{I}_\Theta)\subset \widehat{L}_\Theta$ be the  corresponding Borel subgroup with Lie algebra $\widehat{ \mathfrak{b}}_\Theta$. There is a weight $\Theta_I\in \mathfrak{t}_\mathbb{R}$ such that $\widehat{I}_\Theta=\widehat{\mathcal{P}}_{\Theta_I}$. Then $M_{\Theta_I}(\mathsf{A},e;-c)$ is a subset of   $M_\Theta(\mathsf{A},e;-c)$, and  there is  a natural morphism
$\pi_\Theta:N_{\Theta_I}(\mathsf{A},e;-c) \rightarrow N^0_\Theta(\mathsf{A},e;-c)$. Let $\mathcal{B}_\Theta$ be the set of Borel subalgebras of $\widehat{\mathfrak{l}}_\Theta$, and  we
define the Grothendieck alteration $\mathrm{GA}(\widehat{\mathfrak{l}}_\Theta)$ as the variety of pairs $(\hat X, \widehat{\mathfrak{b}})$, where $\hat X \in \widehat{ \mathfrak{b}},  \widehat{\mathfrak{b}}\in \mathcal{B}_\Theta$. Then  $\chi$ provides a map from $N_{\Theta_I}(\mathsf{A},e;-c)$ to $\mathrm{GA}(\widehat{\mathfrak{l}}_\Theta)/\widehat{\mathcal{L}}_\Theta$.
Parallel argument with that in \cite[Proposition 12]{fz} gives rise to the following proposition.

\begin{proposition} \label{316}For an integer $c\geq 1$, there is a natural Cartesian diagram as follows
\begin{align*}
 \CD
 N_{\Theta_I}(\mathsf{A},e;-c) @>\chi>> \mathrm{GA}(\widehat{\mathfrak{l}}_\Theta)/\widehat{\mathcal{L}}_\Theta\\
  @V \pi_\Theta VV @V f VV  \\
  N^0_\Theta(\mathsf{A},e;-c) @>\chi>>\widehat{\mathfrak{l}}_\Theta/\widehat{\mathcal{L}}_\Theta
\endCD,
\end{align*}
where $f$ denotes the forgetful map $(\hat X, \widehat{\mathfrak{b}})\mapsto \hat X$.

\end{proposition}
\begin{proof}
Assume $\hat g \widehat{\mathcal{P}}_\Theta\in N_\Theta(\mathsf{A},e;-c)$, by definition,  $\chi(\hat g)=\mathrm{Res}(\hat X)+\delta_{c,1}\Theta$, where $\hat X:=z^c\frac{\mathsf{A}({\hat g}^{-1}e)}{dz}\in \widehat{\mathfrak{p}}_\Theta $. In particular, $\hat X$ belongs to $\widehat{\mathfrak{i}}_\Theta$ iff $\chi(\hat g)$ belongs to $\widehat{\mathfrak{b}}_\Theta$. The fiber of the natural projection $\widehat{G}/I_\Theta\rightarrow \widehat{G}/\widehat{\mathcal{P}}_\Theta$ is $\widehat{\mathcal{P}}_\Theta/\widehat{I}_\Theta\simeq\widehat{\mathcal{L}}_\Theta/\widehat{B}_\Theta\simeq\mathcal{B}_\Theta $, hence $\pi_\Theta^{-1}(\hat g \widehat{\mathcal{P}}_\Theta)\simeq\{\widehat{\mathfrak{b}}\in \mathcal{B}_\Theta: \chi(\hat g)\in\widehat{\mathfrak{b}}\}$. The Cartesian property follows.
\end{proof}

\section{Reduction for Regular Cases}

\begin{definition}
Let  $ \mathsf{A}$ be a  $\Theta$-parahoric formal connection on a   formal parahoric principal $G$-bundle $(\mathbf{P}, \Theta, e)$. We say $ \mathsf{A}$ is of Boalch-type if there is a trivialization $e'\in[e]_\Theta$
  \begin{align*}
    \mathsf{A}(e')=\hat X\frac{dz}{z}
  \end{align*}
  with $\hat X=\sum\limits_{i\in\mathbb{Z}}X_{i}z^i\in \widehat{\mathfrak{p}}_\Theta$ satisfies $[R,X_i]=iX_i$ for some semisimple element  $R$ in $\mathfrak{\mathfrak{g}}$.
  \end{definition}

 \begin{definition}\label{d2}
 We say a  formal connection   $\mathsf{A}$ on a formal principal $G$-bundle $\mathbf{P}$ is intrinsically regular if there exists a trivialization $e\in T_\mathbf{P}$ and $\bar g\in G(\overline{\mathbb{K}})$ such that the $\widetilde{\mathrm{Ad}}_{\bar g}$-gauge transformation of $\mathsf{A}(e)$ is of logarithmic form, which is called the reduced form. In other words, there exists $\acute{g}\in G(\mathcal{O}_{\mathrm{Spec}(\overline{\mathbb{K}})})$ such that $\widetilde{\mathrm{Ad}}_{ \acute{g}}$-gauge transformation of the pullback of  $\mathsf{A}$ to $\mathrm{Spec}(\overline{\mathbb{K}})$ is of logarithmic form under some local  trivialization of pullback of $\mathbf{P}$.
\end{definition}

\begin{theorem}\label{d3} Let $\mathsf{A}$ be a formal connection  on  a formal principal $G$-bundle $\mathbf{P}$.
$\mathsf{A}$ is  intrinsically regular if and only if $\mathbf{P}$ is endowed with a $\Theta$-parahoric structure such that $\mathsf{A}$ is of Boalch-type.
\end{theorem}
\begin{proof}Firstly, we show that any  $\Theta$-parahoric formal connection $\mathsf{A}$ of Boalch-type is  intrinsically regular. By definition, under a suitable trivialization $e\in T_\mathbf{P}$, we write
 \begin{align*}
  \mathsf{A}(e)=\sum_{i\in \mathbb{Z}}X_iz^i\frac{dz}{z}
 \end{align*}
 with $[R,X_i]=iX_i$ for some semisimple element $R$. Let $\mathfrak{t}'$ be  a Cartan subalgebra of $\mathfrak{g}$ containing $R$, $T'$ be  the maximal torus of $G$ corresponding to $\mathfrak{t}'$ and $\bigtriangleup'$ be the corresponding root system. It is known that  there is  $\nu\in (E_{T'})_*$ such that $\alpha(\nu)=b\alpha(R)$ for some positive integer $b$ and for  any root $\alpha\in \bigtriangleup'^{(R)}_\mathbb{Z}:=\{\alpha\in \bigtriangleup':\alpha(R)\in \mathbb{Z}\}$ (cf. \cite[Proposition 4.5]{s}). Define $\hat t=\nu(z^{-1})\in T'(\mathbb{F})$, then $\alpha(\hat t)=z^{-\alpha(\nu)}=z^{-b\alpha(R)}$ for any $\alpha\in\bigtriangleup'^{(R)}_\mathbb{Z}$.
  Let $(\Delta^\times)^{\sharp_b}$  be the the $b$-cover of $\Delta^\times$ with $\zeta=\sqrt[b]{z}$, and denote  by $\bullet^{\sharp_b}$  the pullback of $\bullet$ on  $\Delta^\times$ to $(\Delta^\times)^{\sharp_b}$.
Then
 the the formal connection $\mathsf{A}^{\sharp_b}$ on $\mathbf{P}^{\sharp_b}$ is written as
 \begin{align*}
  \mathsf{A}^{\sharp_b}( e^{{\sharp_b}})=b(\sum_{i\in \mathbb{Z}}X_i\zeta^{bi})\frac{d\zeta}{\zeta}
 \end{align*}
 under the trivialization $e^{\sharp_b}$ of $\mathbf{P}^{\sharp_b}$.
We can calculate
\begin{align*}
\widetilde{\mathrm{Ad_{\hat t^{\frac{1}{b}}}}}(\mathsf{A}^\sharp( e^\sharp))=b(\sum_{i\in \mathbb{Z}}X_i-\frac{\nu}{b})\frac{d\zeta}{\zeta},
\end{align*}
 which is a logarithmic form.

Conversely, according to \cite[Theorem 4.2]{s} (also \cite[Theorem 3.6]{h}), we can assume the  intrinsically regular formal connection  $\mathsf{A}$ have the canonical reduced form $Y\frac{dz}{z}$ with $Y\in\mathfrak{g}$, for which the gauge stabilizer over $\overline{\mathbb{K}}$ is the centralizer $C_G(Y)=\{\gamma\in G: \mathrm{Ad}_\gamma Y=Y\}$. It is known that the intrinsically regular formal connections which  have  the canonical reduced form $Y\frac{dz}{z}$ are classified by the Galois cohomology $H^1(\mathrm{Gal}(\overline{\mathbb{K}}/\mathbb{K}),C_G(Y))$, namely the  conjugacy
classes of elements of finite order in $C_G(Y)$ \cite{c,s,h}. Hence, we chooses an element $\gamma\in C_G(Y)$ of order $b$ representing the descent  class of $\mathsf{A}$.
Let  $\sigma$ be  the generator of $\mathrm{Gal}(\mathbb{K}_b/\mathbb{K})$ defined by $\sigma(\zeta)=\exp(\frac{2\pi\sqrt{-1}}{b})\zeta$. There is a trivialization $e\in T_\mathbf{P}$ and $\hat u\in \widehat{G}(\mathbb{K}_b)$   such that
$$\mathsf{A}(e)=\widetilde{\mathrm{Ad}}_{\hat u}(bY\frac{d \zeta}{\zeta}).$$
This means that $\hat u^{-1}\sigma(\hat u)=\gamma$.

 The Jordan decomposition of $Y$ is given by $Y=S+N$ for semisimple $S$ and  nilpotent $N$, then
 $\gamma\in C_G(S)\bigcap C_G(N)$. One introduces the centralizer  $J=C_{C_G(S)}\gamma$   and chooses a maximal torus $T$ of $J$. Since $J$ has the same rank as $G$, $T$ is also a maximal torus of $G$. It is clear that $\gamma$ lies in $T$ and $S$ lies in the Lie algebra $\mathfrak{t}$ of $T$.
  Since $\mathrm{Ad}_{\gamma}N=N$, writing $\gamma=\exp{(2\pi\sqrt{-1}\Gamma)}$ for rational cocharacter $\Gamma$,  $N$ has a finite decomposition
\begin{align*}
 N=\sum\limits_{q\in \mathbb{Z}}N_{i}
\end{align*}
 with $[\Gamma, N_{i}]=iN_{i}$. Consider the   cocharacter  $\chi=b\Gamma\in (E_T)_*$, then $\hat h:=\chi(\zeta)=z^\Gamma$. Since $\sigma(\hat h)=\hat h\gamma$, we have $\sigma(\hat u\hat h^{-1})=\hat u\gamma(\hat h\gamma)^{-1}=\hat u\hat h^{-1}$, thus $\hat u\hat h^{-1}\in \widehat{G}$. It follows that there is a trivialization $e'\in T_\mathbf{P}$ such that
 \begin{align*}
   \mathsf{A}(e')&=\widetilde{\mathrm{Ad}}_{\hat h}(Y\frac{dz}{z})\\
  & = (S+\Gamma+\sum\limits_{i\in \mathbb{Z}}N_{i}z^i)\frac{dz}{z}
 \end{align*}
 We choose $R=\Gamma$ and  $\Theta=-\Gamma$, then  $N_{i}z^i\in\widehat{\mathfrak{l}}_\Theta$, hence $\mathsf{A}$ is a $\Theta$-parahoric formal connection of Boalch type.
\end{proof}

\begin{theorem}[{\cite[Theorem 6]{b}}] Let  $ \mathsf{A}$ be a  $\Theta$-parahoric formal connection on a   formal parahoric principal $G$-bundle $(\mathbf{P}, \Theta, e)$,  then $ \mathsf{A}$ is of Boalch-type.
\end{theorem}

\begin{remark}
Actually, we can write $\mathsf{A}(e')=\hat X\frac{dz}{z}=\sum\limits_{\lambda+i\geq 0}\sum\limits_{i\in\mathbb{Z}}X_{\lambda,i}z^i\frac{dz}{z}$ under a suitable trivialization $e'\in[e]_\Theta$ such that the semisimple element $R$ is chosen as the semisimple part of  $X_{0,0}$.
\end{remark}

\begin{corollary}\label{d4}
Let $\mathsf{A}$ be a formal connection  on  a formal principal $G$-bundle $\mathbf{P}$, then
$\mathsf{A}$ is  intrinsically regular if and only if $\mathbf{P}$ is endowed with a $\Theta$-parahoric structure such that $\mathsf{A}$ is a $\Theta$-parahoric formal connection.
\end{corollary}

\begin{definition}(\cite[Definition 2.13]{bs})\label{d1} We say a  formal connection   $\mathsf{A}$ on a formal principal $G$-bundle $\mathbf{P}$ is extrinsically regular if for
all representations $(V,\rho)$ consisting of a finite-dimensional vector space $V$ and  a homomorphism   $\rho:G\rightarrow \mathrm{GL}(V)$, the induced connection $\mathsf{A}_\rho$ on the associated  vector bundle $\mathbf{P}_\rho=\mathbf{P}\times_\rho V$ is  regular, namely there is a trivialization $e$ of $\mathbf{P}_\rho$ such that $\mathsf{A}_\rho(e)$ is a logarithmic form.
\end{definition}

\begin{proposition}\label{s2}A  formal connection   $\mathsf{A}$ on a formal principal $G$-bundle $\mathbf{P}$ is extrinsically regular if and only if there is a faithful  representation $\rho: G\rightarrow \mathrm{GL}(V)$  such that the induced connection $\mathsf{A}_\rho$ on the associated  vector bundle $\mathbf{P}_\rho=\mathbf{P}\times_\rho V$ is  regular.
\end{proposition}

\begin{proof}
It is known that any representation $\rho:G\rightarrow \mathrm{GL}(V)$ can be identified with a subquotient of the representation $\bigoplus\limits_i(\varrho^{\otimes a_i}\bigotimes(\varrho^*)^{\otimes b_i})$ for some $a_i, b_i\in \mathbb{Z}^{\geq 0} $, where $\varrho:G\rightarrow \mathrm{GL}(U)$ is a faithful representation. Hence,  the  regularity of  induced connection  $\mathsf{A}_{\varrho}$ on the associated  vector bundle $\mathbf{P}\times_{\varrho} U$ guarantees that
the induced connection $\mathsf{A}_{\rho}$  is  regular. 
\end{proof}

In \cite{bs}, for formal connections, Bremer and  Sage developed a geometric analogue of Moy-Prasad's theorem on minimal $K$-types for admissible representations
of $p$-adic groups.   Let us briefly introduce their theory. Let $(V,\rho)$ be a finite dimensional complex representation of $G$, and $(\widehat{V}=V\otimes_\mathbb{C}\mathbb{K},\widehat{\rho}=\rho\otimes \mathrm{Id})$  be the corresponding representation of $\widehat{G}$. For any  $x\in \mathbf{BT}(\widehat{G})$, there is a canonical decreasing $\mathbb{R}$-filtration $\{\widehat{V}_{x,r}\}_{r\in \mathbb{R}}$ on $\widehat{V}$, called the Moy-Prasad filtration, satisfying the following properties
\begin{itemize}
   \item $z\widehat{V}_{x,r}=\widehat{V}_{x,r+1}$,
   \item $\widehat{V}_{\hat gx,r}=\hat g\widehat{V}_{x,r}$ for $\hat g\in \widehat{G}$,
   \item   $\mathrm{St}(x)$  preserves the filtration $\{\widehat{V}_{x,r}\}_{r\in \mathbb{R}}$ ,
   \item the set of critical numbers $r$ with $\widehat{V}_{x,r}^+=\bigcup\limits_{s>r}\widehat{V}_{x,s}\subsetneq \widehat{V}_{x,r}$ is discrete.
 \end{itemize}
In particular, for the adjoint representation, we have the Moy-Prasad filtration $\{\widehat{\mathfrak{g}}_{x,r}\}_{r\in \mathbb{R}}$ on $\widehat{\mathfrak{g}}$. A
triple $(x,r,\beta)$ consisting of $x\in \mathbf{BT}(\widehat{G})$, a nonnegative real number $r$ and $\beta\in(\widehat{\mathfrak{g}}_{x,-r}/\widehat{\mathfrak{g}}^+_{x,-r})\otimes \frac{dz}{z}$, is called a $\widehat{G}$-stratum of depth $r$.

 \begin{theorem}[{\cite[Theorem 2.14]{bs}}]\label{z}
  Let $\mathsf{A}$ be a formal connection on a formal principal $G$-bundle $\mathbf{P}$ with a trivialization $e\in T_{\mathbf{P}}$.
 \begin{enumerate}
   \item There exists a $\widehat{G}$-stratum $(x,r,\beta)$ with $x$ being a rational point in $\mathbf{BT}(\widehat{G})$ such that
   \begin{align}\label{d}
    (d-\mathsf{A}(e)-s\frac{dz}{z}- \beta)(\widehat{V}_{x,s})\subset \widehat{V}^+_{x,s-r}\otimes \frac{dz}{z}
   \end{align}
   for any representation $V$ of $G$ and any real number $s$, where $\mathsf{A}$ and $e$ are viewed as the induced formal connection and trivialization on the associated  bundle, respectively, and $\beta$ is viewed as a representative lying in  $\Omega^1(\widehat{\mathfrak{g}}_{x,-r})$.
   \item Define
   $r_{(\mathbf{P},\mathsf{A})}$ to be the minimal depth of $\widehat{G}$-strata satisfying the condition \eqref{d}, which is independent of the choice of the trivialization $e$,
  then $\mathsf{A}$ is  extrinsically regular if and only if $r_{(\mathbf{P},\mathsf{A})}=0$.
  \item If a $\widehat{G}$-stratum $(x,r,\beta)$ with  $r>0$  satisfies the condition \eqref{d}, then $r=r_{(\mathbf{P},\mathsf{A})}$ if and only if each representative
$\beta$ is non-nilpotent.
 \end{enumerate}
 \end{theorem}

As an application of Bremer-Sage's result, we have the following theorem which provides a criterion for the extrinsic regularity of formal connections.

\begin{theorem}\label{dd} Let $\mathsf{A}$ be a formal connection  on  a formal principal $G$-bundle $\mathbf{P}$, then
$\mathsf{A}$ is  extrinsically regular if and only if $\mathbf{P}$ is endowed with a $\Theta$-parahoric structure such that $\mathsf{A}$ is a $\Theta$-parahoric formal connection.
\end{theorem}

 \begin{proof}
 Firstly, we show that any  $\Theta$-parahoric formal connection $\mathsf{A}$ is extrinsically regular. By Proposition \ref{s2}, it suffices to check the definition for a faithful representation of $G$ into $\mathrm{GL}(V)$. In this representation, $\Theta$ is a real diagonal  matrix,  then we can choose an integral diagonal matrix $\Xi=\mathrm{diag}(\theta_1,\cdots,\theta_n )$. Let $\Xi=\mathrm{diag}(\lfloor\theta_1\rfloor,\cdots,\lfloor\theta_n \rfloor)$, then  due to the the fact that $k+\lfloor\theta_i\rfloor-\lfloor\theta_j\rfloor\geq 0$ if  $k+\theta_i-\theta_j\geq 0$ for an integer $k$,  after the $\widetilde{\mathrm{Ad}}_{z^\Xi}$-gauge transformation for  $z^\Xi\in \widehat{\mathrm{GL}(V)}$, the (induced) formal connection $\mathsf{A}$ is made into a  logarithmic form.

 Conversely, if $\mathsf{A}$ is  extrinsically regular, then due to Theorem \ref{z},  there exists a $\widehat{G}$-stratum $(x,0,\beta)$ with  $x$ being a rational point in $\mathbf{BT}(\widehat{G})$  satisfying the condition \eqref{d} for the adjoint representation of $G$, namely
 \begin{align*}
   (d-\mathsf{A}(e)-s\frac{dz}{z}- \beta)(\widehat{\mathfrak{g}}_{x,s})\subset \widehat{\mathfrak{g}}^+_{x,s}\otimes \frac{dz}{z}
 \end{align*}
 for certain trivialization $e$. By the action of $\widehat{G}$,  we can assume  $x$ is a rational weight in the standard apartment $\mathcal{A}$, then we have (cf. \cite[Section 2.6]{bs})
 \begin{align*}
  \widehat{\mathfrak{g}}_{x,s}=\bigoplus_{\chi(x)+i\geq s}\bigoplus_{\chi\in E^*(T)}\mathfrak{g}_\chi z^i,
 \end{align*}
 hence $\widehat{\mathfrak{g}}_{x,0}=\widehat{\mathfrak{p}}_x, \widehat{\mathfrak{g}}^+_{x,0}=\widehat{\mathfrak{u}}_x$ and $\beta$ can be treated as an element in  $ \mathfrak{l}_x\otimes\frac{dz}{z}$. Since the operator  $z\frac{\partial}{\partial z}$ preserve $\widehat{\mathfrak{p}}_x$, we have $\mathsf{A}(e)(\widehat{\mathfrak{p}}_x)\subset\widehat{\mathfrak{p}}_x\otimes\frac{dz}{z}$.
  Write $\mathsf{A}(e)=\hat X\frac{dz}{z}$, then $\hat X$ lies in the normalizer of $\widehat{\mathfrak{p}}_x$ in $\widehat{\mathfrak{g}}$, thus $\hat X$ is decomposed as $\hat X=\hat Y+\hat Z$, where $\hat Y\in \widehat{\mathfrak{p}}_x$ and $\hat Z\in \mathfrak{z}_\mathfrak{\mathfrak{g}}\otimes_{\mathbb{C}}\mathbb{K}$ for the center $ \mathfrak{z}_\mathfrak{\mathfrak{g}}$ of $\mathfrak{g}$. It remains to show that $\hat Z\in \widehat{\mathfrak{p}}_x$. For every character $\chi: G\rightarrow\mathbb{G}_m$, the induced connection $\mathsf{A}_\chi$ is regular, therefore $\chi_*(\hat X)\in \mathbb{A}$. Writing  $\hat Z=\sum\limits_iZ_iz^i$, for each $i<0$, we have $\chi_*(Z_i)=0$ for all characters $\chi$ since $\chi_*(\hat Y)\in \mathbb{A}$, which implies the vanishing of $Z_i$. Therefore, $\hat Z\in \mathfrak{z}_\mathfrak{\mathfrak{g}}\otimes_{\mathbb{C}}\mathbb{A}\subset \widehat{\mathfrak{p}}_x$.
\end{proof}

 \begin{remark}
 As mentioned in the Introduction, Corollary \ref{d4} and Theorem \ref{dd} can be regarded as an analogue of Deligne extension of regular connections on vector bundles in the context of formal principal $G$-bundles, namely a regular formal connection on a formal principal $G$-bundle can be extended to a logahoric connection on some  parahoric  group scheme over the formal disc.
   Combining this with Boalch's  parahoric Riemann-Hilbert correspondence (cf. \cite[Theorem D, Corollary E]{b}), we can see that regular connections are fully classified by the monodromy  data.
                           \end{remark}

The following theorem provides the criteria of the regularity of formal connections.

 \begin{theorem}\label{ddd} Let $\mathsf{A}$ be a formal connection  on  a formal principal $G$-bundle $\mathbf{P}$. Then the following are equivalent.
\begin{enumerate}
\item There exists a trivialization $e\in T_{\mathbf{P}}$ and $\bar g\in G(\overline{\mathbb{K}})$ such that the $\widetilde{\mathrm{Ad}}_{\bar g}$-gauge transformation of $\mathsf{A}(e)$ is of logarithmic form, where $\overline{\mathbb{K}}$ is  the algebraic closure of $\mathbb{K}$.
 \item For all representations $(V,\rho)$ consisting of a finite-dimensional vector space $V$ and  a homomorphism   $\rho:G\rightarrow \mathrm{GL}(V)$, the induced connection $\mathsf{A}_\rho$ on the associated  vector bundle $\mathbf{P}_\rho=\mathbf{P}\times_\rho V$ is  regular, namely there is a trivialization $e$ of $\mathbf{P}_\rho$ such that $\mathsf{A}_\rho(e)$ is of logarithmic form.
\item $\mathbf{P}$ is endowed with a $\Theta$-parahoric structure such that $\mathsf{A}$ is a $\Theta$-parahoric formal connection.
\end{enumerate}
 When  $G$ is not a torus, the above conditions are also equivalent to the following condition.
  \begin{enumerate}
    \item[(4)] Given  a maximal torus $T\subset G$ with Lie algebra $\mathfrak{t}$ and a Borel subgroup $B\subset G$ containing $T$,
 there is a weight $\Theta\in \mathfrak{t}_{\mathbb{R}}$ belonging to a  Jacobson-Morozov $sl_2(\mathbb{C})$-triple $\{E,2\Theta,F\}$, where $F$ is the sum of basis
vectors of simple negative root spaces with respect to  $B$, such that $\mathbf{P}$ is endowed with a $\Theta$-parahoric structure to make $\mathsf{A}$ into  a $\Theta$-parahoric formal connection.
  \end{enumerate}

 \end{theorem}
\begin{proof}
Combining Corollary \ref{d4} and  Theorem \ref{dd} together immediately leads to the equivalents of the first three conditions. We prove (4).
Employing the same notations in  the proof of Theorem \ref{d3}, we choose a Borel subgroup $B_J$ of $J$ such that $N$ lies in the Lie algebra of unipotent  radical $R_u(B_J)$ and choose a maximal torus $T_0\subset B_J$, moreover we choose a Borel subgroup $B_0\subset G$ containing $B_J$ such that $N$ lies in the Lie algebra of $R_u(B_0)$.   We write $$N=\sum\limits_{\alpha\in \Delta^+(B_0)}N_\alpha$$ with $N_\alpha\in \mathfrak{g}_\alpha$. Since    $\mathrm{Ad}_\gamma N_\alpha=N_\alpha$ for $\gamma=\exp{(2\pi\sqrt{-1}\Gamma)}$, we have $\alpha(\Gamma)\in \mathbb{Z}$ when $N_\alpha\neq 0$. Replacing $\Gamma$, if necessary, by $\Gamma+\kappa$ for a suitable   cocharacter $\kappa\in
(E_{T_0})_*$, we may assume $\Gamma$ is dominant, thus $\alpha(\Gamma)\in \mathbb{Z}^{\geq0}$ when $N_\alpha\neq 0$. There is a suitable trivialization $e$ such that
$$\mathsf{A}(e)=(S'+\sum\limits_{\alpha>0}z^{\alpha(\Gamma)}N_\alpha)\frac{dz}{z}$$ for some $S'$ lies in the Lie algebra $ \mathfrak{t}_0$ of $T_0$, namely
$\mathsf{A}(e)\in \mathfrak{b}_0[[z]]\frac{dz}{z}$ for the Lie algebra $\mathfrak{b}_0$ of $B_0$.

Now conjugating the pair $(T_0,B_0)$ to the prescribed pair $(T,B)$,  we obtain a trivialization $e'$ such that
$
\mathsf A(e')
\in
\mathfrak b[[z]]\otimes\frac{dz}{z}
$.
Let $\Theta\in\mathfrak t_{\mathbb R}$ be the  weigh  such  that
$
\{E,2\Theta,F\}
$
is a  Jacobson--Morozov $\mathfrak{sl}_2(\mathbb C)$-triple, where
$$
F=\sum_{\alpha\in\Delta_s^+(B)}F_{-\alpha}.
$$ for $\Delta_s^+(B)$ consisting of simple roots  is a principal nilpotent element.
It follows from the identities  $$[2\Theta, F]=-2\sum\limits_{\alpha\in\Delta_s^+(B)}\alpha(\Theta)F_{-\alpha}=-2F$$ that
$\alpha(\Theta)>0$
for  every positive root $\alpha$.
 Therefore,
$
\mathfrak b[[z]]
\subset
\widehat{\mathfrak p}_\Theta
$.
\end{proof}

\begin{remark}Assume $G$ is semisimple of adjoint type,
by Frenkel-Zhu's Borel reduction theorem \cite{fz}, for any formal connection $\mathsf{A}$,  there is a trivialization $e$  such that
       \begin{align*}
         \mathsf{A}(e)=(\sum_{i=1}^\ell A_ip_i z^{-r_i}+\sum\limits_{\alpha\in \triangle_{s}^+}F_{-\alpha})dz,
       \end{align*}
  where $\ell$ denotes the rank of $G$,  $A_i\in \mathbb{A}$, $r_i$ is a nonnegative integer if $A_i\neq 0$ and set $r_i=-\infty$ if $A_i=0$,  $p_i\in \mathfrak{g}$ is  an eigenvector $\mathrm{ad}_\Theta$ lying in $\mathrm{Ker}(\mathrm{ad}_E)$ for $\{E,2\Theta,\sum\limits_{\alpha\in \triangle_{s}^+}F_{-\alpha}\}$ forming a  Jacobson-Morozov $sl_2(\mathbb{C})$-triple with nilpotent $E$.  Chen and  Kamgarpour showed that $\mathsf{A}$ is intrinsically regular if and if the following conditions
  \begin{align*}
  r_i-d_i\leq 1,\ i=1,\cdots, \ell,
\end{align*}hold, where the nonnegative integer  $d_i$ is the corresponding eigenvalue of $p_i$ (cf. \cite[Proposition 3]{bk}\footnote{We thank the anonymous referee for pointing out this reference.}).
\end{remark}

\begin{example}Let $\mathsf{A}$ be a formal connection  on  a formal principal $G$-bundle $\mathbf{P}$. If there is a trivialization $e$ of $\mathbf{P}$ such that
\begin{align*}
 \mathsf{A}(e)=\sum\limits_{r\geq r_0}A^{(r)}z^rdz
\end{align*}
with $A^{(r)}\in\bigoplus\limits_{\lambda>0}\mathfrak{g}_\lambda^{(\Theta)}$ (or $A^{(r)}\in\bigoplus\limits_{\lambda<0}\mathfrak{g}_\lambda^{(\Theta)}$), then $\mathsf{A}$ is extrinsically and intrinsically regular.
Actually, assuming $r_0<-1$,  we take the weight
\begin{align*}
  \Theta'=\frac{-r_0-1}{\min\{\lambda:\lambda>0\}} \Theta\ (\textrm{or }\Theta'=\frac{-r_0-1}{\max\{\lambda:\lambda<0\}}\Theta),
\end{align*}
then  $\mathbf{P}$ is endowed with a $\Theta'$-parahoric structure such that $e$ is a $\Theta'$-trivialization and $A$ is  a $\Theta'$-parahoric  connection.

\end{example}

 We have the following  local-to-global extension in the sense of Deligne.
 \begin{corollary}The following three categories are equivalent:
  \begin{itemize}
    \item  the category of regular $G$-connections on the formal punctured disc;
    \item the category of  $G$-connections on the punctured affine line $\mathbb{P}^1\backslash\{0,\infty\}$ that are regular at both 0 and $\infty$;
    \item  the category of $\G$-representations of $\mathbb{Z}$.
  \end{itemize}
\end{corollary}
\begin{remark}For $G=\mathrm{GL}(n,\mathbb{C})$,  this equivalence was established by Deligne (cf. \cite[Proposition 13.35]{del}).  Katz later generalized Deligne's equivalence to the categories of special connections  \cite{kat}. Deligne-Katz correspondence is also generalized to the $D$-modules in positive characteristic \cite{ki} and the quasi-unipotent overconvergent isocrystals \cite{ma}.
\end{remark}
\begin{proof}Let  $\mathsf{A}$ be a   formal connection  on  a formal principal $G$-bundle $\mathbf{P}$, then we may assume $\mathsf{A}$ is of Boalch type as
    $$\mathsf{A}(e)=(\sum_{i\in \mathbb{Z}}X_iz^i)\frac{dz}{z}$$
  with  $[R,X_i]=iX_i$ for some  semisimple $R$.  There are finitely  many nontrivial terms in the above summation. Choosing a maximal torus $T$ of $G$, we way assume $R\in \mathfrak{t}$. According to the decomposition $\mathfrak{t}=\mathfrak{t}_\mathbb{R}\bigoplus\sqrt{-1}\mathfrak{t}_\mathbb{R}$, we write $R=R_a+\sqrt{-1}R_b$ with $R_a, R_b\in \mathfrak{t}_\mathbb{R} $. Decompose $X_i$ as $X_i=X_{i,0}+\sum\limits_{\alpha\in \triangle}X_{i,\alpha}$ for $X_{i,0}\in \mathfrak{t}$ and $X_{i,\alpha}\in \mathfrak{g}_\alpha$. When $X_{i,\alpha}$ is nonzero, $\alpha(R)=i$ implies that  $\alpha(R_a)=i$, thereby $[R_a, X_i]=iX_i$.  Changing variable $w=\frac{1}{z}$ leads to
    $-(S+\sum\limits_{i\in \mathbb{Z}}X_iw^{-i})\frac{dw}{w}$, which can be treated as a $\widehat{\mathfrak{p}}_\Theta$-connection with $\Theta=R_a$. Therefore,  $\mathsf{A}$ can be regularly extended to $\infty$. By Tannakian formalism, this equivalence can be also derived from   Deligne's theorem.
\end{proof}
    \begin{remark}
 Since principal bundle has two equivalent definitions by Tannakian formalism, where one is the intrinsic definition as usual, and the other one is the extrinsic definition via the tensor functor from the category of representations of structure group to the category of vector bundles. Many definitions of principal bundle admit two approaches--intrinsic one and extrinsic one. In general, the equivalence of these two ways (i.e. the Tannakian functoriality) is not evident. Here, we also give another three examples.
        \begin{enumerate}
        \item The notion of  semisimpleness of  formal connection $\mathsf{A}$ on a formal principal $G$-bundle $\mathbf{P}$ is  important for establishing the Jordan decomposition  of  $\mathsf{A}$ \cite{l,kz,k}.  It also can be defined by two approaches: the intrinsic definition is that there exists a trivialization $e$ of $\mathbf{P}$ and $\bar g\in G(\overline{\mathbb{K}})$ such that the $\widetilde{\mathrm{Ad}}_{\bar g}$-gauge transformation of $\mathsf{A}(e)$ is of the form $\bar t dz$ for $\bar t\in \mathfrak{t}(\overline{\mathbb{K}})$; and the extrinsic definition is that the induced connection $\mathsf{A}_\mathfrak{g}$ on the adjoint vector bundle $\mathbf{P}_\mathfrak{g}$ is semisimple, i.e. any  $\mathsf{A}_\mathfrak{g}$-invariant subbundle has an $\mathsf{A}_\mathfrak{g}$-invariant complement. As expected, these two definitions are equivalent (cf. \cite[Theorem 8]{k}).
        \item Ramanathan introduced the semistability condition for principal bundles over algebraic curves in an intrinsic way, and he showed it is equivalent to a certain extrinsic semistability (i.e. the semistability of the adjoint vector bundle), which plays a crucial role in the construction of the moduli space of principal bundles \cite{r}. However,  extrinsic stability is strictly stronger than intrinsic stability. In particular, if $G$ has  a nontrivial proper connected center, there is no extrinsically  stable principal $G$-bundles \cite{hd}.  In addition, it is noteworthy that Ramanathan's equivalence does not hold anymore in positive characteristic \cite{ft}. Also, it is not clear whether we have a similar equivalence for the parabolic (or parahoric) principal bundles considered in \cite{bp,KSZ24,HKSZ22} (for example, cf. the paragraph below \cite[Theorem 5.1]{bp}). For $G=\mathrm{GL}(n,\mathbb{C})$, one easily checks that  Ramanathan's (semi-)stability is equivalent to the (semi-)stability of the associated vector bundle via the fundamental representation (cf. \cite{hd} for the usual case and \cite{ KSZ24} for the parabolic case). In \cite{si2,si}, Simpson used an alternative version of extrinsic semistability defined as the semistability of the associated vector bundle via a faithful representation (similar with our definition of extrinsic regularity). In particular, if $G$ has trivial center and the principal bundle has vanishing Chern classes, Simpson's semistability is equivalent to  Ramanathan's semistability.
        \item  For flat $G$-connections (or $G$-Higgs bundles, $G$-local systems) over a connected smooth quasi-projective variety, we can introduce the notion of rigidity. Roughly speaking,  such an object is called rigid, cohomologically rigid, or physically rigid if, respectively, it defines an isolated (possibly non-reduced) point of the relevant moduli space, a smooth isolated point of that moduli space, or the unique point of that moduli space \cite{klp,ffff}. It is not clear whether these rigid properties are equivalent to those on the adjoint flat bundles (or Higgs bundles, local systems) \cite{si}.
        \end{enumerate}
 \end{remark}

\section{Reduction for Non-regular Cases}
Here non-regularity only means that the $\Theta$-order is greater than 1, which may not be the genuine irregularity in the sense of Definition \ref{d2} or Definition \ref{d1}. However, if for the formal connection $ \mathsf{A}$, there is a $(\Theta,e)$-reduced representation such that   $\mathrm{Res}(\hat A^{(-c)})$  has a nonzero semisimple part with $c>1$ (i.e. the setting in the following Theorem \ref{x}), then $ \mathsf{A}$ is certainly irregular (cf.  \cite[Proposition 4.7]{h}). Indeed, the non-nilpotent Levi residue determines a fundamental stratum of positive depth $c-1$.  $\mathrm{Res}(\hat A^{(-c)})$ and its constant term $\mathrm{Res}_0(\hat A^{(-c)})$ are also both treated as  locally finite endomorphisms  on $\widehat{\mathfrak{p}}_\Theta$, hence
their Jordan decompositions in $\widehat{\mathfrak l}_\Theta$ and $\mathfrak g$, respectively, are compatible with the Jordan decompositions in $\widehat{\mathfrak p}_\Theta$ \cite{ka}.

\subsection{Non-nilpotent Leading Coefficient}

\begin{theorem}\label{x}
Let $\mathsf A$ be a formal connection (or a Higgs field) on a formal
parahoric principal $G$-bundle $(\mathbf P,\Theta\in \mathfrak{t}_\mathbb{R},e)$ with a
$(\Theta,e)$-reduced representation
$$
\mathsf A(e)
=
\sum_{r\geq -c}\hat A^{(r)}z^r dz,
$$ where $c>1$. Fixing an integer $k$ ($2\leq k\leq c$), assume
\begin{itemize}
  \item $\hat A^{(-l)}$ lies in the center $\mathfrak{z}_\mathfrak{g}$ of  $\mathfrak{g}$ for any $l>k$;
  \item the semisimple part of $\mathrm{Res}(\hat A^{(-k)})$ in $\widehat{\mathfrak{l}}_\Theta$  is nonzero.
\end{itemize}
Then there exist a  trivialization $e'\in[e]_\Theta$ and a nonzero
semisimple element $S\in\mathfrak t$ such that the
$(\Theta,e')$-reduced representation
$$
\mathsf A(e')
=
\sum_{r\geq-c}\hat B^{(r)}z^r dz,
$$
where
$$
\hat B^{(r)}
=
\sum_{\lambda+i\geq0}
\sum_{i\in\mathbb Z}
X_{\lambda,i}^{(r)}z^i,
$$
satisfies the following properties:
\begin{enumerate}
\item $\hat B^{(-l)}=\hat A^{(-l)}$ for any $l>k$;
\item
$\mathrm{Res}(\hat B^{(-k)})$ is obtained from
$\mathrm{Res}(\hat A^{(-k)})$ by an $\widehat L_\Theta$-adjoint action;
\item denoting by $\mathrm{Res}_0(\hat B^{(-k)})$ the constant component of $\mathrm{Res}(\hat B^{(-k)})$,
the semisimple part of $\mathrm{Res}_0(\hat B^{(-k)})$ is $S$;
\item
$
[S,X_{\lambda,i}^{(r)}]=0
$
for all $r$, $\lambda$, and $i$.
\end{enumerate}
\end{theorem}

\begin{proof}Since $c>1$, for the  gauge transformation  $\hat g\in \widehat{P}_\Theta$ in our consideration, $\hat gd {\hat g}^{-1} $ does not affect our work, so we only need to consider the $\mathrm{Ad}_{\hat g}$-gauge transformation in the following calculations. By assumption, $\sum\limits_{l=k+1}^{c}\hat A^{(-l)}z^{-l}\in \widehat{\mathfrak{z}_\mathfrak{g}}$ is invariant under the gauge transformation by $\hat g\in \widehat{G}$, hence without loss of generality, we may assume $k=c$.

The identification $\widehat{\mathfrak{l}}_\Theta$ with the Lie algebra $\mathfrak{c}_G(\exp(2\pi\sqrt{-1}\Theta))$  of $ C_G(\exp(2\pi\sqrt{-1}\Theta))$ via the gauge transformation by $z^\Theta$ changes $\mathrm{Res}(\hat A^{(-c)})=\sum\limits_{\lambda+i=0}\sum\limits_{i\in\mathbb{Z}}A_{\lambda,i}z^i$ into $(\mathrm{Res}(\hat A^{(-c)}))|_{z=1}=\sum\limits_{\lambda+i=0}\sum\limits_{i\in\mathbb{Z}}A_{\lambda,i}$. Consider the Jordan decomposition in $\mathfrak{c}_G(\exp(2\pi\sqrt{-1}\Theta))$ $$(\mathrm{Res}(\hat A^{(-c)}))|_{z=1}=S+N$$ with semisimple $S$ and nilpotent $N$. By assumption, $S$ is nonzero.  After a conjugation by an element  $h$ lying in $C_G(\exp(2\pi\sqrt{-1}\Theta))$, we may  assume $S\in \mathfrak{t}$.
For the centralizer $K=C_{C_G(\exp(2\pi\sqrt{-1}\Theta))}(S)$, we choose a Borel subgroup $B_K\subset K$ containing $T$.  After a further conjugation by an
element $k$ lying in $K$, which leaves S fixed, we may assume that $N$ lies in the Lie algebra $\mathfrak{r}_u(B_K)$ of  unipotent radical $R_u(B_K)$. Decompose $N$ into $\mathrm{ad}_\Theta$-eigenspaces as $N=\sum\limits_iN_i$ with $[\Theta, N_i]=-iN_i$.  Since   $\mathfrak{r}_u(B_K)$  is preserved by $\mathrm{ad}_\Theta$,  each component $N_i$ lies in
 $\mathfrak{r}_u(B_K)$, hence it  is nilpotent. Therefore, under the trivialization $\tilde e=\hat ge$ for $\hat g=z^{-\Theta}khz^{\Theta}\in \widehat{L}_\Theta$, we have
$$
\mathsf A(\tilde e)
=
\sum_{r\geq-c}\hat C^{(r)}z^r dz,
$$
where  $\mathrm{Res}\hat C^{(-c)}=S+\sum\limits_iN_iz^i$ with $S$ being exactly the semisimple part of $\mathrm{Res}_0\hat C^{(-c)}$ and satisfying $[S,N_i]=0$.

We next consider all positive parahoric degrees.
Let
$
0=l_0<l_1<l_2<\cdots
$
be the increasing sequence of parahoric degrees occurring in
$\widehat{\mathfrak p}_\Theta$.
For every $n_\nu\geq0$, we set
$$
\hat C[n_\nu]
=
\sum_{m\geq0, \mu\geq0,
m+l_\mu=n_\nu}
z^m\hat C^{(-c+m)}[l_\mu].
$$
It is obvious that $
\hat C[n_\nu]\in\widehat{\mathfrak p}_\Theta[n_\nu]
$ and $
\hat C[0]=\hat C^{(-c)}[0]=\mathrm{Res}\hat C^{(-c)}$.
 We  rewrite
$$
\mathsf A(\tilde e)
=
z^{-c}
(
\sum_{n_\nu\geq0}\hat C[n_\nu]
)dz,
$$where $\hat C[0]=S+\sum\limits_iN_iz^i$.
For each $n_\nu$, decompose the finite-dimensional space
$\widehat{\mathfrak p}_\Theta[n_\nu]$ into
$\mathrm{ad}_S$-eigenspaces with eigenvalues $\rho$ :
$$
\widehat{\mathfrak p}_\Theta[n_\nu]
=
\bigoplus_{\rho}
V_{\rho, n_\nu}.
$$
Since $[S,\widehat C[0]]=0$, every eigenspace  $V_{\rho, n_\nu}$ is preserved by
$\mathrm{ad}_{\widehat C[0]}$.
More precisely, we have
$$
\mathrm{ad}_{\hat C[0]}|_{V_{\rho, n_\nu}}
=
\rho\mathrm{Id}
+
\mathrm{ad}_{\hat N}|_{V_{\rho, n_\nu}},
$$where $\hat N=\sum\limits_iN_iz^i$. In particular, $
\mathrm{ad}_{\hat C[0]}|_{V_{\rho, n_\nu}}$
is invertible whenever $\rho\neq0$.

We now proceed by induction on $n_\nu$.
The conclusion for the case $n_\nu=0$ has already been established.
Suppose that
$$
[S,\hat C[n_{\tau}]]=0
\qquad
\text{for all }n_\tau<n_\nu,
$$ then we need to show $$[S,\hat C[n_{\nu}]]=0.$$
Write
$$
\hat C[n_{\nu}]
=
\sum\limits_\rho \hat C_\rho[n_\nu],$$
where
$ \hat C_\rho[n_\nu]\in V_{\rho, n_\nu}$.
For every $\rho\neq0$, we introduce
$$
\hat U_{\rho, n_\nu}
=(\mathrm{ad}_{\hat C[0]}|_{V_{\rho, n_\nu}})^{-1} \hat C_\rho[n_\nu],
$$
and define
$$
\hat U_{n_\nu}
=
\sum_{\rho\neq0}\hat U_{\rho, n_\nu}.
$$
Since $n_\nu>0$, we have
$
\hat U_{\rho, n_\nu}\in\widehat{\mathfrak u}_\Theta
$,
and hence
$
\hat g_{n_\nu}=\exp(\hat U_{ n_\nu})\in\hat U_\Theta
$.
The gauge transformation by  $\hat g_{n_\nu}$ changes the degree-$n_\nu$
term into
$$
\hat C[n_{\nu}]+[\hat U_{n_\nu},\hat C[0]].
$$
For every $\rho\neq0$, projection onto $V_{\rho, n_\nu} $ gives term
$$
\hat C_\rho[n_\nu]
+
[\hat U_{\rho, n_\nu},\hat C[0]]$$
By definition, we have $$[\hat U_{\rho, n_\nu},\hat C[0]]=-\mathrm{ad}_{\hat C[0]}\hat U_{\rho, n_\nu}=-\hat C_\rho[n_\nu],$$ hence  the above terms vanishes exactly .
Consequently, after gauge transformation by  $\hat g_{n_\nu}$, the degree-$n_\nu$ term belongs entirely
to the zero eigenspace of $\mathrm{ad}_S$, namely we get the desired identity
$
[S,\hat C[n_{\nu}]]=0
$.

Repeating the above construction for
$
\nu=1,2,\ldots
$,
we obtain a product
$$\hat
g_+
=
\cdots \hat g_{n_2}\hat g_{n_1}
\in
\widehat U_\Theta
.$$
Let
$
e'=\hat g_+\tilde e
$,
then we have
$$
\mathsf A(e')
=
z^{-c}
(
\sum_{n_\nu\geq0}\hat{D}[n_\nu]
)dz,
$$
where
$
[S,\hat{D}[n_\nu]]=0$
for each $n_\mu$.
Rewrite this expression in the  $\Theta$-reduced representation as
$$
\mathsf A(e')
=
\sum_{r\geq-c}\hat B^{(r)}z^r\,dz,
$$ with
$$\hat B^{(r)}
=
\sum_{\lambda+i\geq0}
\sum_{i\in\mathbb Z}
X_{\lambda,i}^{(r)}z^i,
$$
then we have $\mathrm{Res}(\hat B^{(-c)})=\mathrm{Res}(\hat C^{(-c)})$
and $
[S,X_{\lambda,i}^{(r)}]=0
$
for all $r$, $\lambda$, and $i$.

We complete the proof.
\end{proof}
\begin{definition}A nonzero semisimple element $\hat{X}\in\widehat{\mathfrak{l}}_\Theta$ is  called regular in $\widehat{\mathfrak{p}}_\Theta$, if the complex dimension of the space $\mathrm{Ker}(\mathrm{ad}_{\hat{X}}: \widehat{\mathfrak{q}}_\Theta\rightarrow \widehat{\mathfrak{q}}_\Theta)$ is exactly the rank of $G$.
\end{definition}

The following corollary reformulates the regular semismple case \cite{Wo,Bo,h} of classical  Hukuhara-Levelt-Turrittin-Babbitt-Varadarajan reduction theory  under the parahoric context. 
\begin{corollary}\label{re}Let $\mathsf A$ be a formal connection (or a Higgs field) on a formal
parahoric principal $G$-bundle $(\mathbf P,\Theta\in \mathfrak{t}_\mathbb{R},e)$ with a
$(\Theta,e)$-reduced representation
$$
\mathsf A(e)
=
\sum_{r\geq -c}\hat A^{(r)}z^r dz,
$$ where $c>1$. Assume the semisimple part of $\mathrm{Res}\hat A^{(-c)}$ is a nonzero regular element in $\widehat{\mathfrak{p}}_\Theta$. Then there exist a  trivialization $e'\in[e]_\Theta$ and a nonzero
semisimple element $S\in\mathfrak t$ such that the
$(\Theta,e')$-reduced representation
$$
\mathsf A(e')
=
\sum_{r\geq-c} B^{(r)}z^r dz
$$
satisfies the following properties:
\begin{enumerate}
\item
$ B^{(-c)}=S$ is obtained from
$\mathrm{Res}\widehat A^{(-c)}$ by an $\widehat L_\Theta$-adjoint action;
\item
all
$B^{(r)}
$ lie in $\mathfrak{t}$.
\end{enumerate}
Furthermore, if $\mathsf A$ is a formal connection, the following property can be also  satisfied
\begin{enumerate}
    \item[(3)] $B^{(r)}=0$ whenever $r\geq0$.
  \end{enumerate}
\end{corollary}

\begin{proof}Denote the semisimple part of $\mathrm{Res}(\hat A^{(-c)})$ by $\hat{S}$, then there is $\hat h\in \widehat{L}_\Theta$ such that $\mathrm{Ad}_{\hat h}(\hat{S})=S\in \mathfrak{t}$, hence $\dim_\mathbb{C}\mathrm{Ker}(\mathrm{ad}_{\hat{S}})=\dim_\mathbb{C}\mathrm{Ker}(\mathrm{ad}_{{S}})$. By definition,
$$\widehat{\mathfrak{q}}_\Theta\simeq \mathfrak{t}\oplus\bigoplus_{\alpha\in\triangle}z^{\lceil-\alpha(\Theta)\rceil}\mathfrak{g}_\alpha,$$ where $\lceil-\alpha(\Theta)\rceil$ is the smallest integer greater than or equal to $-\alpha(\Theta)$. Therefore, we have
$$ \mathrm{Ker}(\mathrm{ad}_{S}: \widehat{\mathfrak{q}}_\Theta\rightarrow \widehat{\mathfrak{q}}_\Theta)\simeq \mathfrak{t}\oplus\bigoplus_{\alpha\in\triangle,\alpha(S)=0}z^{\lceil-\alpha(\Theta)\rceil}\mathfrak{g}_\alpha,$$
which implies that the complex dimension of $ \mathrm{Ker}(\mathrm{ad}_{S}: \widehat{\mathfrak{q}}_\Theta\rightarrow \widehat{\mathfrak{q}}_\Theta) $ coincides with that of the centralizer $\mathfrak{c}_\mathfrak{g}(S)$ of $S$ in $\mathfrak{g}$.
Now by assumption that $\hat{S}$ is regular in $\widehat{\mathfrak{p}}_\Theta$,  we have that $S$ is regular in $\mathfrak{g}$. Consequently,   all  $
X_{\lambda,i}^{(r)}
$ appearing in Theorem \ref{x} lie in $\mathfrak{t}$. This immediately implies
  $\lambda=0$ and $i\geq 0$. However, for  the reduced representation, we have $0\leq\lambda+i<1$, hence $i=0$. Namely, under a suitable trivialization $e'\in[e]_\Theta$, we have a reduced representation
  $$\mathsf A(e')
=
(Sz^{-c}+\sum_{r\geq-c+1}B^{(r)}z^r)dz$$ for $S, B^{(r)}\in \mathfrak{t}$. Consider the gauge transformation by
$\hat g=\exp(\hat{U})\in \widehat{U}_\theta$, where
$$\hat{U}=-\sum_{r\geq0}\frac{B^{(r)}}{r+1}z^{r+1}, $$
due to the identity
$$d\hat g\hat g^{-1}=-\sum_{r\geq0}B^{(r)}z^{r}dz,$$
then we have
$$\mathsf A(\hat ge')
=
(Sz^{-c}+\sum_{r=-c+1}^{-1}B^{(r)}z^r)dz.$$ We complete the proof.
\end{proof}

\begin{remark}
Bremer and Sage developed a reduction theory for toral formal connections  \cite{bs1,bs2,bs3}\footnote{We thank
the anonymous referee for bringing   Bremer-Sage's reduction theory for toral formal connections
 to our attention.} One says   a formal connection $\mathsf{A}$  on  a formal principal $G$-bundle $\mathbf{P}$ is called toral   if under a suitable trivialization it admits a regular stratum  $\widehat{G}$-stratum $(x,r,\beta)$ of positive depth, where  regularity condition means that  for every representative  $\beta$, its connected centralizer  in $\widehat G$ is a maximal torus compatible with $x$, and all these maximal tori are conjugate under the positive-depth Moy-Prasad subgroup $\widehat U_{x}$.
  In particular, if $x\in\mathcal{A}$ and $\beta$ is chosen to be the unique homogeneous representative denoted by $\tilde \beta$, then we only need to check the conditions that  $\tilde \beta$ is regular semisimple in $\widehat{\mathfrak{g}}$ (cf. \cite[Proposition 2.50]{az}).
   An important feature of toral formal connections  is that they can be diagonalized into the centralizer of $\beta$. More precisely,  there is $\hat g\in \widehat{U}_x$ such that $$\mathsf{A}(\hat g e)=\sum\limits_{-r\leq d\leq0, d\in \mathrm{Crit}_{MP}(\tilde\beta)}\hat{D}_d\frac{dz}{z},$$ where the summation is taken over  the set  $\mathrm{Crit}_{MP}(\tilde\beta)$ of  critical degrees of the centralizer  of $\tilde \beta$ in $\widehat{\mathfrak{g}}$
 at which the  induced Moy-Prasad filtration has  nontrivial graded quotients,  $\hat{D}_{-r}\frac{dz}{z}=\tilde \beta$,  and  $\hat{D}_d$ lies in the $d$'s graded piece
of the induced Moy-Prasad filtration.  (cf. \cite[Theorem 5.1]{bs3}). Obviously,  Proposition \ref{x} and Corollary \ref{re} have a similar flavor to the Bremer-Sage theory. We clarify their distinction and relationship.
\begin{itemize}
  \item  Theorem \ref{x} requires only a nonzero semisimple (not necessarily regular) part of the relevant Levi residue and reduces all coefficients to its possibly nonabelian reductive centralizer. In contrast, Bremer-Sage's toral reduction requires a regular stratum whose homogeneous leading term has a maximal-torus centralizer.
  \item The nonzero semisimple Levi-residue assumption in Theorem \ref{x} ensures the existence of a fundamental (not necessarily regular) $\widehat{G}$-stratum of integral depth, and hence the  formal connection has an integral slope. By contrast, Bremer-Sage's theory allows regular fundamental  $\widehat{G}$-strata of fractional depth, whose homogeneous representatives may have nonsplit maximal tori in $\widehat{G}$ as their connected centralizers. In terms of our reduced representations, such fractional slopes allowed in Bremer-Sage's theory arise from homogeneous leading terms of positive parahoric degree,  rather than from the degree-zero Levi component.
  \item Corollary \ref{re} is the common part of these two theories. Indeed, given  a formal connection $\mathsf A$ on $(\mathbf P,\Theta,e)$, due to  the regularity in $\widehat{\mathfrak{p}}_\Theta$ of the semisimple part of $\mathrm{Res}(\hat{A}^{(-c)})$,
  there is $\hat h\in \widehat{L}_\Theta$ such that $\mathrm{Ad}_{\hat h}(\mathrm{Res}(\hat{A}^{(-c)}))=S\in \mathfrak{t}$,
  then
  for the reduced representation $$\mathsf A(\hat h e)
=
(Sz^{-c}+\hat{Y}^{(-c)}z^{-c}+\sum_{r\geq -c+1}\hat{Z}^{(r)}z^{r})dz$$ with $\hat Y^{(-c)}\in \widehat{\mathfrak{u}}_\Theta, \hat{Z}^{(r)}\in \widehat{\mathfrak{p}}_\Theta$, the  homogeneous
representative $\tilde \beta$ is given by
$$\tilde \beta=Sz^{-c+1}\frac{dz}{z},$$ whose  connected centralizer in $\widehat{G}$
is the split maximal torus $T((z))$, hence we have a regular fundamental $\widehat{G}$-stratum $(\Theta, c-1, \tilde \beta)$. The induced  Moy-Prasad filtration indexed by $d\in \mathbb{R}$   on
on $\widehat{\mathfrak{t}}$ is given by
$$\widehat{\mathfrak{t}}_d=z^{\lceil d\rceil]}\mathfrak{t}[[z]],$$ thus the corresponding Moy-Prasad grading reads
$$\widehat{\mathfrak{t}}[d]=\left\{
                              \begin{array}{ll}
                                z^d\mathfrak{t}, & \hbox{$d\in\mathbb{Z}$;} \\
                                0, & \hbox{$d\notin\mathbb{Z}$.}
                              \end{array}
                            \right.
$$ Therefore, Bremer-Sage's theory also gives rise to the reduction in Corollary \ref{re}.
\end{itemize}

\end{remark}

\begin{lemma}\label{lem}Let $\mathsf A$ be a formal connection (or a Higgs field) on a formal
parahoric principal $G$-bundle $(\mathbf P,\Theta\in \mathfrak{t}_\mathbb{R},e)$ with the
$(\Theta,e)$-reduced representation
$$
\mathsf A(e)
=
\sum_{r\geq -c}\hat A^{(r)}z^r dz,$$ where  $\hat A^{(r)}=\sum\limits_{\lambda+i\geq 0}\sum\limits_{i\in\mathbb{Z}}A^{(r)}_{\lambda,i}z^i$.
Assume that $p_\mathfrak{t}(A^{(r)}_{\lambda,i})=0$ for every $r< -1$, where $p_\mathfrak{t}:\mathfrak{g}\rightarrow\mathfrak{t}$ denotes the Cardan projection. Then there exist a positive integer $M$ such that for any   form $\hat Q=\sum\limits_{j=-M}^{-2}Q^{(j)}z^jdz$ with $Q^{(j)}$ lying in  $ \mathfrak{t}$ and in particular,  $Q^{(-M)}$ being regular in $\mathfrak{g}$, one can find gauge transformation by $\hat u\in \widehat{U}_\Theta\bigcap G(\mathbb{A})$  which changes $\hat Q+\mathsf A(e)$ into
$\hat Q+\hat R$,
                       where $\hat R=\sum\limits_{r\geq -1}R^{(r)}z^rdz$ with each  $R^{(r)}\in \mathfrak{t}$,
       \end{lemma}
\begin{proof}Let $p$ be the smallest integer that is bigger than or equal to 1 such that $\sum\limits_{r\geq -c}\hat A^{(r)}z^r \in z^{-p}\mathfrak{g}[[z]]$. We choose a positive integer $M\geq \max\{2, 2p-1\}$ satisfying $M>p+\zeta_\Theta$, where $\zeta_\Theta=\max\{|\lambda|:\mathfrak{\mathfrak{g}}_\lambda^{(\Theta)}\neq\{0\}\}$,  hence we have $\delta\geq1$  and $\delta>\zeta_\Theta$ for $\delta=M-p$.
Since $z^M\mathsf A(e)\in z^\delta\mathfrak{g}[[z]]dz$, writing  $$\mathsf B_{\delta}(e):=z^M\frac{\hat Q+\mathsf A(e)
}{dz}=Q^{(-M)}+\sum\limits_{j\geq 1}B_{\delta}^{(j)}z^j,$$ we have $B_\delta^{(j)}\in \mathfrak{t}$ for $j=1,\cdots \delta-1$.  Suppose   $$\mathsf B_n (e)=Q^{(-M)}+\sum\limits_{j\geq 1}B_n^{(j)}z^j\  (n\geq \delta)$$ satisfies  $B_n^{(j)}\in\mathfrak{t}$ for each $j<n$. Decompose $B^{(n)}_n=H^{(n)}_n+\sum\limits_{\alpha\in \triangle}Y^{(n)}_{n,\alpha}$ with $H^{(n)}_n\in \mathfrak{t}$ and $Y^{(n)}_{n,\alpha}\in \mathfrak{g}_\alpha$. Since $Q^{(-M)}$ is  regular semisimple in $\mathfrak{g}$, $\alpha(Q^{(-M)})$ is nonzero for all  $\alpha\in \triangle$, then  the  gauge transformation  by $\hat u_{n} = \exp(z^{n}X_n)$ with $$X_n=\sum_{\alpha\in\triangle}\frac{Y^{(n)}_{n,\alpha}}{\alpha(Q^{(-M)})}$$ changes $\mathsf B_n (e)$ into
\begin{align*}
 \mathsf B_{n+1} (e)&= \mathsf B_n (e)+z^n[X_n,Q^{(-M)}]+O(z^{n+1})\\
 &=Q^{(-M)}+\sum\limits_{j=1}^{n-1}B_n^{(j)}z^j+z^nH^{(n)}_n+O(z^{n+1}).
\end{align*}
Note that $$n+\alpha(\Theta)\geq \delta+\alpha(\Theta)\geq \delta-|\alpha(\Theta)|\geq \delta-\zeta_\Theta>0,$$ we have $\hat u_n\in \widehat{U}_\Theta\bigcap G(\mathbb{A})$. Repeating this construction for $n=\delta, \delta+1,\cdots$, we get a gauge transformation by the product  $\hat u=\cdots \hat u_{\delta+1}\hat u_\delta\in \widehat{U}_\Theta\bigcap G(\mathbb{A})$ which changes $\hat Q+\mathsf A(e)$ into an element lying in $\widehat{\mathfrak{t}}dz$. Writing  $\hat u=\exp(\hat U)$ with $\hat U\in z^\delta\mathfrak{g}[[z]]$, we  consider the Cartan projection of each summand on the right hand side of the identity
$$\mathrm{Ad}_{\hat u}(\hat Q+ \mathsf A(e))=\hat Q+ \mathsf A(e)+[\hat U,\hat Q]+[\hat U,\mathsf A(e)]+\frac{1}{2}[\hat U,[\hat U,\hat Q]]+\frac{1}{2}[\hat U,[\hat U,\mathsf A(e)]]+\cdots,$$
  where the Cartan projection is defined on the coefficients of series.  By assumption the Cartan projection of $\mathsf A(e)$ lies in $\mathfrak{t}[[z]]\frac{dz}{z}$. Moreover, we have  $p_\mathfrak{t}([\hat U,\hat Q])=0$,  $\mathrm{ad}^m_{\hat U}\hat Q\in z^{m\delta-M}\mathfrak{g}[[z]]dz\subset z^{2\delta-M}\mathfrak{g}[[z]]dz =z^{M-2p}\mathfrak{g}[[z]]dz\subset\mathfrak{g}[[z]]\frac{dz}{z} $ for $m\geq 2$, $\mathrm{ad}^m_{\hat U} \mathsf A(e)\in z^{m\delta-p}\mathfrak{g}[[z]]dz\subset z^{\delta-p}\mathfrak{g}[[z]]dz =z^{M-2p}\mathfrak{g}[[z]]dz\subset\mathfrak{g}[[z]]\frac{dz}{z} $ for $m\geq 1$, and $d\hat u\hat u^{-1}\in z^\delta\mathfrak{g}[[z]]\frac{dz}{z}$. Therefore, after gauge transformation by $\hat u$, $\hat Q+ \mathsf A(e)$ is changed into the form $\hat Q+\hat R$,
                       where $\hat R=\sum\limits_{r\geq -1}R^{(r)}z^r dz$ with each  $R^{(r)}\in \mathfrak{t}$.
\end{proof}

\begin{definition}\label{rrd}Let $ \mathsf{A}$ be  a formal connection (or a Higgs field) on a formal  principal $G$-bundle $\mathbf{P}$. We call $ \mathsf{A}$  relatively regular if $\mathbf{P}$ is endowed with a $\Theta$-parahoric structure $[e]_\Theta$ ($\Theta\in \mathfrak{t}_\mathbb{R}$)  such that the $(\Theta,e)$-reduced representation
\begin{align*}
 \mathsf{A}(e)=\sum\limits_{r\geq -c}\hat A^{(r)}z^{r}dz
\end{align*} with $\hat A^{(r)}=\sum\limits_{\lambda+i\geq 0}\sum\limits_{i\in\mathbb{Z}}A^{(r)}_{\lambda,i}z^i$ satisfies
that  $A^{(r)}_{0,0}$ is a nilpotent element in $\mathfrak{g}$ whenever  $r<-1$.

\end{definition}

\begin{theorem}\label{rl}Let $ \mathsf{A}$ be a formal connection  on a  principal $G$-bundle $\mathbf{P}$, then the following are equivalent.
\begin{enumerate}
  \item $ \mathsf{A}$ is  relatively regular.
  \item The induced formal $G_{\mathrm{ab}}$-connection  is regular, where $G_{\mathrm{ab}}=G/[G,G]$.
  \item $\mathbf{P}$ is endowed with a $\Theta$-parahoric structure $[e]_\Theta$ ($\Theta\in \mathfrak{t}_\mathbb{R}$) such that there are a form $\hat{Q}=\sum\limits_{r=-c\leq-2}^{-2}Q^{(r)}z^r dz$ with each  $Q^{(r)}\in \mathfrak{t}$  regular semisimple in $\mathfrak{g}$  and $\hat u\in \widehat{U}_\Theta$ satisfying that the gauge transformation by $\hat u$ changes $\hat Q+\mathsf A(e)$ into
$\hat Q+\hat R$,
                       where $\hat R=\sum\limits_{r\geq -1}R^{(r)}z^rdz$ with each  $R^{(r)}\in \mathfrak{t}$.
\end{enumerate}
\end{theorem}

\begin{proof}(1) $ \Rightarrow$ (2):
Let $\pi: G\rightarrow G_{\mathrm{ab}}$ be the natural quotient map, which induces the Lie algebra homomorphism $\pi_*: \mathfrak{g}\rightarrow \mathfrak{g}/[\mathfrak{g},\mathfrak{g}]\simeq\mathfrak{z}_\mathfrak{g}$. For a relatively regular formal connection $ \mathsf{A}$,  in terms of the $(\Theta,e)$-reduced representation,   we have $\pi_*(A^{(r)}_{\lambda,i})=0$ whenever $\lambda\neq 0$ and $\pi_*(A^{(r)}_{0,0})=0$ whenever $r<-1$. Therefore, $\pi_*(\mathsf{A}(e))\in \mathfrak{z}_\mathfrak{g}[[z]]\frac{dz}{z}$.

(2) $ \Rightarrow$ (3): By Frenkel-Zhu's  oper existence theorem   \cite{fz}, for any formal connection $\mathsf{A}$,  fixing a maximal torus $T$ of $ G$ and a Borel subgroup $B\subset G$ which contains $T$,  there is a trivialization $e'$  such that
       \begin{align*}
         \mathsf{A}(e')=(\sum_{i=1}^\ell \hat f_i F_i+\hat H+\hat N)dz,
       \end{align*}where $\ell$ is the rank of $[G,G]$, $F_i$ is a basis of $\mathfrak{g}_{\alpha_i}$ for $\alpha_i$ being a negative simple root, $\hat f_i\in \mathbb{K}^\times$, $\hat H\in \widehat{\mathfrak{t}}$ and $\hat N\in \widehat{\mathfrak{n}_+}$ for $ \mathfrak{n}_+=\bigoplus\limits_{\alpha\in \triangle^+(B)}\mathfrak{g}_\alpha$.
       Let $E_i$ be the  simple root vector corresponding to $F_i$ via the commutator $[E_i,F_i]=-\alpha_i^\vee=:H_i\in \mathfrak{t}$, then according to the decomposition $\mathfrak{t}=\mathfrak{z}_\mathfrak{g}\bigoplus(\mathfrak{t}\bigcap [\mathfrak{g},\mathfrak{g}])$,  we write
       $$\hat H=\hat H_0+\sum_{i=1}^\ell \hat h_i H_i,$$ where $\hat H_0\in \widehat{\mathfrak{z}_\mathfrak{g}}$ and $\hat h_i\in \mathbb{K}$. Since the condition (2) holds, we have $\hat H_0\in z^{-1}\mathfrak{z}_\mathfrak{g}[[z]]$.
       Consider the gauge transformation
$$ \mathsf A(\hat ve')=d\hat v\hat v^{-1}+\mathsf A(e')+[\hat V,\mathsf A(e')]+\frac{1}{2}[\hat V,[\hat V,\mathsf A(e')]]+\cdots,$$
       by
       $\hat v= \exp(\hat V)$, where
       $$\hat V=-\sum_{i=1}^\ell\frac{\hat h_i}{\hat f_i}E_i\in \widehat{\mathfrak{n}_+}.$$
       Since $[\hat V,\sum\limits_{i=1}^\ell \hat f_i F_i]=-\sum\limits_{i=1}^\ell \hat h_i H_i$, $\mathrm{ad}_{\hat V}^m(F_i)\in \widehat{\mathfrak{n}_+}$ for $m\geq2$,
       $\mathrm{ad}_{\hat V}^m(\hat H)\in \widehat{\mathfrak{n}_+}$ for $m\geq 1$, $\mathrm{ad}_{\hat V}^m(\hat N)\in \widehat{\mathfrak{n}_+}$ for $m\geq1$ and $d\hat v\hat v^{-1}\in\widehat{\mathfrak{n}_+}dz $,
      we have  $$p_\mathfrak{t}(\mathsf{A}(\hat ve'))=\hat H_0dz\in \mathfrak{t}[[z]]\frac{dz}{z}.$$ Choosing a weight $\Theta\in \mathfrak{t}_{\mathbb{R}}$, the $\Theta$-parahoric structure $\mathbf{P}$ is defined by  $[e]_\Theta$ for $e=\hat ve'$. Then applying Lemma \ref{lem} leads to the conclusion.

    (3) $ \Rightarrow$ (1):
       For  a formal connection $ \mathsf{A}$, assume the condition (3) holds, then it is clear that the induced formal $G_{\mathrm{ab}}$-connection  is regular. Thus, it suffices to show $(2)$ implies (1). From the proof of "(2) $ \Rightarrow$ (3)", we have known that the Cartan projection of $\mathsf{A}(e)$ lies in $\mathfrak{t}[[z]]\frac{dz}{z}$. Then, choosing a regular weight $\Theta\in \mathfrak{t}_\mathbb{R}$, namely $\alpha(\Theta)\neq 0$ for every $\alpha\in \triangle$, in the $(\Theta,e)$-reduced representation, $A^{(r)}_{0,0}\in \mathfrak{t}$ must vanish for every $r<-1$.
\end{proof}

\begin{corollary}Assume $G$ is a connected complex
semisimple  Lie group.  Let $ \mathsf{A}$ be a formal connection  on a  principal $G$-bundle $\mathbf{P}$, then
  \begin{enumerate}
    \item $\mathbf{P}$ is endowed with a $\Theta$-parahoric structure $[e]_\Theta$ ($\Theta\in \mathfrak{t}_\mathbb{R}$)  such that the $(\Theta,e)$-reduced representation
$
 \mathsf{A}(e)=\sum\limits_{r\geq -c}\hat A^{(r)}z^{r}dz$
 with $\hat A^{(r)}=\sum\limits_{\lambda+i\geq 0}\sum\limits_{i\in\mathbb{Z}}A^{(r)}_{\lambda,i}z^i$ satisfies
that  $A^{(r)}_{0,0}$ is a nilpotent element in $\mathfrak{g}$ whenever  $r<-1$;
    \item $\mathbf{P}$ is endowed with a $\Theta$-parahoric structure $[e]_\Theta$ ($\Theta\in \mathfrak{t}_\mathbb{R}$) such that there are a form $\hat{Q}=\sum\limits_{r=-c\leq-2}^{-2}Q^{(r)}z^r dz$ with each  $Q^{(r)}\in \mathfrak{t}$  regular semisimple in $\mathfrak{g}$  and $\hat u\in \widehat{U}_\Theta$ satisfying that the gauge transformation by $\hat u$ changes $\hat Q+\mathsf A(e)$ into
$\hat Q+\hat R$,
                       where $\hat R=\sum\limits_{r\geq -1}R^{(r)}z^r dz$ with each  $R^{(r)}\in \mathfrak{t}$.
  \end{enumerate}

\end{corollary}
\subsection{Nilpotent Leading Coefficient}

\begin{theorem}\label{xx}
Let $ \mathsf{A}$ be a formal connection (or Higgs field) on a   formal parahoric principal $G$-bundle $(\mathbf{P}, \Theta\in\mathfrak{t}_\mathbb{R},e)$  with the  $(\Theta,e)$-reduced representation
\begin{align*}
 \mathsf{A}(e)=\sum\limits_{r\geq -c}\hat A^{(r)}z^{r}dz,
\end{align*}   where $c>1$.  Fixing an integer $k$ ($2\leq k\leq c$), assume
\begin{itemize}
  \item $\hat A^{(-l)}$ lies in the center $\mathfrak{z}_\mathfrak{g}$ of  $\mathfrak{g}$ for any $l>k$;
  \item  $\mathrm{Res}\hat A^{(-k)}$ is a nonzero nilpotent element in  $\widehat{\mathfrak{l}}_\Theta$.
\end{itemize}
Then there exist a trivialization
$e'\in [e]_\Theta$ and a Jacobson-Morozov $\mathfrak{sl}(2,\mathbb{C})$-triple $(\hat P, \hat Q,  H)$ in $\widehat{\mathfrak{l}}_\Theta$ with $H\in \mathfrak{t}$ and $\hat P$ being obtained from
$\mathrm{Res}(\hat A^{(-k)})$ by an $\widehat L_\Theta$-adjoint action such that the $(\Theta,e')$-reduced representation
\begin{align*}
 \mathsf{A}(e')=\sum\limits_{r\geq -c}\hat B^{(r)}z^{r}dz
\end{align*}
satisfies the following properties:
  \begin{enumerate}
      \item $\hat B^{(-l)}=\hat A^{(-l)}$ for $l>k$;
      \item $\mathrm{Res}(\hat B^{(-k)})=\hat P$;
      \item   writing $\mathsf{A}(e')=(\sum\limits_{l=k+1}^{c}\hat B^{(-l)}z^{-l}+z^{-k}(\hat P+\sum\limits_{n_\nu>0}\hat B[n_\nu]))dz$ with $\hat B[n_\nu]\in\widehat{\mathfrak{p}}_\Theta[n_\nu]$, we have $[\hat Q,\hat B[n_\nu]]=0$ for each $n_\nu>0$.
    \end{enumerate}
\end{theorem}

\begin{proof} Without loss of generality, we may assume $k=c$. For the nonzero nilpotent element
$\hat R:=\mathrm{Res}\widehat A^{(-c)}
\in\widehat{\mathfrak l}_{\Theta}$,
we have  the Jacobson--Morozov
$\mathfrak{sl}(2,\mathbb C)$-triple
$
(\hat R,\hat Q_0,\hat H_0)
$ in $\widehat{\mathfrak l}_{\Theta}
$ with semisimple $\widehat H_0$.
There exists
$\hat h\in\widehat L_{\Theta}$ such that $H=\mathrm{Ad}_{\hat h}\hat H_0$ lies in $\mathfrak{t}$ and  setting
$(\hat P, \hat Q, H)$ with $\hat P=\mathrm{Ad}_{\hat h}\hat R$ and $\hat Q=\mathrm{Ad}_{\hat h}\hat Q_0$ is again a
Jacobson--Morozov $\mathfrak{sl}(2,\mathbb C)$-triple.
Write
$$
\mathsf A(\hat h e)
=
\sum_{r\geq-c}\hat C^{(r)}z^r dz.
$$
 By Proposition \ref{1},  we have
$$
\mathrm{Res}(\hat C^{(-c)})
=
\mathrm{Ad}_{\hat h}
(\mathrm{Res}(\widehat A^{(-c)}))
=
\hat P.
$$

Following
the proof of Theorem \ref{x},  we rewrite
$$
\mathsf A(\hat h e)
=
z^{-c}
(
\widehat P+
\sum_{n_\nu>0}\widehat C[n_\nu]
)dz.
$$
Since
the finite-dimensional space
$\widehat{\mathfrak p}_{\Theta}[n_\nu]$
is preserved by the adjoint actions of
$\hat P, \hat Q, \hat H$, it
 is a finite-dimensional
$\mathfrak{sl}(2,\mathbb C)$-module.
Then we have
$$
\widehat{\mathfrak p}_{\Theta}[n_\nu]
=
\mathrm{Im}
(
\mathrm{ad}\hat P
)
\bigoplus
\mathrm{Ker}
(
\mathrm{ad}\widehat Q
).
$$
We now proceed by induction on $n_\nu$.
Suppose that
$
[\hat Q,\hat C[n_\tau]]=0$
for all $n_\tau<n_\nu$.
There exists
$
\hat U_{n_\nu}
\in\widehat{\mathfrak p}_{\Theta}[n_\nu]\in \widehat{\mathfrak u}_{\Theta}
$
such that
$$
\hat C[n_\nu]
+
[\hat U_{n_\nu},\hat P]
\in
\mathrm{Ker}(\mathrm{ad}\hat Q).
$$
After the gauge transformation by
$
\hat g_{n_{\nu}}
=
\exp(\hat U_{n_\nu})
$,
 the
degree-$n_\nu$ component  is precisely
$
\hat C[n_\nu]
+
[\hat U_{n_\nu},\hat P]$.
Repeating this construction for
$\nu=1,2,\ldots$, we obtain a product
$$
\hat g_+
=
\cdots
\hat g_{n_2}
\hat g_{n_1}
\in\widehat U_{\Theta}.
$$
Then rewriting the $(\Theta,e')$-reduced representation for $e'=\hat g_+\hat he$ as
$$
\mathsf A(e')
=
z^{-c}
(
\hat P+
\sum_{\nu\geq1}\hat B[n_\nu]
)dz,
$$
we have
$$
[\hat Q,\hat B[n_\nu]]=0
$$ for any $n_\nu>0$.
Moreover,  the gauge transformation by $\hat g_+$
 preserves the  Levi residue, hence
$
\mathrm{Res}\hat B^{(-c)}
=
\hat P
$.

The proof is complete.
\end{proof}

The following theorem is the paraboric version of classical  Hukuhara-Levelt-Turrittin-Babbitt-Varadarajan reduction theory.
\begin{theorem}\label{rd}Let $\mathsf A$ be a formal connection on a formal parahoric principal
$G$-bundle $(P,\Theta\in\mathfrak t_{\mathbb R},e)$.
 Then there exist a positive integer $b$, a finite cover
$
        z=\zeta^b,
$
and a trivialization $e'\in T_{\mathbf{P}^{\sharp_b}}$  such that
$$
\mathsf A^{\sharp_b}(e')
 =
 (
   \sum_{j=-c}^{-2} D_j\zeta^j
   +\hat C\,\zeta^{-1}
 )d\zeta,
 $$
where
$
        D_j\in\mathfrak t
       $, $
        \hat C\in\mathfrak{b}[\zeta]\bigcap\widehat{\mathfrak p}_{b\Theta}
$ for some Borel subalgebra $\mathfrak{b}\subset \mathfrak{g}$ containing $\mathfrak{t}$,
and
$
          [D_j,\hat C]=0
$
for all $j$.
\end{theorem}

\begin{proof}  By the classical reduction theorem,
after passing to a  finite  cover $z=\zeta^b$,
we may write
$$
        \mathsf A^{\sharp_b}(\tilde e)
        =
        (
          \sum_{j=-c}^{-2}D_j\zeta^j
          +C\zeta^{-1}
        )d\zeta
$$ under a suitable trivialization $\tilde e\in T_{\mathbf{P}^{\sharp_b}}$,
where
$ D_j\in\mathfrak t,\
        [D_j,C]=0$
       for all $j$.
Let
$
        L
        =
        C_G(D_{-2},\cdots D_{-c})
$
be the common centralizer of $D_{-2},\cdots, D_{-c}$, which
is a connected reductive subgroup containing $T$, and let $\mathfrak{l}$ be the Lie algebra of $L$.
Write the Jordan decomposition of $C$ in
$\mathfrak l$ as
$
        C=S+N
$ with semisimple
$S$ and nilpotent  $N$.
Since $S$ and $N$ are both polynomial expressions in
$C$, we have
$$
        [D_j,S]=0,
        \
        [D_j,N]=0$$
       for every $j$. After  a suitable conjugation by an element in  $L$,
  we may assume that
$     S\in\mathfrak t$ and
$N =\sum\limits_{\alpha\in\triangle^+(B)}N_\alpha
        $ (cf. the proof of Theorem \ref{ddd}), where   $B\subset G$ is a certain Borel subgroup containing $T$ and
        $N_\alpha\in\mathfrak g_\alpha$.
For every root $\alpha$ such that $N_\alpha\neq0$, we have
$
        \alpha(D_j)= \alpha(S)=0$.

We can
choose a cocharacter
$
        \Xi\in (E_T)_*
$
such that, for every root $\alpha$ satisfying  $N_\alpha\neq0$, we have
$$
        \alpha(\Xi)\geq0,\  \alpha(\Xi)+b\alpha(\Theta)\geq0.
$$
Now apply the gauge transformation by
$
        g=\zeta^\Xi
$, we obtain
$$ \mathsf A^{\sharp_b}(g\tilde e)=
        \sum_{j=-c}^{-2}D_j\zeta^j
d\zeta
+
\hat C\frac{d\zeta}{\zeta},
$$
where
$$
        \hat C
        =
        S+\Xi+
        \sum_{\alpha\in\triangle^+(B)}
        \zeta^{\alpha(\Xi)}N_\alpha.
$$
The choice of $\Xi$ ensures that $
      \hat C \in
        \mathfrak b[\zeta]
        \cap
        \widehat{\mathfrak p}_{b\Theta}$. Finally, it follows from $[D_j,N_\alpha]=0$ that $[D_j, \hat C]=0$ for all $j$.
\end{proof}
\begin{corollary}\label{cor}Assume the weight $\Theta\in \mathfrak{t}_\mathbb{R}$ is rational, namely
$ \alpha(\Theta)\in\mathbb Q
       $ for every $\alpha\in \triangle$.
Then there exist a positive integer $b$, a finite cover
$
        z=\zeta^b,
$
and a trivialization $e'\in T_{\mathbf{P}^{\sharp_b}}$  such that
$$
\mathsf A^{\sharp_b}(e')
 =
 (
   \sum_{j=-c}^{-2} D_j\zeta^j
   +\hat C\,\zeta^{-1}
 )d\zeta,
 $$
where
$
        D_j\in\mathfrak t
       $,  $
        \hat C\in\widehat{\mathfrak l}_{b\Theta}
$,
and
$
          [D_j,\hat C]=0
$
for all $j$.
\end{corollary}

\begin{proof}By the above Theorem, after passing to a finite cover
$z=\eta^{b_0}$, we may write
$$
\mathsf A^{\sharp_{b_0}}(\tilde e)
=
\sum_{j=-c}^{-2}D_j\eta^j\,d\eta
+
(
S+\Xi+\sum_{\alpha\in \triangle^+(B)}\eta^{\alpha(\Xi)}N_\alpha
)\frac{d\eta}{\eta},
$$
where $D_j,S,\Xi\in\mathfrak t$, and
$N_\alpha\in\mathfrak g_\alpha$.
Since $\Theta$ is rational,  there exists a cocharacter
$
\Lambda\in (E_T)_*
$
such that for every $\alpha\in \triangle$ we have
$
\alpha(\Lambda)=qb_0\alpha(\Theta)$  with some positive integer $q$.
Put $b=qb_0$, and pass to the further cover
$
\eta=\zeta^q
$.
Then after  the gauge transformation by $
\hat g=\zeta^{-q\Xi-\Lambda}
$,  we have
$$\mathsf A^{\sharp_b}(\hat g \tilde e)
=
\sum_{j=-c}^{-2}
qD_j\zeta^{qj+q-1}\,d\zeta
+
\widehat C\,\frac{d\zeta}{\zeta},$$
where
$$
\widehat C
=qS-\Lambda
+
q\sum_{\alpha\in\triangle^+(B)}
\zeta^{-b\alpha(\Theta)}N_\alpha.
$$
It is clear that $\widehat C\in\widehat{\mathfrak l}_{b\Theta}$.
\end{proof}
\begin{remark}For the rational weight, we can also  apply the Babbitt-Varadarajan  algorithm (cf. \cite[Section 6]{bv} and \cite[Section 4]{h}) to deduce the reduction.
 Without loss of generality, we assume $\alpha(\Theta)\in \mathbb{Z}$ for any $\alpha\in \triangle$, then we write
$$
       \mathsf A(e)
        =
        z^{-c}(
         \hat X_0+\sum_{d\geq1}\hat X_d
        )dz,
        $$ where
        $\hat X_0\in\widehat{\mathfrak l}_{\Theta},
        X_d\in \widehat{\mathfrak p}_{\Theta}[d]=
        z^d\widehat{\mathfrak l}_{\Theta}
$. Suppose $c>1$.
Write
$
        \hat X_0=\hat S_0+\hat N_0
$
for its Jordan decomposition in
$\widehat{\mathfrak l}_{\Theta}$ with semisimple $\hat S_0$ and nilpotent $\hat N_0$.
Assume first that $\hat S_0\neq0$.  By Theorem \ref{x},  after a $\widehat{P}_\Theta$-gauge transformation,  we may
assume
$$
        \hat S_0=S\in\mathfrak t,\
        [S,\mathsf A(e)]=0.
$$
Hence all coefficients take values in the centralizer  $\mathfrak{c}_{\widehat{\mathfrak{l}}_\Theta}(S)$. If $S$ is not central, we have $\dim_\mathbb{C}([\mathfrak{c}_{\widehat{\mathfrak{l}}_\Theta}(S),\mathfrak{c}_{\widehat{\mathfrak{l}}_\Theta}(S)])<\dim_\mathbb{C}([\widehat{\mathfrak{l}}_\Theta,\widehat{\mathfrak{l}}_\Theta])$.
Next we assume  $\hat S_0=0$, by Theorem \ref{xx},
 after a $\widehat{P}_\Theta$-gauge transformation, we can write
$$
       \mathsf A(e)
        =
        z^{-c}
        (
         \hat  P+\sum_{d\geq1}z^d \hat Z_d
       )dz,$$ where
        $
       \hat Z_d\in\widehat{\mathfrak l}_{\Theta}$,
        $
        [\hat Q,\hat Z_d]=0$ for a $\mathfrak{sl}(2,\mathbb{C})$-triple $(\hat P, \hat Q,  H\in\mathfrak{t})$ in $\widehat{\mathfrak l}_{\Theta}$.
Decompose
$$
      \hat   Z_d=\sum_\lambda \hat Z_{d,\lambda},
       $$
     where $   [H,\hat Z_{d,\lambda}]=\lambda\hat Z_{d,\lambda}$.
Since $\hat Z_d\in\ker(\mathrm{ad} \hat Q)$, the $\mathfrak{sl}_2$-module structure gives
$
        \lambda\in\mathbb Z^{\geq0}
$
whenever $\hat Z_{d,\lambda}\neq0$.  Put
$$
        \Upsilon
        =
        \min_{\hat Z_{d,\lambda}\neq0}
        \frac{d}{\lambda/2+1},
$$
with $\Upsilon=\infty$ if all $Z_d$ vanish.
Suppose first that $\Upsilon\geq c-1$.  Take the double cover
$z=\zeta^2$ and the gauge transformation by $\hat g=\zeta^{-(c-1)H}$, then we calculate
$$\mathsf A^{\sharp_2}(\hat ge^{\sharp_2})=2\hat P^{\sharp_2}\frac{d\zeta}{\zeta}+2\sum_{d,\lambda}\hat Z_{d,\lambda}^{\sharp_2}\zeta^{q(d,\lambda)}d\zeta-(c-1)H\frac{d\zeta}{\zeta},$$
where $q(d,\lambda)=-2c+1+2d-(c-1)\lambda\geq -1$.
Assume now that
$
        0<\Upsilon<c-1
$.
Choose an integer  $\delta>0$ such that $\delta\Upsilon\in\mathbb Z$, and put
$
        b=2\delta,
        n=-\delta\Upsilon
$.
After the cover $z=\zeta^b$, apply the gauge transformation by $\hat g=\zeta^{nH}$, then
we have
$$\mathsf A^{\sharp_b}(\hat ge^{\sharp_b})=\zeta^{-c'}(b\hat P^{\sharp_b}+b\sum_{d,\lambda, s(d,\lambda)=0}\hat Z_{d,\lambda}^{\sharp_b}+b\sum_{d,\lambda, s(d,\lambda)>0}\hat Z_{d,\lambda}^{\sharp_b}\zeta^{s(d,\lambda)}+nH\zeta^{c'-1})d\zeta,$$
where $c'=bc-b+1+2n>1$ and $s(d,\lambda)=bd+n(\lambda+2)\geq 0$. Define
$$\hat P'=b\hat P^{\sharp_b}+b\sum_{d,\lambda, s(d,\lambda)=0}\hat Z_{d,\lambda}^{\sharp_b}\in \widehat{\mathfrak{l}}_{b\Theta}.$$
It is clear that $\hat P'\neq b\hat P^{\sharp_b}$ and  $\hat P'\in b\hat P^{\sharp_b}+\Ker (\mathrm{ad}_{\hat Q^{\sharp_b}})$. If $\hat P'$ has nonzero
semisimple part, we return to the semisimple branch. If $\hat P'$ is nilpotent, we have $\dim_\mathbb{C}\mathfrak{c}_{ \widehat{\mathfrak{l}}_{b\Theta}}(\hat P')<\dim_\mathbb{C}\mathfrak{c}_{ \widehat{\mathfrak{l}}_{b\Theta}}(b\hat P^{\sharp_b})$.
 Since there is no infinite decreasing chain of reductive
centralizers and a complex reductive Lie algebra has only finitely
many nilpotent orbits, the recursion procedure described above  terminates after finitely many
steps.\end{remark}

As an application of Theorem  \ref{xx},  we generalize  Frenkel-Zhu's Borel reduction theorem of formal connections \cite{fz,ar} to the  parahoric context.

The following proposition is the parahoric version of  \cite[Lemma 3]{fz}.
\begin{proposition}\label{513}
Let $\mathsf{A}$ be a formal connection (or a Higgs field) on a formal
parahoric principal $G$-bundle $(P,\Theta,e)$ with the
$(\Theta,e)$-reduced representation
$$
    \mathsf{A}(e)=\sum_{r\geq -c}\hat A^{(r)}z^rdz,
   $$
where $c>1$. Assume that
$
   \mathrm{ Res}(\hat A^{(-c)})
$
is  a nilpotent element of $\widehat{\mathfrak l}_\Theta$.
Then, for every
$
    \hat g\in M_\Theta(A,e;-c)
$,
 $\chi(\widehat g)
$
is also  a nilpotent element of $\widehat{\mathfrak l}_\Theta$.
\end{proposition}

\begin{proof}
For  $\hat g\in M_\Theta(A,e;-c)$, assume the leading index of $(\Theta, \hat g^{-1}e)$-reduced representation is $c$. It is known that there is the  Cartan decomposition\footnote{In some literature, this decomposition is also called Bruhat decomposition \cite{la} or Birkhoff decomposition \cite{fz}.}
\begin{align*}
  \widehat{G}=\coprod_{\hat w\in\widehat{W}_\Theta\backslash\widehat{W}/\widehat{W}_\Theta }\widehat{P}_\Theta \hat w\widehat{P}_\Theta,
\end{align*}
where $\widehat{W}_\Theta=(N_T(\mathbb{K})\bigcap\widehat{P}_\Theta)/T(\mathbb{A})$ (cf. \cite[Proposition 8.17]{la}, \cite[Proposition 8]{hr}), also cf. \cite{hb,y}).
According this decomposition, we write $\hat g^{-1}
    =
    \hat p_1 w\hat p_2$ for
$\hat p_1,\hat p_2\in\widehat P_\Theta, \hat w\in \widehat{W}$. Let $w=\pi(\hat w)$, where $\pi: \widehat{W}\rightarrow W$ be the natural projection, then
we can write   $\hat w=wz^\Xi$ for  some $\Xi\in (E_T)_*$. Write $\mathsf{A}( \widehat p_2e)  =z^{-c}\hat C dz$, then we have
\begin{align*}
   \mathsf{A}( w\widehat p_2e)
       =z^{-c}\hat D dz=z^{-c}(\mathrm{Ad}_{\hat w}(\hat C)+z^{c-1}\mathrm{Ad}_{ w}(\Xi))dz,
\end{align*}
hence  $$\mathrm{Res}(\hat D)=\mathrm{Res}(\mathrm{Ad}_{\hat w}(\hat C)).$$ It remains  show   that $\mathrm{Res}(\hat D)$ is nilpotent.

One defines
$$
    \widehat{\mathfrak h}_{\hat w}
    =
    \widehat{\mathfrak p}_\Theta
    \bigcap
    \mathrm{Ad}_{\hat w}^{-1}\widehat{\mathfrak p}_\Theta,\  \widehat{\mathfrak h}'_{\hat w}=\mathrm{Ad}_{\hat w}\widehat{\mathfrak h}_{\hat w}=
    \widehat{\mathfrak p}_\Theta
    \bigcap
    \mathrm{Ad}_{\hat w}\widehat{\mathfrak p}_\Theta.
$$
Obviously, $\hat C\in \widehat{\mathfrak h}_{\hat w} $ and $\hat D\in  \widehat{\mathfrak h}'_{\hat w}$.
Then one  defines
$$
    \widehat{\mathfrak q}_{\hat w}
    =
   \mathrm{ Res}(\widehat{\mathfrak h}_{\hat w}),
    \
   {\widehat {\mathfrak q}}'_{\hat w}
    =
   \mathrm{ Res}(
       \widehat{\mathfrak h}'_{\hat w}
   ).
$$
We claim that $\widehat{\mathfrak q}_{\hat w}$ and $ {\widehat {\mathfrak q}}'_{\hat w}$  are both parabolic subalgebras of
$\widehat{\mathfrak l}_\Theta$. Consider $\hat X_{\alpha,i}=X_\alpha z^i$ for $X_\alpha\in \mathfrak{g}_\alpha, \alpha\in \triangle$, we calculate $$\mathrm{ Ad}_{\hat w}(\hat X_{\alpha,i})=\mathrm{ Ad}_{w}( X_{\alpha})z^{i+\alpha(\Xi)},$$
then  the conditions that $\hat X_{\alpha,i}\in \widehat{\mathfrak{l}}_\Theta$  and $\mathrm{ Ad}_{\hat w}\hat X_{\alpha,i}\in \widehat{\mathfrak{p}}_\Theta$  hold iff $\alpha(\Theta)\in \mathbb{Z}$ and $\alpha(w^{-1}\Theta-\Theta+\Xi)\geq0$. It follows that there is a decomposition
$$ \widehat{\mathfrak q}_{\hat w}=\mathfrak{t}\oplus\bigoplus_{\alpha\in \triangle^+(\Theta,w)}z^{-\alpha(\Theta)}\mathfrak{g}_\alpha,$$
where $$\triangle^+(\Theta,w)=\{\alpha\in\triangle: \alpha(\Theta)\in \mathbb{Z}, \alpha(w^{-1}\Theta-\Theta+\Xi)\geq0\}.$$ Similarly, we have
$$ \widehat{\mathfrak q}'_{\hat w}=\mathfrak{t}\oplus\bigoplus_{\alpha\in \triangle^-(\Theta,w)}z^{-\alpha(\Theta)}\mathfrak{g}_\alpha,$$
where $$\triangle^-(\Theta,w)=\{\alpha\in\triangle: \alpha(\Theta)\in \mathbb{Z}, \alpha(w\Theta-\Theta-w\Xi)\geq0\}.$$
These two decompositions confirm our claim.
Moreover, the Levi parts of $\widehat{\mathfrak q}_{\hat w}$ and $ {\widehat {\mathfrak q}}'_{\hat w}$ are given, respectively, by
\begin{align*}
  \widehat{\mathfrak m}_{\hat w}&=\mathfrak{t}\oplus\bigoplus_{\alpha\in \triangle_0^+(\Theta,w)}z^{-\alpha(\Theta)}\mathfrak{g}_\alpha,\\
  \widehat{\mathfrak m}'_{\hat w}&=\mathfrak{t}\oplus\bigoplus_{\alpha\in \triangle_0^-(\Theta,w)}z^{-\alpha(\Theta)}\mathfrak{g}_\alpha,
\end{align*}
where \begin{align*}
       \triangle_0^+(\Theta,w)&=\{\alpha\in\triangle: \alpha(\Theta)\in \mathbb{Z}, \alpha(w^{-1}\Theta-\Theta+\Xi)=0\},\\
       \triangle_0^-(\Theta,w)&=\{\alpha\in\triangle: \alpha(\Theta)\in \mathbb{Z}, \alpha(w\Theta-\Theta-w\Xi)=0\}
      \end{align*}
Let $\pi_{\hat w}: \widehat{\mathfrak q}_{\hat w}\rightarrow  \widehat{\mathfrak m}_{\hat w} $ and $\pi'_{\hat w}: \widehat{\mathfrak q}'_{\hat w}\rightarrow  \widehat{\mathfrak m}'_{\hat w} $ be the natural Levi projections. One easily checks that the following diagram
$$\CD
   \widehat{\mathfrak h}_{\hat w}\ @> \pi_{\hat w}\circ\mathrm{Res}>> \widehat{\mathfrak m}_{\hat w} \\
  @V \mathrm{Ad}_{\hat w} VV @V  \mathrm{Ad}_{\hat w} VV  \\
  \widehat{\mathfrak h}'_{\hat w}\ @> \pi'_{\hat w}\circ\mathrm{Res}>> \widehat{\mathfrak m}'_{\hat w}
\endCD
$$
commutes.

An element of a parabolic subalgebra is nilpotent if and only if its Levi projection is nilpotent.
Now by assumption that $\mathrm{Res}(\hat C)$ is nilpotent, we have  $\pi_{\hat w}(\mathrm{Res}(\hat C))$ is nilpotent, hence $$\pi'_{\hat w}(\mathrm{Res}(\hat D))=\pi'_{\hat w}(\mathrm{Res}(\mathrm{Ad}_{\hat w}(\hat C)))=\mathrm{Ad}_{\hat w}(\pi_{\hat w}(\mathrm{Res}(\hat C)))$$
is nilpotent, which implies that $\mathrm{Res}(\hat D)$ itself is nilpotent.
\end{proof}

 \begin{theorem}\label{xz}
 Let  $ \mathsf{A}$ be a formal connection on a   formal parahoric principal $G$-bundle $(\mathbf{P}, \Theta, e)$ with  $\Theta$-reduced representation
\begin{align*}
 \mathsf{A}(e)=\sum\limits_{r\geq -c}\hat A^{(r)}z^{r}dz,
\end{align*}  where $c>1$. Assume that
$
   \mathrm{ Res}(\hat A^{(-c)})
$
is  a nilpotent element of $\widehat{\mathfrak l}_\Theta$,
then there is   $\hat g \in M_\Theta(\mathsf{A},e;-c)$ such that
$\chi(\hat g)$ is a regular nilpotent element in $\widehat{\mathfrak{l}}_\Theta$.
\end{theorem}

\begin{proof}
 We choose a Borel subgroup
$
    \widehat B_\Theta\subset\widehat{\mathcal{L}}_\Theta
$ with Lie algebra $\widehat{\mathfrak{b}}_\Theta$, then after a
$\widehat{\mathcal{P}}_\Theta$-gauge transformation, we may assume
$\mathrm{ Res}(\hat A^{(-c)})$ lies in the Lie algebra $\mathfrak{r}_u(\widehat B_\Theta)$ of unipotent radical $R_u(\widehat B_\Theta)$.
Let
$
    \widehat I_\Theta
    =
    l_\Theta^{-1}(\widehat B_\Theta)=\widehat{\mathcal{P}}_{\Theta_I}
    \subset
    \widehat{\mathcal{P}}_\Theta
$
be the corresponding Iwahori subgroup. $I_\Theta$ determines the affine simple reflections $\{s_i\}_{i\in \mathbb{I}}$ in $\widehat{W}$. There is a   subset $\mathbb{J}_\Theta\subset \mathbb{I}$ such that \begin{itemize}
  \item the subgroup  $\widehat{W}_{\mathbb{J}_\Theta}\subset \widehat{W}$ generated by $\{s_i:i\in \mathbb{J}_\Theta\}$ is finite,
  \item $\widehat{\mathcal{P}}_\Theta=\coprod\limits_{\hat w\in\widehat{W}_{\mathbb{J}_\Theta}}\widehat{I}_\Theta \hat w\widehat{I}_\Theta$.
\end{itemize}
Since $\widehat{\mathcal{L}}_\Theta\simeq\widehat{\mathcal{P}}_\Theta/\widehat{U}_\Theta\simeq\coprod\limits_{w\in W(\widehat{\mathcal{L}}_\Theta)}\widehat{B}_\Theta  w\widehat{B}_\Theta$, where
$ W(\widehat{\mathcal{L}}_\Theta)$ is the Weyl group of $\widehat{\mathcal{L}}_\Theta$, we have $\widehat{W}_{\mathbb{J}_\Theta}\simeq W(\widehat{\mathcal{L}}_\Theta)$.
For every $\hat g \in \mathcal{N}_\Theta:= N^0_{\Theta}(A,e;-c)$,  by Proposition \ref{513}, $\chi(\hat g) $
 is  a nilpotent element of $\widehat{\mathfrak l}_\Theta$.
Since   the fiber $\pi^{-1}_\Theta(\hat g)$
is  identified with $\{\widehat{\mathfrak{b}}\in \mathcal{B}_\Theta: \chi(\hat g)\in\widehat{\mathfrak{b}}\}$ due to proposition \ref{316},
we have that   $\chi(\hat g)$ is regular  nilpotent if and only if $\pi^{-1}_\Theta(\hat g)$ contains only one point.
Hence, we only need to show the morphism $\pi_\Theta: \mathcal{N}_{\Theta_I}\rightarrow  \mathcal{N}_{\Theta}$ must contain a one-point preimage.

Let
$
    \mathcal{F}_{\widehat I_\Theta}=G(K)/\widehat I_\Theta
$
be the affine flag variety associated with $\widehat I_\Theta$. The inclusion $\iota:\mathcal{N}_{\Theta_I}\rightarrow \mathcal{F}_{\widehat I_\Theta}$ induces a morphism $\iota_*: H_{2d}^{\mathrm{BM}}(\mathcal{N}_{\Theta_I}, \mathbb{C})\rightarrow  H_{2d}^{\mathrm{BM}}(\mathcal{F}_{\widehat I_\Theta}, \mathbb{C})$ between  Borel-Moore homology groups, where $d$ is the complex dimension of $\mathcal{N}_{\Theta_I}$ (cf. Proposition \ref{cv}).  The image of $\iota_*$ is denoted by $\mathcal{V}$.
Let
$Z$ be  a $d$-dimensional irreducible component of $\mathcal{N}_{\Theta_I}$, then
$  \iota_* [Z]\in H_{2d}^{\mathrm{BM}}( \mathcal{F}_{\widehat I_\Theta})
$
is nonzero \cite[Lemma 6]{kl}.  In particular, $\mathcal{V}\neq0$, and $\mathcal{V}$ is generated by the
classes of the $d$-dimensional irreducible components of $\mathcal{N}_{\Theta_I}$.
On the other hand, it is known that $\widehat{W}_{\mathbb{J}_\Theta}$ acts on  $H_{2d}^{\mathrm{BM}}(\mathcal{N}_{\Theta_I}, \mathbb{C})$ and  $\widehat{W}$ acts on  $H_{2d}^{\mathrm{BM}}(\mathcal{F}_{\widehat I_\Theta}, \mathbb{C})$, moreover, the following diagram
$$\CD
  H_{2d}^{\mathrm{BM}}(\mathcal{N}_{\Theta_I}, \mathbb{C})\ @> \iota_*>> H_{2d}^{\mathrm{BM}}(\mathcal{F}_{\widehat I_\Theta}, \mathbb{C}) \\
  @V \widehat{W}_{\mathbb{J}_\Theta}VV @V  \widehat{W} VV  \\
  H_{2d}^{\mathrm{BM}}(\mathcal{N}_{\Theta_I}, \mathbb{C})\ @> \iota_*>> H_{2d}^{\mathrm{BM}}(\mathcal{F}_{\widehat I_\Theta}, \mathbb{C})
\endCD$$
 commutes.
Hence,  $\mathcal{V}$ is  endowed with   an  action of $\widehat{W}_{\mathbb{J}_\Theta}$. We define
 $$\mathcal{V}_0= \{ v\in \mathcal{V}: w\cdot v=v \textrm{ for every } w\in \widehat{W}_{\mathbb{J}_\Theta}\}.$$
We claim that $\mathcal{V}_0\neq\{0\}$.
Consider the infinitesimal transformation at  $\hat g  \in \mathcal{F}_{\widehat I_\Theta}$ given by
\begin{align*}
\exp(-\varepsilon z^c \hat A)\hat g(z+\varepsilon z^c)&=(1-\varepsilon z^c \hat A)(\hat g+\varepsilon z^c\frac{d \hat g}{dz})\\
&=\hat g+\varepsilon (-z^c \hat A\hat g+z^c\frac{d \hat g}{dz})=\hat g+\varepsilon \delta\hat g,
\end{align*}
where one writes  $  \mathsf{A}(e)=\hat A dz$, and $\varepsilon$ is the infinitesimal parameter satisfying $\varepsilon^2=0$. Such infinitesimal transformation generates a vector field $\mathbb{V}$ on $\mathcal{F}_{\widehat I_\Theta}$. Since
$$\hat g^{-1}\delta\hat g=-z^c(\hat g^{-1}\hat A\hat g-\hat g^{-1}\frac{d \hat g}{dz})=-z^c\frac{\mathsf{A}(\hat g^{-1}e)}{dz},$$ the zero locus of $\mathbb{V}$ is determined by $z^c\mathsf{A}(\hat g^{-1}e)\in \widehat{\mathfrak{i}}_\Theta dz$,  namely it is exactly $\mathcal{N}_{\Theta_I}$.
Moreover, since $\mathrm{Res}(z^c\hat A)\in \mathfrak{r}_u(\widehat B_\Theta)$ implies that  $z^c\hat A$ lie in the Lie algebra $\widehat{u}_{\Theta_I}$ of $U_{\Theta_I}=l_\Theta^{-1}( R(\widehat B_\Theta))$ and the infinitesimal coordinate transformation
$z\mapsto z+\varepsilon z^c$ for $c>1$
lies in the Lie algebra of the pro-unipotent group  of formal coordinate changes tangent to the identity,  the vector field $\mathbb{V}$ gives rise to a $\mathbb{G}_a$-action on $\mathcal{F}_{\widehat I_\Theta}$, thus, $\mathcal{N}_{\Theta_I}$ is the locus in $\mathcal{F}_{\widehat I_\Theta}$ of the fixed points under the $\mathbb{G}_a$-action. Then by Frenkel-Zhu's result \cite[Proposition 7]{fz}) and Kazhdan-Lusztig's result  \cite[Lemma 8 in Section 4]{kl}, $\mathcal{V}_0$ contains nonzero vectors.

Next we complete the proof by contradiction.
Let
$
   x=\hat g\widehat I_\Theta
      \in
   \mathcal N_{\Theta_I}$
   and
  $ y=\hat g\widehat{\mathcal{ P}}_\Theta
      =
   \pi_\Theta(x)
$. Since $z^c\mathsf{A}(\hat g^{-1}e)\in \widehat{\mathfrak{i}}_\Theta$, we have $\chi(y)=\mathrm{Res}(\frac{z^c\mathsf{A}(\hat g^{-1}e)}{dz})\in \widehat{\mathfrak{b}}_\Theta$. By Proposition  \ref{513}, $\chi(y)$ is nilpotent, hence  $\chi(y)\in \mathfrak{r}_u(\widehat B_\Theta)$.
The generators $s_i, i\in \mathbb{J}_\Theta$, correspond to the simple reflections in $W(\widehat{\mathcal{L}}_\Theta)$
 determined by $\widehat{B}_\Theta$, and hence correspond to positive simple roots $\alpha_i$ of
$\widehat{\mathfrak{l}}_\Theta$.
A nilpotent element of
$\mathfrak r_u(\widehat{\mathfrak b}_\Theta)$ is regular if and only
if all its simple-root components are non-zero.  Now assume  $\chi(y)$ is
not regular, then there exists
$i\in\mathbb J_\Theta$ such that its $\alpha_i$-component vanishes. Then there is a decomposition
 $$\frac{z^c\mathsf{A}(\hat g^{-1}e)}{dz}=\sum_{\beta\in \triangle^+(\widehat{B}_\Theta), \beta\neq \alpha_i}r_\beta,$$
 where $r_\beta\in \widehat{\mathfrak{l}}_{\Theta,\beta}$.
Let
$ \widehat P_{(i)}^L
      \subset
   \widehat{\mathcal{L}}_\Theta
$
be the corresponding minimal parabolic subgroup determined by  $\widehat B_\Theta$ and $\alpha_i$, then its Lie algebra is
$\widehat{\mathfrak{p}}_{(i)}^L=\widehat{\mathfrak{b}}_\Theta\bigoplus\widehat{\mathfrak{l}}_{\Theta,-\alpha_i}$  and the  unipotent radical of $\widehat{\mathfrak{p}}_{(i)}^L$ is
$\widehat{\mathfrak{p}}_{{(i)},u}^L=\bigoplus\limits_{\beta\in \triangle^+(\widehat{B}_\Theta), \beta\neq \alpha_i} \widehat{\mathfrak{l}}_{\Theta,\beta}$.
We put
$
   \widehat P_{(i)}
      :=
   l_\Theta^{-1}(\widehat P_{(i)}^L)
$ and its pro-unipotent radical is denoted by $\widehat P_{{(i)},u}$.
Since
$\mathrm{Res}(\frac{z^c\mathsf{A}(\hat g^{-1}e)}{dz})
      \in
   \widehat{\mathfrak{p}}_{{(i)},u}^L
$,
  we have that $\frac{z^c\mathsf{A}(\hat g^{-1}e)}{dz}$ lies in the Lie algebra  $\widehat{\mathfrak{p}}_{{(i)},u}$ of $\widehat P_{{(i)},u}$.
Let $\hat h\in\widehat P_{(i)}$, then $\frac{z^c\mathsf{A}((\hat g\hat h)^{-1}e)}{dz}\in  \widehat{\mathfrak{p}}_{{(i)},u}$ since $\mathrm{Ad}_{\hat h^{-1}}
      (z^c\mathsf{A}(\hat g^{-1}e))\in \widehat{\mathfrak{p}}_{{(i)},u}dz$ due to $\widehat P_{(i),u}$ being  normal in $\widehat P_{(i)}$,  and $z^c d(\hat h^{-1})\hat h\in \widehat{\mathfrak{p}}_{{(i)},u}dz$ due to $c>1$.
     It follows that
$
   \hat g\hat h\widehat I_\Theta
      \in
   \mathcal N_{\Theta_I}
$. Let $
   \pi_{(i)}:
   \mathcal F_{\widehat I_\Theta}
      \rightarrow
   \mathcal F_{\widehat P_{(i)}}=\widehat{G}/\widehat{P}_{(i)}
$ be the natural projection. Since  $\widehat{P}_{(i)}/\widehat{I}_\Theta\simeq \mathbb{P}^1$, $\pi_{(i)}$ is regarded as a $\mathbb{P}^1$-fibration.
We have shown that
$$
   \pi_{(i)}^{-1}(\pi_{(i)}(x))
      \subset
   \mathcal N_{\Theta_I}.
$$

Let
$
   Z\subset\mathcal N_{\Theta_I}
$
be an irreducible component of dimension $d$.  For each
$i\in\mathbb J_\Theta$, let $Z_i\subset Z$ be the closed subset on
which the $\alpha_i$-component of $\chi\circ\pi_\Theta$ vanishes.
By assumption, we have
$$
   Z
      =
   \bigcup_{i\in\mathbb J_\Theta} Z_i.
$$
As $Z$ is irreducible and $\mathbb J_\Theta$ is finite, there exists
$i\in\mathbb J_\Theta$ such that
$
   Z=Z_i
$. Define $Y= \bigcup\limits_{x\in Z}\pi_{(i)}^{-1}(\pi_{(i)}(x))$, then $Y$ is a closed irreducible subset of $\mathcal N_{\Theta_I}$. Obviously, $Z\subset Y$, hence $Z=Y$.

So far, we can apply Kazhdan-Lusztig's argument in \cite{kl} (cf. the proof of \cite[THeorem 4]{fz}), which gives
gives
$$
   (1+ s_i)\cdot [Z]=0
$$for every $d$-dimensional irreducible component
$Z\subset\mathcal N_{\Theta_I}$.
Set
$
  S_\Theta
      :=
   \sum\limits_{w\in W_{\mathbb J_\Theta}} w
$.
For every $i\in\mathbb J_\Theta$ one may write
$ S_\Theta
      =
   O_i(1+s_i)
$
for a suitable operator $O_i$.  Therefore we get
$$
    S_\Theta\cdot [Z]=0.
$$
  Since the classes of these $d$-dimensional irreducible component
components generate $\mathcal V$, we obtain
$$
    S_\Theta\cdot \mathcal V=0.
$$
On the other hand,   choose
$
   0\neq v\in\mathcal V_0
$,
thereby
$
  w\cdot v=v
  $
 for every
   $w\in W_{\mathbb J_\Theta}$,
which means that
$$
    S_\Theta \cdot v
      =
   \sum_{w\in W_{\mathbb J_\Theta}}w\cdot v
      =
   |W_{\mathbb J_\Theta}|\,v
      \neq0.
$$
This  is a contradiction.
\end{proof}

\begin{corollary}\label{xzz}
Let  $ \mathsf{A}$ be a formal connection on a   formal parahoric principal $G$-bundle $(\mathbf{P}, \Theta, e)$  with the $\Theta$-reduced representation
\begin{align*}
 \mathsf{A}(e)=\sum\limits_{r\geq -c}\hat A^{(r)}z^{r}dz,
\end{align*}  where $c>1$. Put
 $$c'=\left\{
             \begin{array}{ll}
              c , & \hbox{$\mathrm{Res}(\hat A^{(-c)})$\textrm{ is a nilpotent element in  } $\widehat{\mathfrak{l}}_\Theta$;} \\
              c+1, & \hbox{otherwise.}
             \end{array}
           \right.
  $$
 Then there are a  Jacobson--Morozov $\mathfrak{sl}(2,\mathbb C)$-triple $(\hat P, \hat Q, H)$ in $\widehat{\mathfrak{l}}_\Theta$ with regular nilpotent $\hat P, \hat Q$, and
 $\hat g \in \widehat{G}$ such that under the trivialization $\hat ge$, $\mathsf{A}$ has a  representation 
\begin{align*}
  \mathsf{A}(\hat ge)=z^{-c'}(\hat P+\sum\limits_{n_\nu>0}\hat B[n_\nu])dz
\end{align*}satisfying $B[n_\nu]\in \mathrm{Ker(\mathrm{ad}_{\hat Q}|_{\widehat{\mathfrak{p}}_\Theta[n_\nu]})} $.
\end{corollary}

\begin{proof}This is deduced from Theorem \ref{xx} and Theorem \ref{xz}.
\end{proof}
\begin{remark}Taking  $\Theta=0$, the above corollary recovers Frenkel-Zhu's result \cite[Theorem 1]{fz}. Indeed, for that case,  we have
$
\widehat{\mathfrak{l}}_0=\mathfrak{g}$, $\widehat{\mathfrak p}_0[n]=z^n\mathfrak g
$, and
$
\hat B[n]\in
z^nC_\mathfrak{g}(\hat Q)
$.
Let  $\mathfrak b\subset\mathfrak g$ be  the unique Borel subalgebra
containing $\hat Q$, then
$
C_\mathfrak{g}(\hat Q)\subset\mathfrak b
$.
Consequently, after writing $\hat P=f$ and $\hat Q=e$ for a principal
Jacobson--Morozov $\mathfrak{sl}(2,\mathbb C)$-triple $(f,e,H)$ in $\mathfrak{g}$, the above representation becomes
$$
\mathsf A(\hat g e)
=
z^{-c'}
(
f+\sum_{n>0}b_n z^n
)dz,$$
with
$b_n\in \mathfrak b$.
This is precisely the  form in
Frenkel-Zhu's Borel reduction theorem.
\end{remark}

Finally, we present a few  examples which explicitly  illustrates our theory.
\begin{example}
For $G=GL(3, \mathbb{C})$,  choose the weight $\Theta=\mathrm{diag}(\frac{1}{2},0,-\frac{1}{2})$, then we have
\begin{align*}
  \widehat{\mathfrak{l}}_\Theta&=\{\left(
                                  \begin{array}{ccc}
                                   a & 0& bz^{-1} \\
                                   0 & c & 0 \\
                                   dz& 0& e
                                  \end{array}
                                \right):a, b, c, d, e\in \mathbb{C}\}, \\
                                 \widehat{\mathfrak{p}}_\Theta&=
                                 \{\left(
                                  \begin{array}{ccc}
                                   \hat a & \hat b& \hat cz^{-1} \\
                                   z\hat d & \hat e & \hat f \\
                                   z \hat g& z\hat h& \hat i
                                  \end{array}
                                \right):\hat a, \hat b, \hat c, \hat d, \hat e,\hat f, \hat g, \hat h, \hat i\in \mathbb{A}\}.
\end{align*}

Putting $S=\mathrm{diag}(1,1,-2)$ and $E_{ij}$ the  standard matrix units, we
consider the formal connection
\begin{align*}
  \mathsf A( e)&=z^{-2}(S+E_{12}+3E_{23}+E_{13}+zE_{21})dz\\
  &=(\frac{\hat A^{(-2)}}{z^2}+\frac{\hat A^{(-1)}}{z})dz,
\end{align*}
where  the second line is the $(\Theta,e)$-reduced representation given by $\hat A^{(-2)}=S+E_{12}+3E_{23}+zE_{21}$ and $\hat A^{(-1)}=z^{-1}E_{13}$. Obviously, $\mathrm{Res}(\hat A^{(-2)})=S$ is semisimple in $\widehat{\mathfrak{l}}_\Theta$ but not regular in $\widehat{\mathfrak{p}}_\Theta$. Applying gauge transformation  by $U=\exp(E_{23})=Id+E_{23}\in \widehat{U}_\Theta$, we get
\begin{align*}
 \mathsf A(u e)=z^{-2}(S+E_{12}+zE_{21})dz
\end{align*}
Since $[S,E_{12}]=[S,E_{21}]=0$, the above $(\Theta, ue)$-reduced representation offers the reduction  in Theorem \ref{x}.

Moreover, to obtain regular nilpotent reduction, we consider the following element in $\widehat{W}$
$$\hat w=\left(
           \begin{array}{ccc}
             1 & 0 & 0 \\
             0 & 0 & 1 \\
             0 & 1 & 0 \\
           \end{array}
         \right)z^{\mathrm{diag}(0,-1,0)}=\left(
           \begin{array}{ccc}
             1 & 0 & 0 \\
             0 & 0 & 1 \\
             0 & z^{-1} & 0 \\
           \end{array}
         \right),
$$ then we obtain  a $(\Theta,\hat w u e )$-reduced representation
\begin{align*}
 \mathsf A(\hat w u e)=\sum_{r=-3}^0\hat B^{(r)}z^rdz=
 (\frac{zE_{31}}{z^3}+\frac{S'}{z^2}-\frac{E_{33}}{z}+z^{-1}E_{13})dz,
\end{align*}
where $S'=\mathrm{diag}(1,-2,1)$. Now $\mathrm{Res}(B^{(-3)})=zE_{31}$ is regular nilpotent in $\widehat{\mathfrak{l}}_\Theta$. The Jacobson--Morozov $\mathfrak{sl}(2,\mathbb C)$-triple $(\hat P, \hat Q, H)$ is given by $\hat P= zE_{31}, \hat Q=z^{-1}E_{13}, H=E_{33}-E_{11}$. The gauge transformation by $\hat v=\exp(-\frac{zE_{13}}{2})\in \widehat{U}_\Theta$ yields
\begin{align*}
 \mathsf A(\hat v\hat w u e)=z^{-3}(\hat P+\sum_{j=1}^4\hat C[j])dz,
\end{align*}
where
 \begin{align*}
  \hat C[1]&=zS',\\
  \hat  C[2]&=-\frac{z^2}{2}(E_{11}+E_{33}),\\
   \hat  C[3]&=z^2E_{13}, \\
   \hat  C[4]&=-\frac{z^3}{4}E_{13}.
 \end{align*}
 Since $[\hat Q, \hat C[j]]=0, j=1,\cdots,4$, we get the desired normal form in Corollary \ref{xzz}.
  \end{example}

 \begin{example}Let $G$ be a simple group  of adjoint type with Coxeter number $h$, and let $T$ be a maximal torus of $G$ and $B\subset G$ be  Borel subgroup  containing $T$. Choosing an integer $m\in [1,h-1]$ with $\mathrm{gcd}(m,h)=1$, a formal connection $\mathsf{A}_{c,m}$ on a formal  principal $G$-bundle $\mathbf{P}$ is called a homogeneous Coxeter connection if
 \begin{align}\label{sw2}
  \mathsf{A}_{c,m}(e)=(\frac{E_m}{z^{c+2}}+\frac{F_m}{c^{c+1}})dz
 \end{align}
under a suitable trivialization $e\in T_\mathbf{P}$, where $c\geq 0$, $F_m$ (respectively, $E_m$) is the sum of
  root vectors of height $-m$ (respectively, $h-m$) with  respect to $B$.
In particular,  $\mathsf{A}_{0,1}$ is called a homogeneous Frenkel-Gross connection and $\mathsf{A}_{1,1}$  it is called a homogeneous Airy  connection \cite{ks,ks1,j}.

For the trivial weight $\Theta=0$, the representation \eqref{sw2} is a $(0,e)$-reduced representation. The leading  nonzero nilpotent Levi
residue is  $E_m$. In general, $E_m$ is not regular. We explain how to make the leading coefficient  regular nilpotent as the crucial step in Corollary \ref{xzz}.
Since   $G$ is of adjoint type,  $\rho^\vee$ is a cocharacter.  Let 
$$ \Theta'=\frac{\rho^\vee}{h}\in\mathfrak t_{\mathbb R},$$
then $\exp(2\pi \sqrt{-1}\Theta')$ and $\exp(2\pi \sqrt{-1}q\Theta')$
are both regular semisimple elements  of order $h$ lying in $T$. Therefore, there is an element $w_m$ in the the Weyl group $W_{T}$ with respect to $T$ such that
$$w_m(\exp(2\pi \sqrt{-1}q\Theta'))=\exp(2\pi \sqrt{-1}\Theta'),$$ where $q$ is an integer satisfying  $qm\equiv 1 (\mathrm{mod}\ h)$. Define $$\Phi_m=\{\alpha\in\triangle: \alpha(\rho^\vee)=-m\textrm{ or }h-m\}=\{\alpha\in\triangle: \alpha(\exp(2\pi \sqrt{-1}q\Theta')=\exp(-\frac{2\pi\sqrt{-1}}{h})\}.$$ In particular,
\begin{align*}
    \Phi_1&=\{\alpha\in\triangle: \alpha(\rho^\vee)=-1\textrm{ or }h-1\}=\{\alpha\in\triangle: \alpha(\exp(2\pi \sqrt{-1}\Theta')=\exp(-\frac{2\pi\sqrt{-1}}{h})\}\\
    &=\{-\alpha_1,\cdots,-\alpha_l,\eta\},
\end{align*}
where $\alpha_1,\cdots,\alpha_l$ are all  simple roots and $\eta$ is the root with height $h-1$. $\Phi_1$ is also called the cyclic root set, which plays a role in the theory of cyclic Higgs bundles
 \cite{lab}. Consequently, we write  $\Phi_m=\{\beta_0,\beta_1,\cdots \beta_l\}$ with $w_m(\beta_0)=\eta, w_m(\beta_i)=-\alpha_i$. Define the cocharacter
$$
 \Lambda_m=\sum_{i=1}^\ell\varepsilon(\beta_i)\omega_i^\vee
 \in (E_T)_*,
$$ where $\varepsilon(\beta_i)=\left\{
                                 \begin{array}{ll}
                                   1, & \hbox{$\beta_i(\rho^\vee)=-m$} \\
                                   0, & \hbox{$\beta_i(\rho^\vee)=h-m$}
                                 \end{array}
                               \right.
$, and $\omega_i^\vee$ is the  dual basis of  $\alpha_i$. After a gauge transformation by $\hat g=z^{\Lambda_m}w_m$, we have
$$\mathsf{A}_{c,m}(\hat g e)=z^{-c-2}(P+z^{h-m}R_{\eta}
                    +z^{c+1}\Lambda_m)dz,$$
where $P=\sum\limits_{i=1}^\ell  P_{-\alpha_i}$ for certain root vectors $P_{-\alpha_i}\in \mathfrak{g}_{-\alpha_i}$ and  $R_{\eta}$ is a root vector in $\mathfrak{g}_{\eta}$. Now the leading coefficient $P$ is regular nilpotent.

A more suitable choice of the weight is 
as that in Example \ref{exa}, namely we choose the  weight
$
 \Theta'$, called the principal weight, 
so that $\widehat{\mathcal{P}}_{\Theta'}$ is an Iwahori subgroup. By definition, we have
\begin{align*}
[\Theta', E_m]&=\frac{h-m}{h}E_m, \\ [\Theta', F_m]&=-\frac{m}{h}F_m,
\end{align*}
then  $\mathsf{A}_{c,m}$ has a $(\Theta,e)$-reduced representation
$$\mathsf{A}_{c,m}(e)=\frac{E_m+zF_m}{z^{c+2}}dz,
       $$
where  $E_m+zF_m\in \widehat{\mathfrak{p}}_{\Theta'}[1-\frac{m}{h}]$. Since $E_m+zF_m$ is a  semisimple element in $\widehat{\mathfrak{g}}$,  the fundamental $\widehat{G}$-stratum associated to $\mathsf{A}_{c,m}$ is given by
$$(\Theta', s, \beta)=(\frac{\rho^\vee}{h}, c+\frac{m}{h}, -z^{-c-1}(E_m+zF_m)\frac{dz}{z} ). $$ Moreover, $E_m+zF_m$ is also regular in $\widehat{\mathfrak{g}}$, this
$\widehat{G}$-stratum is a regular $\widehat{G}$-stratum, thus $\mathsf{A}_{c,m}$ is a toral form connection in the sense of  Bremer-Sage. Note that the identity component of the  centralizer $C_{\widehat{G}}(E_m+zF_m)$ is a nonsplit Coxeter maximal torus in $\widehat{G}$ \cite{kzd}.

Passing to the finite cover
$z=\zeta^h$ and then  applying the gauge transformation by
$\hat g=\zeta^{\rho^\vee}$, we have
$$
 \mathsf{A}_{c,m}^{\sharp_h}(\hat g e^{\sharp_h})
 =(\frac{h(E_m+F_m)}{\zeta^{hc+m+1}}
          +\frac{\rho^\vee}{\zeta})d\zeta.
$$
Since  $E_m+F_m$ is a regular
semisimple  element in $\mathfrak{g}$, the slope of $\mathsf{A}_{c,m}$ reads $\frac{hc+m}{h}$ which coincides with the depth of the fundamental $\widehat{G}$-stratum. By Kamgarpour-Sage's result \cite[Theorem 35]{kzd},  the term $\rho^\vee\frac{d\zeta}{\zeta}$ can be eliminated. One can show there is $Y\in\mathfrak{g}$ such that $[Y, h(E_m+F_m)]=-\rho^\vee$, then consider the gauge transformation by $\hat u=\exp(\zeta^{hc+m}Y)$, we have
$$
 \mathsf{A}_{c,m}^{\sharp_h}(\hat u\hat g e^{\sharp_h})
 =(\frac{h(E_m+F_m)}{\zeta^{hc+m+1}}
          +O(\zeta^{hc+m-1}))d\zeta.
$$
The  non-singular terms  can be further  removed by successive  gauge transformations, namely  there exists $\hat v\in G(\mathbb{C}[[\zeta]])$ such that
\begin{align}\label{jjjj}
\mathsf{A}_{c,m}^{\sharp_h}(\hat v\hat u\hat g e^{\sharp_h})
 =\frac{h(E_m+F_m)}{\zeta^{hc+m+1}}
          d\zeta.
\end{align}
The centralizer $C_\mathfrak{g}(E_m+F_m)$ is a Cartan subalgebra. It is denoted by $\mathfrak{t}'$ and the corresponding maximal torus in $G$ is denoted by $T'$. One easily checks that
$$\mathrm{Ad}_{\exp(2\pi \sqrt{-1}\Theta')}(E_m+F_m)=\exp(-2\pi \sqrt{-1}\frac{m}{h})(E_m+F_m),$$
therefore   $\exp(2\pi \sqrt{-1}\Theta')$ lies in the normalizer $N_G(T')$ and corresponds to an element $w_{\Theta'}$ in the Weyl group $W_{T'}$ with respect to $T'$. In particular, $w_{\Theta'}$ has the order $h$, i.e. it is a Coxeter element.
As a consequent, the representation \eqref{jjjj} is  Kamgarpour-Sage's Jordan normal form.

\end{example}

\vspace{4mm}

\textbf{Conflict of Interest}.
On behalf of all authors, the corresponding author states that there is no conflict of interest.

\vspace{2mm}

\textbf{Data Availability Statement}.
Since this is an article in the area of algebraic geometry regarding the reduction theory of formal connections, data availability is not applicable to this article as no new data were created or analyzed in this study.


\begin{thebibliography}{9}



\bibitem{ac} P.  Achar,
L. Rider: Parity sheaves on the affine Grassmannian
and the Mirkovi\'{c}-Vilonen conjecture, Acta Math. \textbf{215 }(2015), 183-216

\bibitem{ach}P. Achinger: Regular logarithmic connections, Math. Ann. \textbf{391} (2025),  5293-5339

\bibitem{a}S. Arkhipov, R. Bezrukavnikov, V. Ginzburg: Quantum groups, the loop Grassmannian,
and the Springer resolution, J. Amer. Math. Soc. \textbf{17} (2004), 595-678

\bibitem{ar}D. Arinkin:  Irreducible connections admit generic oper structures, arXiv:1602.08989

\bibitem{ab}A. Aubert, P. Baum, R. Plymen, M. Solleveld: Depth and the local Langlands correspondence, In:  Arbeitstagung Bonn, Progr. Math. \textbf{319} (2013), 17-41

\bibitem{az}A.  Azhang,
Rigid Connections on the Projective Line with Elliptic Toral Singularities, Ph. D. Thesis (2021), Louisiana State University

\bibitem{bv}D. Babbitt, V. Varadarajan: Formal reduction theory of
meromorphic differential equations: a group theoretic view, Pacific J.
Math. \textbf{109} (1983), 1-80

\bibitem{BS} V. Balaji,
C. Seshadri: Moduli of parahoric $G$-torsors on a compact Riemann surface,
J.  Alg.  Geom. \textbf{24} (2015), 1-49


\bibitem{bl}A. Beauville, Y. Laszlo: Conformal blocks and generalized theta functions, Comm. Math. Phys. \textbf{164 }(1994), 385-419

\bibitem{be} R. Bezrukavnikov: On two geometric realizations of an affine Hecke algebra,  Publ. Math. IHES \textbf{123} (2016), 1-67

\bibitem{bp}O. Biquard, O. Garc\'{i}a-Prada, I. Mundet i Riera: Parabolic Higgs bundles and representations of the fundamental group of a punctured surface into a real group, Adv. Math. \textbf{372} (2020),  107305

\bibitem{bis}I. Biswas:
Connections on a Parabolic Principal
Bundle Over a Curve,  Canad. J. Math. \textbf{58} (2006), 262-281
\bibitem{Bo} P. Boalch: Symplectic manifolds and isomonodromic deformations, Adv. Math. \textbf{163} (2001), 137-205
\bibitem{b}P. Boalch: Riemann-Hilbert for tame complex parahoric connections, Transform. Groups \textbf{16} (2011),  27-50

\bibitem{bq}P. Boalch: Wild character varieties, meromorphic Hitchin systems and Dynkin diagrams. In: Geometry and Physics
Vol. 2 (2018),  433-454,   Oxford University Press








\bibitem{bs1}C. Bremer, D. Sage: Moduli spaces of irregular singular connections, Int. Math. Res. Not.  \textbf{2013} (2013), 1800-1872

\bibitem{bs2}C. Bremer, D. Sage: Isomonodromic deformation of connections with singularities of
parahoric formal type,  Comm. Math. Phys.  \textbf{313} (2012), 175-208

\bibitem{bs3}C. Bremer, D. Sage: Flat $G$-bundles and regular strata for reductive groups,
arXiv:1309.6060

\bibitem{bs}C. Bremer, D. Sage: A theory of minimal $K$-types for flat $G$-bundles,  Int. Math. Res. Not.  \textbf{11} (2018), 3507-3555

\bibitem{bk}T.-H. Chen, M. Kamgarpour: Preservation of depth in local geometric Langlands correspondence,   Trans. Amer. Math. Soc. \textbf{369} (2017), 1345-1364

\bibitem{c}V. Chernousov, P. Gille, A. Pianzola: Torsors over the punctured affine line, Amer. J. Math.  \textbf{134} (2012), 1541-1583

\bibitem{ft}F. Coiai, I. Yogish:  Extension of structure groups of principal
bundles in positive characteristic, J. reine angew. Math. \textbf{595} (2006), 1-24


\bibitem{de}P. Deligne: \'{E}quations Diff\'{e}rentielles \`{a} Points Singuliers R\'{e}guliers, Lecture Notes in
Math. \textbf{ 163} (1970), Springer-Verlag

\bibitem{del} P. Deligne: Le groupe fondamental de la droite projective moins trois points,  Math. Sci. Res. Inst. Publ. \textbf{16} (1987), 79-297

\bibitem{ds}V. Drinfeld, C. Simpson:  $B$-structures on $G$-bundles and local triviality,  Math. Res. Lett. \textbf{2} (1995), 823-829

\bibitem{fa}G. Faltings: Algebraic loop groups and moduli spaces of bundles, J. Eur. Math. Soc. \textbf{5}  (2003), 41-68

\bibitem{ffff}J.  F{\ae}rgeman: Motivic realization of rigid $G$-local systems on curves and tamely ramified geometric Langlands, arXiv:2405.18268

\bibitem{f1}E.  Frenkel: Ramifications of the geometric Langlands program, In: Representation Theory and Complex Analysis (2004), 51-135

\bibitem{f} E. Frenkel: Langlands correspondence for loop groups (2007), Cambridge University Press

\bibitem{fg1}E. Frenkel, D. Gaitsgory: Local geometric Langlands correspondence and affine Kac-Moody algebras, In:  Algebraic Geometry and Number Theory, Progress in Mathematics \textbf{253} (2006), 69-260, Birkh\"{a}user

\bibitem{fg}E. Frenkel, D. Gaitsgory: D-modules on the affine flag variety and representations of affine Kac-Moody algebras, Repr.  Theory
\textbf{13} (2009), 470-608

\bibitem{fs}E. Frenkel, B.  Gross: A rigid irregular connection on the projective line,  Ann. Math. \textbf{170}  (2009), 1469-1512

\bibitem{fz}E. Frenkel, X. Zhu:  Any flat bundle on a punctured disc has an oper structure, Math. Res. Lett. \textbf{17} (2010),  27-37

\bibitem{g}D. Gaitsgory: Construction of central elements in the affine Hecke algebra via nearby cycles, Invent. Math. \textbf{144 } (2001), 253-280

\bibitem{hr}T. Haines, M. Rapoport:  Appendix: On parahoric subgroups, Adv. Math. \textbf{219 } (2008), 188-198

\bibitem{hb}J. Heinloth, B. Ng\^{o}, Z. Yun: Kloosterman sheaves for reductive groups, Ann. Math. \textbf{177} (2013), 241-310

\bibitem{h}A. Herrero: Reduction theory for connections over the formal punctured disc, arXiv:2003.00008

\bibitem{hh}Z. Hu, P. Huang: Parahoric generic opers, in preparation

\bibitem{HKSZ22} P. Huang, G. Kydonakis, H. Sun, L. Zhao: Tame parahoric nonabelian Hodge correspondence on curves, arXiv:2205.15475

\bibitem{HS23} P. Huang, H. Sun: Meromorphic parahoric Higgs torsors and filtered Stokes $G$-local systems, Adv. Math. \textbf{429} (2023), 109183

\bibitem{huk} M. Hukuhara: Th\'{e}or\`{e}mes fondamentaux de la th\'{e}orie des \'{e}quations diff\'{e}rentielles ordinaires. II, Mem. Fac. Sci. Ky\={u}sy\={u} Imp. Univ. A. \textbf{2} (1941), 1-25

\bibitem{hd}D.  Hyeon, D. Murphy:   Note on the stability of principal bundles,  Proc. Amer. Math. Soc.  \textbf{132} (2004), 2205-2213

\bibitem{j}K. Jakob, M. Kamgarpour, L. Yi:  Airy sheaves for reductive groups, Proc.  London Math. Soc.  \textbf{126} (2023), 390-428

\bibitem{ka}V. Kac: Constructing groups associated to infinite-dimensional algebras, In: Infinite
Dimensional Groups with Applications (1985), 167-216,  Springer-Verlag

\bibitem{ks}M. Kamgarpour, D.  Sage: A geometric analogue of a conjecture of Gross and Reeder, Amer. J. Math. \textbf{141}
(2019), 1457-1476
\bibitem{k}M.  Kamgarpour, S. Weatherhog: Jordan decomposition for formal $G$-connections, Graduate J. Math. \textbf{555} (2020), 111-121


\bibitem{ks1}M. Kamgarpour, D.  Sage:
Rigid connections on $\mathbb{P}^1$ via the Bruhat-Tits building, Proc. London  Math. Soc. \textbf{22} (2021), 359-376



\bibitem{kzd}M.  Kamgarpour, D.Sage: Differential Galois groups of $G$-connections with Coxeter singularities, arXiv:2309.11742

\bibitem{kz} N. Katz: Nilpotent connections and the monodromy theorem: Applications of a result of Turrittin, Publ. Math. IHES \textbf{39 }(1970), 175-232

\bibitem{kat}N. Katz: On the calculation of some differential Galois groups,  Invent. Math. \textbf{87} (1987), 13-61.

\bibitem{kl}D. Kazhdan, G. Lusztig:  Fixed point varieties on affine flag manifolds, Israel J. Math. \textbf{ 62} (1988),  129-168

\bibitem{ke}K. Kedlaya: Good formal structures for flat meromorphic connections, I:  surfaces, Duke Math. J. \textbf{154 }(2010),  343-418

 \bibitem{ke1}K. Kedlaya: Good formal structures for flat meromorphic connections, II: excellent schemes, J. Amer. Math. Soc. \textbf{24}  (2011), 183-229

\bibitem{klp}C. Klevdal, S. Patrikis: $G$-cohomologically rigid local systems are integral, Trans. Amer. Math. Soc. \textbf{375}  (2022), 4153-4175

\bibitem{ki} L. Kindler: Local-to-global extensions of $D$-modules in positive characteristic, Int. Math. Res. Not. \textbf{19} (2015), 9139-9174


\bibitem{KSZ24} G. Kydonakis, H. Sun, L. Zhao: Logahoric {H}iggs torsors for a complex reductive group, Math. Ann. \textbf{388} (2024), 3183-3228
\bibitem{lab}F.  Labourie: Cyclic surfaces and Hitchin components in rank 2,  Anna. Math. \textbf{185} (2017), 1-58
\bibitem{Laf} V. Lafforgue: Chtoucas pour les groupes r\'{e}ductifs et param\'{e}trisation de Langlands globale, J. Amer. Math. Soc.  \textbf{31} (2018), 719-891

\bibitem{la} E. Landvogt: A compactification of the Bruhat-Tits building, Lecture Notes in Mathematics \textbf{1619 }(1996), Springer-Verlag

\bibitem{l} G. Levelt: Jordan decomposition for a class of singular differential operators, Ark. Mat. \textbf{13} (1975), 1-27

\bibitem{ma} S. Matsuda: Katz correspondence for quasi-unipotent overconvergent isocrystals, Compo. Math. \textbf{134} (2002),  1-34

\bibitem{MV1}I. Mirkovi\'{c}, K. Vilonen: Perverse sheaves on affine Grassmannians and Langlands duality, Math. Res. Lett. \textbf{7} (2000), 13-24

\bibitem{MV2}I. Mirkovi\'{c}, K. Vilonen: Geometric Langlands duality and representations of algebraic groups over commutative rings, Ann. Math. \textbf{166} (2007), 95-143

 \bibitem{moc} T. Mochizuki:   Good formal structure for meromorphic flat connections on smooth projective surfaces, In: Algebraic Analysis and Around, Advanced Studies in Pure Math. \textbf{54}  (2009), 223-253

 \bibitem{moc1}T. Mochizuki: Wild harmonic bundles and wild pure twistor $D$-modules, Ast\'{e}risque \textbf{340} (2011)

 \bibitem{m1} A. Moy, G. Prasad: Unrefined minimal $K$-types for $p$-adic groups, Invent. Math. \textbf{116} (1994), 393-408

\bibitem{m2}A. Moy, G. Prasad: Jacquet functors and unrefined minimal $K$-types, Comm. Math. Helv. \textbf{71} (1996), 98-121

  \bibitem{p}  G. Pappas, M. Rapoport: Twisted loop groups and their affine flag varieties, Adv. Math. \textbf{219} (2008),
118-198

 \bibitem{ppp} G. Pappas, M. Rapoport: Some questions about $G$-bundles on curves, In: Algebraic and Arithmetic Structures of Moduli Spaces (2010), 159-171

 \bibitem{ppp1} G. Pappas, M. Rapoport:
 On tamely ramified $\mathcal{G}$-bundles on curves,  	arXiv:2209.13457

\bibitem{r} A. Ramanathan: Moduli for principal bundles over algebraic curves I, II, Proc. Indian Acad. \textbf{06} (1996), 301-328, 421-449
 \bibitem{re}B. R\'{e}my, A. Thuillier, A. Werner:  Bruhat-Tits building and analytic geometry,  In: Berkovich spaces
and applications,   Lecture Notes in Math. \textbf{2119} (2015), 141-202, Springer

\bibitem{s} O.  Schn\"{u}re: Regular connections on principal fiber bundles over the infinitesimal punctured disc, J.  Lie Theory
\textbf{17} (2007), 427-448

\bibitem{si} C.  Simpson: Higgs bundles and local systems, Publ. Math. IHES  \textbf{75} (1992), 5-95

\bibitem{si2}C.  Simpson: Moduli of representations of the fundamental group of a smooth projective variety II, Publ. Math. IHES \textbf{80} (1994), 5-79

\bibitem{Wo} W. Wasow: Asymptotic Expansions for Ordinary Differential Equations, Interscience (1965)
\bibitem{y}Z. Yun:  Motives with exceptional Galois groups and the inverse Galois problem, Invent. Math.\textbf{196} (2014),  267-337

\bibitem{y1}Z. Yun: Lectures on Springer theories and orbital integrals,  In: Geometry of moduli spaces and representation theory, IAS/Park City Mathematics Series \textbf{24} (2017)

\bibitem{z}X. Zhu: Affine Grassmannians and the geometric
Satake in mixed characteristic, Ann. Math. \textbf{185 }(2017), 403-492



\end{thebibliography}
 \end{document}